\documentclass[a4paper]{amsart}
\pdfoutput=1
\usepackage[utf8]{inputenc}
\usepackage[T1]{fontenc}
\usepackage{lmodern}
\usepackage{amssymb,amsxtra}
\usepackage{graphicx}
\usepackage{mathabx}
\usepackage{a4wide}
\usepackage{nicefrac,mathtools}
\usepackage[lite]{amsrefs}
\renewcommand*{\MR}[1]{ \href{http://www.ams.org/mathscinet-getitem?mr=#1}{MR \textbf{#1}}}
\newcommand*{\arxiv}[1]{\href{http://www.arxiv.org/abs/#1}{arXiv: #1}}
\newcommand*{\ZBmath}[1]{ \href{https://zbmath.org/#1}{zbMath \textbf{#1}}}

\renewcommand{\PrintDOI}[1]{\href{http://dx.doi.org/\detokenize{#1}}{doi: \detokenize{#1}}}

\BibSpec{book}{%
  +{}  {\PrintPrimary}                {transition}
  +{,} { \textit}                     {title}
  +{.} { }                            {part}
  +{:} { \textit}                     {subtitle}
  +{,} { \PrintEdition}               {edition}
  +{}  { \PrintEditorsB}              {editor}
  +{,} { \PrintTranslatorsC}          {translator}
  +{,} { \PrintContributions}         {contribution}
  +{,} { }                            {series}
  +{,} { \voltext}                    {volume}
  +{,} { }                            {publisher}
  +{,} { }                            {organization}
  +{,} { }                            {address}
  +{,} { \PrintDateB}                 {date}
  +{,} { }                            {status}
  +{}  { \parenthesize}               {language}
  +{}  { \PrintTranslation}           {translation}
  +{;} { \PrintReprint}               {reprint}
  +{.} { }                            {note}
  +{.} {}                             {transition}
  +{} { \PrintDOI}                   {doi}
  +{} { available at \url}            {eprint}
  +{}  {\SentenceSpace \PrintReviews} {review}
}

\usepackage{enumitem} 
\setlist[enumerate,1]{label=\textup{(\arabic*)}}
\usepackage[all]{xy}
\usepackage{tikz}
\usetikzlibrary{matrix,calc,arrows,decorations.pathmorphing}
\tikzset{node distance=2cm, auto}

\tikzset{cd/.style=matrix of math nodes,row sep=2em,column sep=2em, text height=1.5ex, text depth=0.5ex}
\tikzset{cdar/.style=->,auto}
\tikzset{mid/.style={anchor=mid}} 
\tikzset{narrowfill/.style={inner sep=1pt, fill=white}}
\tikzset{rndblock/.style={rounded corners,rectangle,draw,outer sep=0pt}}

\usepackage{tikz-cd}

\usepackage{microtype}
\usepackage[pdftitle={Irreducible regular C*-inclusions},   pdfauthor={Bartosz Kwa\'sniewski, Ralf Meyer},  pdfsubject={Mathematics}]{hyperref}

\theoremstyle{plain}
\newtheorem{theorem}{Theorem}
\newtheorem{thmx}{Theorem}

\newtheorem{lemma}[theorem]{Lemma}

\newtheorem{proposition}[theorem]{Proposition}
\newtheorem{corollary}[theorem]{Corollary}
\theoremstyle{definition}
\newtheorem{definition}[theorem]{Definition}
\theoremstyle{remark}
\newtheorem{remark}[theorem]{Remark}
\newtheorem{example}[theorem]{Example}

\numberwithin{theorem}{section}
\numberwithin{equation}{section}

\DeclareMathOperator{\Aut}{Aut}
\DeclareMathOperator{\Bis}{Bis}

\newcommand{\Her}{\mathbb H}

\DeclareMathOperator*{\clsp}{\overline{span}} 

\newcommand*{\nb}{\nobreakdash}
\newcommand*{\Star}{\(^*\)\nobreakdash-}

\newcommand*{\C}{\mathbb C}
\newcommand*{\Z}{\mathbb Z}

\newcommand*{\N}{\mathbb N}
\newcommand*{\T}{\mathbb T}

\newcommand*{\Null}{\mathcal N}
\newcommand*{\Bound}{\mathbb B}
\newcommand*{\Comp}{\mathbb K}
\newcommand*{\red}{\mathrm{r}}
\newcommand*{\ess}{\mathrm{e}}

\newcommand*{\Cst}{\textup C^*}
\newcommand*{\Wst}{\textup W^*}
\newcommand*{\Mult}{\mathcal M}
\newcommand*{\Locmult}{\mathcal{M}_{\mathrm{loc}}}
\newcommand*{\UMult}{\mathcal UM}

\newcommand*{\Cont}{\textup C}
\newcommand*{\Contc}{\Cont_\textup c} 

\newcommand*{\Slice}{\mathcal S}
\newcommand*{\Sub}{\textup{Sub}}
\newcommand*{\Id}{\textup{Id}}
\newcommand*{\Ad}{\textup{Ad}}

\newcommand*{\Hilm}[1][E]{\mathcal #1}

\newcommand*{\B}{\mathcal B}
\newcommand*{\A}{\mathcal A}
\renewcommand*{\L}{\mathcal L}
\newcommand*{\defeq}{\mathrel{\vcentcolon=}}
\newcommand*{\congto}{\xrightarrow\sim}

\DeclarePairedDelimiter{\abs}{\lvert}{\rvert}
\DeclarePairedDelimiter{\norm}{\lVert}{\rVert}
\DeclarePairedDelimiterX{\braket}[2]{\langle}{\rangle}{#1\,\delimsize\vert\,\mathopen{}#2}
\DeclarePairedDelimiterX{\BRAKET}[2]{\langle}{\rangle}{\!\delimsize\langle#1\,\delimsize\vert\,\mathopen{}#2\delimsize\rangle\!}
\DeclarePairedDelimiterX{\setgiven}[2]{\{}{\}}{#1\,{:}\,\mathopen{}#2}

\newcommand*{\Gr}{\mathcal{G}}
\newcommand*{\sr}{s}
\newcommand*{\rg}{r}

\newcommand*{\su}{\boldsymbol{s}}
\newcommand*{\pu}{\boldsymbol{p}}

\newcommand*{\onto}{\twoheadrightarrow}

\begin{document}
\title[C*-irreducible regular inclusions]{\(\Cst\)-irreducible regular inclusions,\\
Galois correspondence and aperiodicity}

\author{Bartosz Kosma Kwa\'sniewski}
\email{bartoszk@math.uwb.edu.pl}
 \address{Institute of Mathematics\\
   University  of Bia\l ystok\\
   ul.\@ K.~Cio\l kowskiego 1M\\
   15-245 Bia\l ystok\\
   Poland}

\author{Ralf Meyer}
\email{rmeyer2@uni-goettingen.de}
\address{Mathematisches Institut\\
 Georg-August-Universit\"at G\"ottingen\\
 Bunsenstra\ss e 3--5\\
 37073 G\"ottingen\\
 Germany}

\begin{abstract}
  We characterise \(\Cst\)\nb-irreducible regular
  \(\Cst\)\nb-inclusions \(A\subseteq B\) using a number of different
  conditions considered by different authors.
  In particular, we show that all \(\Cst\)\nb-irreducible regular
  \(\Cst\)\nb-inclusions \(A\subseteq B\) are modelled by outer Fell
  bundles \((B_{g})_{g\in G}\) over discrete groups with a simple unit
  fibre \(A=B_1\).
  In this case, we prove a bijection between intermediate
  \(\Cst\)\nb-algebras \(A\subseteq C \subseteq B\) and subgroups
  \(H\) of~\(G\).
  This extends the Galois correspondence for reduced crossed products
  by discrete group actions established by Cameron--Smith.
  We relate it to the Galois correspondences of Izumi and Mukohara
  for fixed-point algebras of actions of compact abelian groups, and
  the mixed inclusion of a fixed-point algebra in a reduced crossed
  product considered by Echterhoff--Rørdam.

  In addition, using a recent result of Geffen--Ursu, we show that
  the inclusion of a fixed-point subalgebra \(A\subseteq B\) of an
  action of \(\T\)
  or~\(\Z/p\) for a square-free number \(p>0\) is aperiodic if
  and only if~\(A\) detects ideals in~\(B\).
  We apply this to give examples of \(\Cst\)\nb-irreducible inclusions
  coming from Cuntz--Pimsner algebras, including crossed products by
  endomorphisms or transfer operators.
  In particular, we characterise when a core subalgebra of a graph
  \(\Cst\)\nb-algebra is \(\Cst\)\nb-irreducible.
  Lastly,  we show that a general regular topologically graded
  \(\Cst\)\nb-inclusion \(A\subseteq B\) is aperiodic and has a unique
  pseudo-expectation provided~\(A\) detects ideals in all intermediate
  \(\Cst\)\nb-algebras of~\(B\).
  This partially answers a question by Pitts--Zarikian.
\end{abstract}
\subjclass[2010]{46L55, 20M18, 22A22}
\maketitle
\setcounter{tocdepth}{1}

\section{Introduction}
\label{sec:introduction}

Understanding the structure of intermediate algebras arising from an
inclusion of operator algebras is a classical theme in the theory of
von Neumann algebras.
In particular, irreducible inclusions and their associated Galois
correspondences have played a central role in the study of subfactors
\cites{Jones:Index, Popa:Classification_2, Izumi-Longo-Popa:Galois}.
Motivated by this, Rørdam~\cite{Rordam:Irreducible_inclusions}
introduced \(\Cst\)\nb-irreducible inclusions as a
\(\Cst\)\nb-algebraic analogue.
He showed that important classes of inclusions have this property,
including those in \cites{Izumi:Inclusions_simple,
  Cameron-Smith:Galois_reduced, Amrutam:Intermediate_subalgebras,
  Amrutam_Kalantar:Simplicity_intermediate}.
This stimulated further research in several settings, see, for
instance, \cites{Bedos_Omland:C*-irreducibility,
  Li_Scarparo:C*-irreducibility, Echterhoff-Rordam:Inclusions,
  Mukohara:Inclusions, Zarikian:Unique_pseudo,
  Bakshi-Gupta:Regular_inclusions}.

In the present paper we show that any \(\Cst\)\nb-irreducible
inclusion which is regular in the sense of Kumjian and Renault
\cites{Kumjian:Diagonals, Renault:Cartan.Subalgebras} is Morita
equivalent to an inclusion coming from a crossed product of an outer
action of a discrete group by automorphisms on a simple
\(\Cst\)\nb-algebra.
In fact, there is a bijection between isomorphism classes of regular,
\(\Cst\)\nb-irreducible inclusions and isomorphism classes of outer
actions of discrete groups by Morita equivalence bimodules on simple
\(\Cst\)\nb-algebras.
Here an action by Hilbert bimodules is another name for a Fell bundle, see~\cite{Buss-Meyer:Actions_groupoids}.
In addition, we characterise regular, \(\Cst\)\nb-irreducible
inclusions by a wide variety of conditions.
Many of them have already been considered independently by other
authors.
So we are in an interesting, natural setting where many different
concepts converge.
The following theorem (proven on page \pageref{proof:main_theorem})
summarises these equivalent characterisations:

\begin{thmx}[Characterisations of regular \(\Cst\)\nb-irreducible inclusions]
  \label{thm:main_theorem}
  Let \(A\subseteq B\) be a regular \(\Cst\)\nb-inclusion.
  If~\(A\) is simple, then the following conditions are equivalent:
  \begin{enumerate}
  \item \label{enu:main1}%
    \(A\subseteq B\) is \(\Cst\)\nb-irreducible, that is, all
    intermediate \(\Cst\)\nb-algebras \(A\subseteq C\subseteq B\) are
    simple;
  \item \label{enu:main2}%
    \(A\) detects ideals in all intermediate \(\Cst\)\nb-algebras
    \(A\subseteq C \subseteq B\);
  \item \label{enu:main3}%
    \(A\) supports all intermediate \(\Cst\)\nb-algebras \(A\subseteq
    C \subseteq B\);
  \item \label{enu:main4}%
    \(A\subseteq B\) is aperiodic in the sense
    of\/~\cite{Kwasniewski-Meyer:Essential} and the unique
    pseudo-expectation \(E\colon B\to I(A)\) is almost faithful;
  \item \label{enu:main4.5}%
    \(A\subseteq B\) is aperiodic and~\(B\) is simple;
  \item \label{enu:main8}%
    \(B\) is simple and \(A\subseteq B\) is irreducible, that is, \(A'
    \cap \Mult(B) = \C\cdot 1\);
  \item \label{enu:main8.5}%
    \(A\subseteq B\) is irreducible and admits an almost faithful
    pseudo-expectation \(E\colon B\to I(A)\);
  \item \label{enu:main5}%
    there is a unique faithful conditional expectation \(E\colon B\to A\);
  \item \label{enu:main6}%
    there is a faithful conditional expectation \(E\colon B\to A\)
    which is outer in the sense of Osaka~\cite{Osaka:SP_property}, see
    also Izumi~\cite{Izumi:Inclusions_simple};
  \item \label{enu:main7}%
    there is a faithful conditional expectation \(E\colon
    B\to A\) which is  pinching in the sense of
    Rørdam~\cite{Rordam:Irreducible_inclusions};
  \item \label{enu:main9}%
    \(A\subseteq B\) is a noncommutative Cartan subalgebra in the
    sense of Exel~\cite{Exel:noncomm.cartan};
  \item \label{enu:main10}%
    \(B\cong \Cst_\red(\B)\) for an outer, saturated Fell bundle
    \(\B=(B_g)_{g\in G}\) over a discrete group~\(G\) with an
    isomorphism that restricts to the identity \(A=B_1\) on the unit fibre.
  \end{enumerate}
  If~\(A\) is separable, then the above conditions are further
  equivalent to
  \begin{enumerate}[resume]
  \item \label{enu:main11}%
    \(B\cong \Cst_\red(\B)\) for a topologically free, saturated
    Fell bundle \(\B=(B_g)_{g\in G}\) with an isomorphism that
    restricts to the identity \(A=B_1\).
  \end{enumerate}
  If~\(B\) is separable, then the above conditions are further
  equivalent to each of the following:
  \begin{enumerate}[resume]
  \item \label{enu:main12}%
    \(A\subseteq B\) has the almost extension property of
    Nagy--Reznikoff~\cite{Nagy-Reznikoff:Pseudo-diagonals} and~\(B\)
    is simple;
  \item \label{enu:main13}%
    \(A\subseteq B\) has the almost extension property and the unique
    pseudo-expectation \(E\colon B\to I(A)\) is almost faithful.
  \end{enumerate}
  In general, \ref{enu:main13}\(\Leftrightarrow
  \)\ref{enu:main12}\(\Rightarrow \)\ref{enu:main11}\(\Rightarrow
  \)\ref{enu:main1}--\ref{enu:main10}.
  The group~\(G\) and the Fell bundle \(\B=(B_g)_{g\in G}\)
  in~\ref{enu:main10} are uniquely determined up to isomorphism.
\end{thmx}

As a consequence, there is a bijection between the sets of isomorphism
classes of regular, \(\Cst\)\nb-irreducible inclusions
\(A\subseteq B\) and of saturated, outer Fell bundles \(\B=(B_g)_{g\in
  G}\) where~\(G\) is a discrete group and \(B_1=A\).
If~\(B\) is separable and~\(A\) is stable, then \(\B=(B_g)_{g\in G}\)
must be given by an outer action~\(\alpha\) of~\(G\) on~\(A\), and the
action~\(\alpha\) is determined by the inclusion \(A\subseteq B\) up
to cocycle conjugacy (see Corollary~\ref{cor:Gabe_Szabo_actions}).
Restricting further to the case where~\(B\) is nuclear and~\(A\) is
purely infinite, cocycle conjugacy is the same as \(KK^G\)-equivalence
by the dynamical Kirchberg--Phillips Theorem of Gabe--Szab\'o
\cite{Gabe-Szabo:Dynamical_Kirchberg} (see
Corollary~\ref{cor:Gabe_Szabo_inclusions}).

Our second goal is to describe the intermediate
\(\Cst\)\nb-subalgebras for regular \(\Cst\)\nb-irreducible
inclusions.
The characterisation~\ref{enu:main10} above and a Morita globalisation
from~\cite{Kwasniewski-Meyer:Aperiodicity} allow us to reduce this
problem to the case where~\(B\) is the reduced crossed product
\(A\rtimes^\red_\alpha G\) for an outer action~\(\alpha\) by
automorphisms of~\(A\) (see Proposition
\ref{prop:Morita_globalisation}).
When~\(A\) is unital, the resulting Galois connection between
intermediate subalgebras of \(A\subseteq A\rtimes^\red_\alpha G\) and
subgroups of~\(G\) is already known, see
\cites{Cameron-Smith:Galois_reduced,
  Kennedy-Ursu:Intermediate_subalgebras_crossed,
  Rordam:Irreducible_inclusions}.
However, the Morita globalisation process usually produces nonunital
algebras and so the existing results do not apply.
Therefore, for general \(\Cst\)\nb-irreducible inclusions modelled by
Fell bundles, we establish the Galois correspondence directly.
Our approach follows a recent idea of
Kennedy--Ursu~\cite{Kennedy-Ursu:Intermediate_subalgebras_crossed} and
is based on Magajna's Separation Theorem.
In particular, Theorem~\ref{thm:Galois_correspondence_explained}
implies the following result:

\begin{thmx}[Galois correspondence]
  \label{Thm:Galois_correspondence}
  Let \(\B=(B_g)_{g\in G}\) be an outer, saturated Fell bundle over a
  discrete group~\(G\) with a simple unit fibre \(A\defeq B_1\).
  The map
  \(G\supseteq H \mapsto \Cst_\red(\B|_{H})\subseteq \Cst_\red(\B)\)
  is a bijection between the sets of subgroups of~\(G\)
  and of intermediate \(\Cst\)\nb-subalgebras for the inclusion
  \(A\subseteq \Cst_\red(\B)\).
\end{thmx}

The above result contains twisted outer actions \((\alpha,\sigma)\) of
a discrete group~\(G\) as a special case.
So we have generalised the Galois correspondence of
Cameron--Smith~\cite{Cameron-Smith:Galois_reduced} for inclusions
\(A\subseteq A\rtimes^\red_{\alpha,\sigma} G\) to the nonunital case
(see Corollary~\ref{cor:Cameron-Smith}).
By passing to spectral subspaces, our results also apply to the
inclusion \(A^\beta\subseteq A\) of a fixed-point algebra of faithful
action~\(\beta\) of a compact abelian group~\(K\) on a
\(\Cst\)\nb-algebra~\(A\).
We characterise when \(A^\beta\subseteq A\) is \(\Cst\)\nb-irreducible
by a number of conditions, one of which is a Galois correspondence
between closed subgroups of~\(K\) and intermediate
\(\Cst\)\nb-algebras of \(A^\beta\subseteq A\) (see
Theorem~\ref{thm:Mukohara_compact_abelian}).
This recovers and extends recent deep results of
Mukohara~\cite{Mukohara:Inclusions} and
Izumi~\cite{Izumi:Minimal_compact} in the special case where~\(K\) is
abelian.
In particular, the relevant conditions are related to quasi-product
actions and strongly prime inclusions, which were studied by Bratteli,
Elliott, Evans and Kishimoto in
\cite{Bratteli-Elliott-Evans-Kishimoto:Quasi-product},
\cite{Bratteli-Elliott-Kishimoto:Quasi-product}.
If the actions \(\alpha\) and~\(\beta\) above commute, then the
mixed inclusion \(A^\beta \subseteq A\rtimes^{\red}_{\alpha,\sigma}
G\) is regular as well, and so our results apply and lead to a
far-reaching generalisation of the corresponding Galois correspondence
obtained by Echterhoff--Rørdam~\cite{Echterhoff-Rordam:Inclusions}
(see Theorem~\ref{thm:Echterhoff-Rordam}).

The proofs of Theorems \ref{thm:main_theorem}
and~\ref{Thm:Galois_correspondence} rely crucially on the notion of
aperiodicity and its properties, which were developed by the authors
in \cites{Kwasniewski-Meyer:Aperiodicity, Kwasniewski-Meyer:Cartan,
  Kwasniewski-Meyer:Essential,
  Kwasniewski-Meyer:Aperiodicity_pseudo_expectations}.
We also take the opportunity to formulate some results in this theory
in a cleaner way, without any assumptions on the algebra~\(A\) such as
separability or Type~I.
In particular, combining a recent result of
Geffen--Ursu~\cite{Geffen-Ursu:Simple_crossed_products} and Morita
globalisation techniques, we prove the following (see
Theorem~\ref{the:special_groups}):

\begin{thmx}[Inclusions graded by \(\Z\) or square-free cyclic groups]
  \label{thmx:fixed_point_for special_compact}
  Let~\(A\) be the fixed-point algebra for a faithful action on~\(B\)
  of\/ \(\T\) or~\(\Z/p\) for a square-free number \(p>0\).
  Then~\(A\) detects ideals in~\(B\) if and only if \(A\subseteq B\)
  is aperiodic.
  In particular, \(A\subseteq B\) is \(\Cst\)\nb-irreducible if and
  only if \(A\) and~\(B\) are simple.
\end{thmx}

The last part of this theorem gives a very efficient criterion when
the inclusion of the fixed-point algebra of a circle action is
\(\Cst\)\nb-irreducible.
This applies, for instance, to the core inclusions of Cuntz--Pimsner
algebras (see Theorem~\ref{thm:irreducible_from_Cuntz_Pimsner}).
We also discuss some examples coming from crossed products by
endomorphisms and transfer operators, and we characterise when the
core inclusion associated to a directed graph is
\(\Cst\)\nb-irreducible (see
Theorem~\ref{thm:irreducible_graph_algebras}).

In \cite{Pitts-Zarikian:Unique_pseudoexpectation}*{Section 7.1, Q6},
Pitts and Zarikian asked if a \(\Cst\)\nb-inclusion \(A\subseteq B\)
must have a unique pseudo-expectation when~\(A\) detects ideals in all
intermediate \(\Cst\)\nb-algebras.
Zarikian showed in~\cite{Zarikian:Unique_pseudo} that this fails for
irregular inclusions and therefore restricted the question to regular
inclusions in \cite{Zarikian:Unique_pseudo}*{Question 3.2.1}.
We answer this question affirmatively for topologically graded regular
inclusions.
More specifically, regular inclusions are always inverse-semigroup
graded~\cite{Kwasniewski-Meyer:Stone_duality}.
We extend Exel's notion of a topological grading by a group
from~\cite{Exel:Partial_dynamical} to the setting of inverse-semigroup
gradings and show that the corresponding inclusions are equivalent to
exotic inverse-semigroup crossed products in the sense
of~\cite{Kwasniewski-Meyer:Essential}.
We use Theorem~\ref{thmx:fixed_point_for special_compact} to prove the
following result, which also improves some of the results
in~\cite{Kwasniewski-Meyer:Aperiodicity_pseudo_expectations} (see
Theorem~\ref{thm:topologically_graded_aperiodic} and
Proposition~\ref{prop:topological_grading_characterisation}):

\begin{thmx}[Topologically graded regular inclusions]
  \label{thmx:topologically_graded_aperiodic}
  Let \(A\subseteq B\) be a regular, topologically graded
  \(\Cst\)\nb-inclusion.
  The following are equivalent:
  \begin{enumerate}
  \item \label{enu:topologically_graded_aperiodic1}%
    \(A\subseteq B\) is aperiodic and admits a faithful
    pseudo-expectation;
  \item \label{enu:topologically_graded_aperiodic2}%
    \(A\subseteq B\) has a unique pseudo-expectation and this
    pseudo-expectation is faithul;
  \item \label{enu:topologically_graded_aperiodic3}%
    \(A\) supports all intermediate \(\Cst\)\nb-algebras \(A\subseteq
    C\subseteq B\);
  \item \label{enu:topologically_graded_aperiodic4}%
    \(A\) detects ideals in all intermediate \(\Cst\)\nb-algebras
    \(A\subseteq C \subseteq B\);
  \item \label{enu:topologically_graded_aperiodic5}%
    \(B\cong \Cst_\ess(\B)\) is the essential \(\Cst\)\nb-algebra for
    an aperiodic saturated Fell bundle \(\B=(B_s)_{s\in S}\) over a
    unital inverse semigroup~\(S\), and the isomorphism restricts to
    the identity \(A=B_1\).
  \end{enumerate}
  If \(A\subseteq B\) admits a conditional expectation~\(E\), then the
  above equivalent conditions hold if and only if~\(E\) is faithful
  and pinching.
\end{thmx}

The paper is organised as follows.
In Section~\ref{sec:General inclusions and aperiodicity} we discuss
the nontriviality conditions in Theorem~\ref{thm:main_theorem} for
general \(\Cst\)\nb-inclusions.
In Section~\ref{sec:Topologically graded regular C*-inclusions} we
consider regular inclusions and the inverse-semigroup Fell bundles
that model them.
In particular, we prove Theorems \ref{thmx:fixed_point_for
  special_compact} and~\ref{thmx:topologically_graded_aperiodic}.
Section~\ref{sec:regular C*-irreducible inclusions} is devoted to
regular \(\Cst\)\nb-irreducible inclusions.
It contains proofs of Theorem~\ref{thm:main_theorem} and of a slightly
more general version of Theorem~\ref{Thm:Galois_correspondence}.
We also apply our results to study tensor products of
\(\Cst\)\nb-irreducible inclusions with simple \(\Cst\)\nb-algebras.
Section~\ref{sec:group_actions} discusses specific applications to
crossed products by twisted group actions and fixed-point algebras for
compact group actions.
Finally, in Section~\ref{sect:Cuntz-Pimnser}, we present a number of
examples of \(\Cst\)\nb-irreducible inclusions which are special cases
of core inclusions for Cuntz--Pimsner algebras.

\subsection*{Acknowledgments}

The first-named author would like to thank to Tristan Bice, Karen
Strung and R\'eamonn \'O Buachalla for the hospitality and invitation
to Prague, where Theorem~\ref{thm:main_theorem} saw the light of day
at a seminar, a recording of which is available on YouTube and cited
in~\cite{Zarikian:Unique_pseudo}.
He is grateful to Kang Li for the invitation to deliver a series of
lectures in Erlangen, where we managed to persuade Dan Ursu to include
Corollary~8.6 in~\cite{Geffen-Ursu:Simple_crossed_products}, which
then allows us to prove Theorems \ref{thmx:fixed_point_for
  special_compact} and~\ref{thmx:topologically_graded_aperiodic}.
The first-named author would further like to thank Matt Kennedy for
his request to read the
preprint~\cite{Kennedy-Ursu:Intermediate_subalgebras_crossed}, which
proved crucial to the proof of
Theorem~\ref{Thm:Galois_correspondence}.
He is grateful to Yemon Choi for suggesting that the Galois
correspondence be placed in the context of a Galois connection; this
led us to improve Theorem~\ref{Thm:Galois_correspondence} to
Theorem~\ref{thm:Galois_correspondence_explained}.
Finally, he would like to thank Siegfried Echterhoff for discussions
about the assumptions in the theorems
of~\cite{Echterhoff-Rordam:Inclusions}, and to David Pask for a
discussion concerning the right formulation of
\cite{Pask-Rho:simple_graph_core}*{Theorem~6.1}.
The research of Bartosz Kwa\'sniewski was supported by the National
Science Centre, Poland, through the WEAVE-UNISONO grant
no.~2023/05/Y/ST1/00046.
This work is part of the project Graph Algebras partially supported by
EU grant HORIZON-MSCA-SE-2021 Project 101086394.

\section{General C*-inclusions and aperiodicity}
\label{sec:General inclusions and aperiodicity}

\begin{definition}[\cite{Rordam:Irreducible_inclusions}*{Definition 3.1}]
  Let \(A\subseteq B\) be an inclusion of \(\Cst\)\nb-algebras.
  An \emph{intermediate \(\Cst\)\nb-algebra} for \(A\subseteq B\) is a
  \(\Cst\)\nb-subalgebra \(C\subseteq B\) that contains~\(A\), that
  is, \(A\subseteq C \subseteq B\).
  The \(\Cst\)\nb-inclusion \(A\subseteq B\) is
  \emph{\(\Cst\)\nb-irreducible} if all intermediate
  \(\Cst\)\nb-algebras are simple.
\end{definition}

Rørdam's original definition in~\cite{Rordam:Irreducible_inclusions}
assumed the inclusion to be unital.
Mukohara in \cite{Mukohara:Inclusions}*{Definition 3.8} still assumes
the inclusion \(A\subseteq B\) to be \emph{nondegenerate}, that is,
\(AB=B\) or, equivalently, \(A\) and \(B\) have a common approximate
unit.

\begin{definition}
  A \(\Cst\)\nb-inclusion \(A\subseteq B\) is \emph{irreducible} if
  the commutant of~\(A\) in the multiplier algebra of \(\Mult(B)\) is
  trivial, that is, \(A' \cap \Mult(B) = \C\cdot 1\).
\end{definition}

\begin{lemma}[see \cite{Rordam:Irreducible_inclusions}*{Remark~3.8}]
  \label{C-irreducible_implies_irreducible}
  Every nondegenerate \(\Cst\)\nb-irreducible inclusion is irreducible.
\end{lemma}

\begin{proof}
  Assume that \(A\subseteq B\) is nondegenerate and that \(A\subseteq
  B\) is \(\Cst\)\nb-irreducible.
  Let \(b\in A' \cap \Mult(B)\) be self-adjoint and nonzero.
  Then \(C\defeq A+ bA\) is an intermediate \(\Cst\)\nb-algebra, which
  makes it simple, and~\(bA\) is a nonzero ideal in~\(C\).
  Thus \(bA=A\).
  As a self-adjoint and surjective multiplier of~\(A\), \(b\) must be
  invertible.
  The only \(\Cst\)\nb-algebra where all nonzero self-adjoint elements are
  invertible is~\(\C\).
  Hence \(A\subseteq B\) is irreducible.
\end{proof}

\begin{example}
  For an infinite-dimensional Hilbert space~\(\mathcal{H}\), the
  unital inclusion
  \(\Comp(\mathcal{H})+\C 1\subseteq \Bound(\mathcal{H})\) is
  irreducible but not \(\Cst\)\nb-irreducible.  In fact, there is no
  simple intermediate \(\Cst\)\nb-algebra in this case.
\end{example}

The following example shows, among other things, that there are unital
inclusions of simple \(\Cst\)\nb-algebras that are irreducible but not
\(\Cst\)\nb-irreducible.

\begin{example}[Crossed products by discrete groups]
  \label{ex:crossed_product_inclusions}
  Let \(B\defeq A\rtimes^\red_\alpha G\) be the reduced crossed
  product for an action~\(\alpha\) of a discrete group~\(G\) on a
  unital \(\Cst\)\nb-algebra~\(A\).
  Then both \(A\) and \(\Cst_\red(G)\) are unital
  \(\Cst\)\nb-subalgebras of~\(B\).
  Following Izumi~\cite{Izumi:Inclusions_simple} and
  Kishimoto~\cite{Kishimoto:Outer_crossed},
  Cameron--Smith~\cite{Cameron-Smith:Galois_reduced} proved that
  \(A\subseteq B\) is \(\Cst\)\nb-irreducible if and only if~\(A\) is
  simple and the action~\(\alpha\) is outer and that, in this case,
  any intermediate \(\Cst\)\nb-algebra is of the form
  \(C=A\rtimes^\red_\alpha H\) for a subgroup \(H\subseteq G\).
  The inclusion \(A\subseteq B\) is irreducible if and only if it is
  \(\Cst\)\nb-irreducible, see
  \cite{Rordam:Irreducible_inclusions}*{Theorem~5.8}.
  As shown in~\cite{Rordam:Irreducible_inclusions} (see
  also~\cite{Amrutam_Kalantar:Simplicity_intermediate}),
  \(\Cst_\red(G)\subseteq B\) is \(\Cst\)\nb-irreducible if and only
  if~\(G\) is \(\Cst\)\nb-simple and for every \(a\in
  A^+\setminus\{0\}\) there is a finite subset \(F\subseteq G\) with
  \(\sum_{g\in F} \alpha_g(a)\ge 1\).
  In some cases, intermediate \(\Cst\)\nb-algebras for
  \(\Cst_\red(G)\subseteq B\) are described as algebras of the form
  \(C=D\rtimes_\red G\) for some \(G\)-\(\Cst\)\nb-subalgebra
  \(D\subseteq B\), see~\cite{Amrutam:Intermediate_subalgebras}.
  Rørdam gave an example of an action on the Cuntz
  algebra~\(\mathcal{O}_\infty\) such that the corresponding
  inclusion \(\Cst_\red(G)\subseteq
  \mathcal{O}_\infty\rtimes_\red G\) is irreducible and
  \(\Cst_\red(G)\) and \(\mathcal{O}_\infty\rtimes_\red G\) are
  simple, but the inclusion is not \(\Cst\)\nb-irreducible, see
  \cite{Rordam:Irreducible_inclusions}*{Example~5.14}.
\end{example}

The main aim of this note is to generalise the first part of the above
example to regular inclusions and to provide some more general
applications of aperiodicity, which is a concept that generalises
outer actions on simple \(\Cst\)\nb-algebras to modules over general
\(\Cst\)\nb-algebras.

\begin{definition}[{\cite{Kwasniewski-Meyer:Essential}*{Definition~5.9}}]
  \label{def:aperiodic_bimodule}%
  Let~\(A\) be a \(\Cst\)\nb-algebra and let \(\Her(A)\) be the set of
  nonzero hereditary \(\Cst\)\nb-subalgebras \(D\subseteq A\).
  For \(D\in\Her(A)\), let
  \[
    D^+_1 \defeq \setgiven{a\in D}{a\ge 0,\, \norm{a}=1}.
  \]
  We call a normed \(A\)\nb-bimodule~\(M\) \emph{aperiodic} if for any
  \(m\in M\), \(D\in \Her(A)\), and \(\varepsilon>0\), there is \(a\in
  D^+_1\) with \(\norm{a \cdot m\cdot a}<\varepsilon\).
\end{definition}

\begin{example}
  \label{ex:conditions_automorphisms}
  Let \(\alpha\colon A\to A\) be a \Star{}automorphism.
  Define an \(A\)\nb-bimodule \(M_\alpha\defeq A\) with the operations
  \(a\cdot x=\alpha(a)x\) and \(x\cdot a=xa\) for all \(x\in
  M_{\alpha}\) and \(a,b\in A\).
  In~\cite{Kishimoto:Freely_acting}, Kishimoto calls~\(\alpha\)
  \emph{freely acting}, which he quickly renamed \emph{properly
    outer}, see \cite{Hamana:injective-C-dyn}*{p.~477}, if for every
  nonzero \(\alpha\)\nb-invariant ideal~\(I\) in~\(A\) the Borchers
  spectrum of the \(\Z\)\nb-action generated by~\(\alpha|_I\) is
  nontrivial.
  This terminology seems to now become standard, see
  \cites{Hamana:injective-C-dyn, Zarikian:Unique_expectations,
    Zarikian:Unique_pseudo,
    Kennedy-Ursu:Intermediate_subalgebras_crossed}, though some other
  names such as \emph{spectrally nontrivial}
  in~\cite{Pasnicu-Phillips:Spectrally_free} or \emph{Kishimoto
    condition} in~\cite{Kwasniewski-Meyer:Aperiodicity} occur
  sometimes.
  By~\cite{Kishimoto:Freely_acting}*{Theorem~2.1} (see also
  \cite{Kwasniewski-Meyer:Aperiodicity}*{Theorem~2.10}),
  \(\alpha\) is properly outer if and only if~\(M_{\alpha}\) is
  aperiodic, that is, if \(b\in B\) and \(D\in \Her(A)\), then
  \(\inf\setgiven{\norm{\alpha(a)ba}}{ a\in D^+_1}=0\).
  By a result of Olesen, \cite{Olesen:Inner}*{Corollary 4.3}, if~\(A\)
  is simple, then this is further equivalent to~\(\alpha\) being
  \emph{outer}, that is, there is no unitary \(u\in \Mult(A)\) with
  \(\alpha(a)= u a u^*\) for all \(a\in A\).
  We call~\(\alpha\) \emph{topologically free} if the set of fixed
  points of the dual homeomorphism \(\widehat{\alpha}\colon
  \widehat{A}\to \widehat{A}\) on the space of isomorphism classes of
  irreducible representations of~\(A\) has empty interior
  in~\(\widehat{A}\).
  By \cite{Kwasniewski-Meyer:Aperiodicity}*{Theorem~8.1},
  topologically free homeomorphisms are properly outer, and the converse
  holds whenever~\(A\) contains an essential ideal which is either
  separable or of Type~I.
\end{example}

The bimodule~\(M_\alpha\) in the above example becomes a Hilbert
bimodule with the left and right inner products \({}_A\braket{x}{y}
\defeq \alpha^{-1}(xy^*)\) and \(\braket{x}{y}_A\defeq x^*y\) for all
\(x,y\in M_{\alpha}=A\).
The last part of the above example generalises to Hilbert bimodules as
follows:

\begin{example}
  Let~\(M\) be a Hilbert \(A\)\nb-bimodule.
  We call~\(M\) \emph{outer} if it is not isomorphic to the trivial
  Hilbert \(A\)\nb-bimodule~\(A\), and \emph{purely outer} if there
  is no nonzero \(M\)\nb-invariant ideal~\(I\) in~\(A\) for which the
  restriction \(M|_I = M\cdot I = I\cdot M\cdot I\) is outer.
  We call \(M\) \emph{topologically free} (or \emph{topologically
    nontrivial}) if the set of fixed points for the induced partial
  homeomorphism~\(\widehat{M}\) of~\(\widehat{A}\) has empty interior.
  The known implications between these conditions and the aperiodicity
  of~\(M\) are presented in Figure~\ref{fig:diagram},
  \begin{figure}[htbp]
  $
  \xymatrix{
  \text{$\left(\begin{array}{c}
  \text{topologically}
  \\
  \text{free}\end{array}\right)
  $}
  &  &  \ar@{->}@/^2pc/[ll]_{\text{$A$ ess. separable}}  \ar@{<=}[ll]
  \text{$\left(\begin{array}{c}
  \text{aperiodic}
  \\
  \end{array}\right)
  $}
  &  & \ar@{->}@/^4pc/[llll]_{\text{$A$ ess. Type I}}
  \ar@{<=}[ll]
  \ar@{->}@/^2pc/[ll]_{\text{$A$ ess. simple}}
  \text{$\left(\begin{array}{c}
  \text{purely}
  \\
  \text{outer}
  \end{array}\right)$}
  }
  $
  \caption{Implications among properties of Hilbert $\Cst$-bimodules over~$A$.}
  \label{fig:diagram}
  \end{figure}
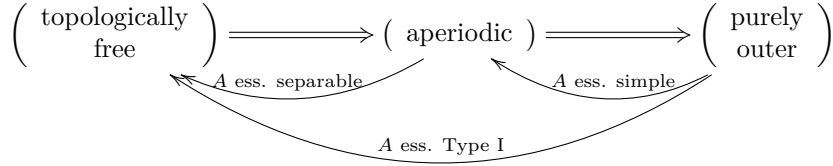
  where an implication marked by a thin arrow holds under the
  assumption that~$A$ contains an essential ideal with the given
  property (is separable, Type~I or simple).
  These implications follow from
  \cite{Kwasniewski-Meyer:Aperiodicity_pseudo_expectations}*{Theorem~3.6}
  and \cite{Kwasniewski-Meyer:Aperiodicity}*{Theorem~8.1}.
  In particular, if \(A\) is simple, then~\(M\) is aperiodic if and
  only if it is outer.
\end{example}

For any \(\Cst\)\nb-inclusion \(A\subseteq B\), the
\(\Cst\)\nb-algebra~\(B\) is an \(A\)\nb-bimodule, and this bimodule
structure descends to the quotient Banach space~\(B/A\).
This allows us to define aperiodicity for \(\Cst\)\nb-inclusions:

\begin{definition}[\cite{Kwasniewski-Meyer:Essential}*{Definition~5.14}]
  A \(\Cst\)\nb-inclusion \(A\subseteq B\) is \emph{aperiodic} if the
  Banach \(A\)\nb-bimodule~\(B/A\) is aperiodic, that is, if for any
  \(b\in B\), \(D\in \Her(A)\) and \(\varepsilon>0\), there are \(a\in
  D^+_1\) and \(c\in A\) with \(\norm{a \cdot b\cdot a
    -c}<\varepsilon\).
\end{definition}

Aperiodic inclusions are closely related to the outer expectations
introduced by Osaka~\cite{Osaka:SP_property} for unital
\(\Cst\)\nb-inclusions.
We generalise the latter to general \(\Cst\)\nb-inclusions:

\begin{definition}[\cite{Osaka:SP_property}*{Definition~2.2}]
  Let \(A\subseteq B\) be a \(\Cst\)\nb-inclusion.
  A conditional expectation \(E\colon B\to A\subseteq B\) is
  \emph{outer} if \(\inf\setgiven{\norm{aba}}{a\in D^+_1}=0\) for
  any \(b\in \ker E\) and \(D\in \Her(A)\).
\end{definition}

\begin{remark}
  \label{rem:outer_expectation_means_aperiodic}
  For any conditional expectation \(E\colon B\to A\), the map \(\ker E
  \ni b-E(b) \mapsto b+A \in B/A\) is a bijective linear contraction
  and hence a linear homeomorphism \(\ker E\cong B/A\).
  Thus~\(E\) is outer if and only if \(A\subseteq B\) is aperiodic.
  Therefore, Osaka's outerness is not a property of a conditional
  expectation but of the inclusion.
\end{remark}

If~\(E\) is outer, then a standard argument shows the following
pinching property due to Rørdam, which has appeared in several
papers:

\begin{definition}[\cite{Rordam:Irreducible_inclusions}*{Definition~3.13}]
  A conditional expectation \(E\colon B\to A\) is \emph{pinching} if
  for each \(b\in B^+\) and \(\varepsilon >0\) there is a contraction
  \(a\in A^+\) such that
  \[
    \norm{ab a- aE(b)a}<\varepsilon
    \quad\text{ and }\quad
    \norm{aE(b)a}\ge \norm{E(b)}- \varepsilon.
  \]
\end{definition}

\begin{lemma}
  \label{lem:aperiodic_implies_pinching}
  Assume that a conditional expectation \(E\colon A\to B\) is outer
  or, equivalently, the \(\Cst\)\nb-inclusion \(A\subseteq B\) is
  aperiodic.
  Then~\(E\) is pinching.
\end{lemma}

\begin{proof}
  See, for instance, the proof of \cite{Osaka:SP_property}*{Corollary~2.3}.
\end{proof}

\begin{remark}
  We do not know an example of a pinching conditional expectation
  \(E\colon A\to B\) which is not outer.
\end{remark}

\begin{definition}[\cite{Kwasniewski-Meyer:Aperiodicity}*{Definition~2.2},
  \cite{Kwasniewski:Crossed_products}*{Definition~2.39}]
  \label{def:Cuntz_preorder}
  Let \(A\subseteq B\) be a \(\Cst\)\nb-inclusion.
  We say that~\(A\) \emph{detects ideals} in~\(B\) if \(J\cap A\neq
  0\) for all nonzero ideals~\(J\) in~\(B\).
  We say that~\(A\) \emph{supports}~\(B\) if for each \(b\in
  B^+\setminus\{0\}\) there is \(a\in A^+\setminus\{0\}\) which is
  Cuntz below~\(b\) in~\(B\), written, \(a\precsim b\).
\end{definition}

\begin{remark}
  \label{rem:supporting_detecting}
  Detection of ideals is often called the (ideal) \emph{intersection
    property} (see~\cite{Sierakowski:IdealStructureCrossedProducts}).
  Pitts and Zarikian prefer to speak of \emph{essential inclusions}
  instead (see \cites{Pitts-Zarikian:Unique_pseudoexpectation,
    Zarikian:Unique_pseudo}).
  If~\(A\) supports~\(B\), then it detects ideals in~\(B\) because
  Cuntz's below-relation \(a\precsim b\) implies that~\(a\) is in the
  ideal generated by~\(b\).
  However, the converse is not true.
  Indeed, \cite{Rordam:finite_and_infinite}*{Corollary~7.1
    and Proposition~2.1} give a unital \(\Cst\)\nb-inclusion
  \(\mathcal{O}_\infty\subseteq B\) where~\(B\) is a nuclear, unital,
  separable, simple \(\Cst\)\nb-algebra satisfying the UCT and there
  is a nonzero finite projection \(q_0\in B\).
  Then~\(\mathcal{O}_\infty\) trivially detects ideals in~\(B\)
  as~\(B\) is simple, but it does not support~\(B\) because for any
  nonzero \(a\in \mathcal{O}_\infty^+\), there is a properly infinite
  projection \(p\in \mathcal{O}_\infty^+\) with \(p\precsim a\), and
  so \(a\precsim q_0\) would imply \(p\le q_0\), which is impossible
  because~\(q_0\) is finite.
\end{remark}

The following lemma generalises the part of
\cite{Rordam:Irreducible_inclusions}*{Proposition 3.15} about the
pinching property.

\begin{lemma}
  \label{lem:pinching_plus_faithful}
  Let \(E\colon A\to B\) be a pinching, faithful conditional expectation.
  Then~\(A\) supports all intermediate \(\Cst\)\nb-algebras, and hence
  it also detects ideals in all intermediate \(\Cst\)\nb-algebras.
  In particular, if~\(A\) is simple then \(A\subseteq B\) is
  \(\Cst\)\nb-irreducible.
\end{lemma}

\begin{proof}
  We follow the argument in the proof of
  \cite{Rordam-Sierakowski:Purely_infinite}*{Lemma~3.2}.
  Let \(b\in B^+\setminus \{0\}\).
  Since~\(E\) is faithful, \(E(b)\neq 0\).
  We may assume without loss of generality that \(\norm{E(b)}=1\).
  The pinching property with \(\varepsilon=1/2\) gives us a
  contraction \(h\in A^+\) with \(\norm{hE(b)h-hbh}<1/2\) and
  \(\norm{hE(b)h} >1/2\).
  Then \(a \defeq(hE(b)h-1/2)_+ \in A^+\) satisfies \(a \precsim
  hbh\), see \cite{Rordam:On_simple_II}*{Proposition~2.2}.
  Hence \(a\precsim b\) and clearly \(a\neq 0\).
  Thus~\(A\) supports~\(B\).
  The same argument works for intermediate \(\Cst\)\nb-algebras
  because any restriction of~\(E\) remains pinching.
\end{proof}

For a pinching conditional expectation, we may compress any \(b\in B\)
with elements in~\(A\) to get an element with norm close
to~\(\norm{E(b)}\).
We are going to prove that if the \(\Cst\)\nb-inclusion is aperiodic
and~\(A\) is simple, then we may compress~\(b\) to an element close
to~\(E(b)\).
To this end, we will use a recent variant of Magajna's separation
theorem by Kennedy and Ursu
in~\cite{Kennedy-Ursu:Intermediate_subalgebras_crossed}.
These results are stated for absolutely \(A\)\nb-convex sets and
unital \(\Cst\)\nb-algebras.
Here we get rid of the unitality requirement and restrict attention to
\(A\)\nb-bimodules, as this is sufficient for our purposes.

\begin{proposition}[Magajna--Kennedy--Ursu Separation Theorem for
  bimodules]
  \label{thm:Magajna_Kennedy_Ursu}
  Let \(A\subseteq B\) be a \(\Cst\)\nb-inclusion, let \(M\subseteq
  B\) be a norm-closed \(A\)-bimodule, and let \(a_0\in A\setminus
  M\).
  Then there is a completely bounded map \(\phi\colon B\to \Bound(H)\)
  such that \(\phi|_M=0\), \(\phi(a_0)\neq 0\), \(\pi\defeq\phi|_A\)
  is a \Star{}homomorphism, and \(\phi(ab)=\pi(a)\phi(b)\),
  \(\phi(ba)=\phi(b)\pi(a)\) for all \(a\in A\), \(b\in B\).
\end{proposition}

\begin{proof}
  If \(A\) and~\(B\) do not have a common unit, we may pass to the
  unital \(\Cst\)\nb-inclusion \(\widetilde{A}\subseteq
  \widetilde{B}\) obtained by adding a common unit.
  Then~\(M\) is still a closed \(\widetilde{A}\)\nb-bimodule, and so
  all the more it is a closed absolutely \(A\)\nb-convex set.
  Thus
  \cite{Kennedy-Ursu:Intermediate_subalgebras_crossed}*{Proposition
    3.4} provides \(\delta>0\) and a \Star{}homomorphism \(\pi\colon
  A\to \Bound(X)\) that extends to a completely bounded
  \(A\)\nb-bimodule map \(\phi\colon B\to \Bound(H)\) such that
  \(\norm{\phi(a_0)}> \delta\) and \(\phi(m)<\delta\) for all \(m\in
  M\).
  Since~\(M\) is a linear space, the last inequality is equivalent to
  \(\phi|_M=0\).
\end{proof}

\begin{proposition}
  \label{prop:module_approximation}
  Let \(A\subseteq B\) be an aperiodic \(\Cst\)\nb-inclusion with a
  conditional expectation \(E\colon B\to A\subseteq B\).
  Let~\(A\) be simple and let \(b\in B\).
  Then \(E(b)\in \overline{AbA}\).
\end{proposition}

\begin{proof}
  Assume that \(E(b)\notin \overline{AbA}\).
  We are going to show that \(m\defeq E(b)^*(E(b)-b)\in \ker E\) fails
  to satisfy the aperiodicity condition in
  Definition~\ref{def:aperiodic_bimodule}.
  Then \(A\subseteq B\) fails to be aperiodic by
  Remark~\ref{rem:outer_expectation_means_aperiodic}.
  Proposition~\ref{thm:Magajna_Kennedy_Ursu} gives a
  \Star{}homomorphism \(\pi\colon A\to \Bound(H)\) with
  \(\pi(E(b))\neq0\) and a completely bounded \(A\)\nb-bimodule map
  \(\phi\colon B\to \Bound(H)\) that extends~\(\pi\) and satisfies
  \(\phi(\overline{AbA})=0\).
  Since~\(A\) is simple and \(\pi(E(b))\neq0\), the
  \Star{}homomorphism~\(\pi\) must be isometric.
  By \cite{Kwasniewski-Meyer:Aperiodicity}*{Lemma~2.19}, there is
  \(D\in \Her(A)\) such that \(\norm{E(b)a}\ge \norm{E(b)}/2\) for
  all \(a\in D^+_1\).
  If \(a\in D^+_1\), then
  \[
    \phi(aE(b)^*(E(b)-b)a)
    = \phi(aE(b)^* E(b)a)-\phi(aE(b)^*ba)=\pi(aE(b)^* E(b)a)
  \]
  and \(\norm{\pi(aE(b)^* E(b)a)}=\norm{E(b)a}^2\ge
  \norm{E(b)}^2/4\).
  This implies
  \[
    \norm{aE(b)^*(E(b)-b)a}\ge  \frac{\norm{E(b)}^2}{4\norm{\phi}}>0
  \]
  for all \(a\in D_1^+\).
  This contradicts aperiodicity.
\end{proof}

\begin{remark}
  The above \(A\)\nb-module approximation for general conditional
  expectations may be viewed as a counterpart of an absolutely
  \(A\)\nb-convex approximation for reduced crossed products in
  \cite{Kennedy-Ursu:Intermediate_subalgebras_crossed}*{Proposition~3.7}
  or an even more concrete approximation result for single
  automorphisms in
  \cite{Cameron-Smith:Galois_reduced}*{Proposition~3.1}.
\end{remark}

Pseudo-expectations were introduced by
Pitts~\cite{Pitts-Zarikian:Unique_pseudoexpectation} as a replacement
for conditional expectations in settings where no conditional
expectation exists.
They use Hamana's \emph{injective envelope} \(I(A)\) of a
\(\Cst\)\nb-algebra~\(A\) introduced in
\cite{Hamana:Injective-Envelope-Cstar}*{Definition~2.2} (and extended
to the nonunital case in \cite{Hamana:tensorI}*{Section~6}).
Namely, for any \(\Cst\)\nb-inclusion \(A\subseteq B\), the canonical
\(\Cst\)\nb-inclusion \(A\subseteq I(A)\) extends to a completely
positive contraction \(E\colon B\to I(A)\)
because \(I(A)\) is injective.
Such maps~\(E\) are called \emph{pseudo-expectations} for \(A\subseteq
B\).
They are a special case of what we call generalised expectations.
A \emph{generalised expectation} for \(A\subseteq B\) consists of
another \(\Cst\)\nb-inclusion \(A\subseteq \tilde{A}\) and a
completely positive, contractive map \(E\colon B \to \tilde{A}\) that
restricts to the identity map on~\(A\).
Then~\(E\) is necessarily an \(A\)\nb-bimodule map and
\(\Null_E=\setgiven{b\in B}{E(x b y)=0 \text{ for all }x,y\in B}\) is
the largest ideal in~\(A\) which is contained in \(\ker E\).
We call~\(\Null_E\) the \emph{ideal kernel} for~\(E\), and call~\(E\)
\emph{almost faithful} if \(\Null_E=0\).
In general, since \(\Null_E\cap A=0\), the inclusion \(A\subseteq B\)
descends to the \emph{essential inclusion} \(A\subseteq B_\ess\defeq
B/\Null_E\) and~\(E\) descends to an almost faithful expectation
\(E_\ess\colon B_\ess\to \tilde{A}\).
By \cite{Kwasniewski-Meyer:Essential}*{Lemma~3.10}, \(E_\ess\) is
faithful if and only if~\(E\) is \emph{symmetric} in the sense that
\(E(b^*b)=0\) implies \(E(bb^*)=0\) for all \(b\in B\).
A generalised expectation is faithful if and only if it is symmetric
and almost faithful.
For symmetric conditional expectations, the ideal kernel is
\[
\Null_E=\setgiven{b\in B}{E(b^*b)=0}.
\]
We shall see that the difference between faithful and almost faithful
pseudo-expectations is closely related to the difference between
detecting ideals in a \(\Cst\)\nb-algebra and in all intermediate
\(\Cst\)\nb-subalgebras.

\begin{lemma}
  A \(\Cst\)\nb-subalgebra~\(A\) detects ideals in~\(B\) if and only
  if all pseudo-expectations for \(A\subseteq B\) are almost faithful.
\end{lemma}

\begin{proof}
  If~\(A\) detects ideals in~\(B\), then every generalised
  expectations for \(A\subseteq B\) is almost faithful by
  \cite{Kwasniewski-Meyer:Essential}*{Corollary~3.9}.
  Assume then that~\(A\) does not detect ideals in~\(B\).
  So there is a nonzero ideal~\(J\) in~\(B\) with \(A\cap J=0\).
  Then the canonical map is an embedding \(A\hookrightarrow B/J\), and
  so the inclusion \(A\subseteq I(A)\) extends to a completely
  positive contraction \(F\colon  B/J\to I(A)\).
  Composing~\(F\) with the quotient map \(B\to B/J\) gives a
  pseudo-expectation which has~\(J\) in its kernel.
  Hence it is not almost faithful.
\end{proof}

Pitts and Zarikian proved a hereditary version of the above lemma for
unital inclusions in
\cite{Pitts-Zarikian:Unique_pseudoexpectation}*{Theorem 3.5}, which
was generalised to arbitrary inclusions by Zarikian in
\cite{Zarikian:Unique_pseudo}:

\begin{proposition}[\cite{Zarikian:Unique_pseudo}*{Theorem 1.3.1}]
  \label{prop:faithful_pseudo_expectations}
  A \(\Cst\)\nb-subalgebra~\(A\) detects ideals in all intermediate
  \(\Cst\)\nb-algebras \(A\subseteq C \subseteq B\) if and only if all
  pseudo-expectations for \(A\subseteq B\) are faithful.
\end{proposition}

Combining this with
\cite{Kwasniewski-Meyer:Aperiodicity_pseudo_expectations}*{Theorem~3.6}
gives the following:

\begin{proposition}
  \label{prop:aperiodic_implies_supports}
  Let \(A\subseteq B\) be aperiodic.
  Then it admits a unique pseudo-expectation \(E\colon B\to I(A)\) and
  the following conditions are equivalent:
  \begin{enumerate}
  \item \label{enu:aperiodic_implies_supports1}%
    \(E\) is faithful;
  \item \label{enu:aperiodic_implies_supports2}%
    \(A\) supports all intermediate \(\Cst\)\nb-algebras \(A\subseteq
    C \subseteq B\);
  \item \label{enu:aperiodic_implies_supports3}%
    \(A\) detects ideals in all intermediate \(\Cst\)\nb-algebras
    \(A\subseteq C \subseteq B\).
  \end{enumerate}
  In particular, \(A\subseteq B\) is \(\Cst\)\nb-irreducible if and
  only if~\(A\) is simple and \(E\colon B\to I(A)\) is faithful.
\end{proposition}

\begin{proof}
  By
  \cite{Kwasniewski-Meyer:Aperiodicity_pseudo_expectations}*{Theorem~3.6},
  there is exactly one pseudo-expectation \(E\colon B\to I(A)\) and
  if~\(E\) is faithful, then~\(A\) supports~\(B\).
  Let~\(C\) be an intermediate \(\Cst\)\nb-algebra.
  Since \(A\subseteq B\) is aperiodic, so is \(A\subseteq C\), and hence
  the restriction~\(E|_C\) is the unique pseudo-expectation for
  \(A\subseteq C\).
  Since~\(E\) is faithful, so is~\(E|_C\), and
  so~\ref{enu:aperiodic_implies_supports1}
  implies~\ref{enu:aperiodic_implies_supports2}.
  It is clear that~\ref{enu:aperiodic_implies_supports2}
  implies~\ref{enu:aperiodic_implies_supports3},
  and~\ref{enu:aperiodic_implies_supports3}
  implies~\ref{enu:aperiodic_implies_supports1} by Proposition
  \ref{prop:faithful_pseudo_expectations}.
\end{proof}

\begin{example}
  \label{exm:twisted_groupoid}
  Let \((\Gr, \L)\)  be a  twisted \'etale groupoid with locally
  compact Hausdorff unit space~\(X\), and a twist given by a Fell line
  bundle~\(\L\).  Then \(\Cont_0(X)\) is naturally a \(\Cst\)\nb-subalgebra of the
   associated reduced \(\Cst\)\nb-algebra \(\Cst_\red(\Gr,
   \L)\).
    Let \(\Gr(x)\) denote the isotropy group of a point \(x\in X\).
  The  \(\Cst\)\nb-inclusion \(\Cont_0(X)\subseteq \Cst_{\red}(\Gr,
  \L)\) is aperiodic if and only if~\(\Gr\) is \emph{topologically
    free}, that is, for every open bisection \(U\subseteq \Gr\setminus
  X\), the set \(\setgiven{x\in X}{\Gr(x)\cap U\neq\emptyset}\) has
  empty interior in~\(X\), see \cite{Kwasniewski-Meyer:Essential}*{Theorem 7.24 and Proposition 5.15}.
  The injective envelope of \(\Cont_0(X)\) can be realised as the
  quotient \(I(\Cont_0(X))=\mathcal{B}(X)/\mathcal{M}(X)\) of
  the \(\Cst\)\nb-algebra \(\mathcal{B}(X)\) of all bounded Borel
  functions by the ideal \(\mathcal{M}(X)\) of all bounded Borel
  functions with meagre support.
  In particular, there is always a pseudo-expectation given by
  the restriction map \(f\mapsto f|_X\) composed with the quotient map
  \[
    E\colon \Cst_{\red}(\Gr, \L)\to \mathcal{B}(X)/\mathcal{M}(X)
    = I(\Cont_0(X)).
  \]
  It  is a genuine conditional expectation if and only if~\(\Gr\) is
  Hausdorff (equivalently, \(X\) is closed in~\(\Gr\)).
  As a result, if~\(\Gr\) is non-Hausdorff and topologically free,
  then the inclusion \(\Cont_0(X)\subseteq \Cst_\red(\Gr, \L)\) does
  not admit a genuine conditional expectation.
  In general, \(E\) is symmetric, so that it descends to a faithful
  pseudo-expectation on the quotient \(\Cst_\ess(\Gr,\L)\defeq
  \Cst_\red(\Gr, \L)/\mathcal{N}\), which is called the
  \emph{essential groupoid algebra}.
  In accordance with
  Proposition~\ref{prop:aperiodic_implies_supports}, see also Theorem
  \ref{thm:topologically_graded_aperiodic} below, the following
  conditions are equivalent:
  \begin{enumerate}
  \item \(\Gr\) is topologically free;
  \item \(\Cont_0(X)\) supports all intermediate \(\Cst\)\nb-algebras
    in \(\Cst_\ess(\Gr, \L)\);
  \item \(\Cont_0(X)\) detects ideals in all intermediate
    \(\Cst\)\nb-algebras in \(\Cst_\ess(\Gr, \L)\).
  \end{enumerate}
  In particular,  \(\Cont_0(X)\) detects ideals in all intermediate
  \(\Cst\)\nb-subalgebras in \(\Cst_\red(\Gr, \L)\) if the above
  equivalent conditions hold and \(\mathcal{N}=\{0\}\).
\end{example}

Another related property was introduced by Nagy and Reznikoff
in~\cite{Nagy-Reznikoff:Pseudo-diagonals}.
It weakens the \emph{extension property} introduced by Anderson
in~\cite{Anderson:Extensions_states}*{3.3}.

\begin{definition}[\cite{Nagy-Reznikoff:Pseudo-diagonals}*{p.~265}]
  \label{def:almost_extension_property}
  A \(\Cst\)\nb-inclusion \(A\subseteq B\) has the \emph{almost
    extension property} if the set of all pure states on~\(A\) that
  extend uniquely to a state on~\(B\) is weak-\(\star\)-dense in the
  set of all pure states on~\(A\).
\end{definition}

\begin{proposition}[\cite{Kwasniewski-Meyer:Aperiodicity_pseudo_expectations}*{Theorem~5.5}]
  \label{prop:almost_extension_aperiodic}
  A \(\Cst\)\nb-inclusion \(A\subseteq B\)  with the almost extension
  property is aperiodic.
  The two properties are equivalent when~\(B\) is separable.
\end{proposition}

\begin{example}
  The \(\Cst\)\nb-inclusion \(\Cont_0(X)\subseteq \Cst_\red(\Gr, \L)\)
  from Example~\ref{exm:twisted_groupoid} has the almost extension
  property if and only if~\(\Gr\) is \emph{topologically principal},
  that is, the set of points with nontrivial isotropy \(\setgiven{x\in
    X}{\Gr(x)\neq \{x\}}\) has empty interior in~\(X\).
  By the Baire category theorem, if $\Gr$ has a countable
  cover by open bisections, then it is topologically free if and only if it is
  topologically principal.
  This fails, however, without the countability assumption.
  Hence the almost extension property is strictly stronger than
  aperiodicity.
\end{example}

\section{Topologically graded regular C*-inclusions and Fell bundles}
\label{sec:Topologically graded regular C*-inclusions}

Adopting our convention from \cites{Kwasniewski-Meyer:Stone_duality,
  Kwasniewski-Meyer:Pure_infiniteness, Kwasniewski-Meyer:Cartan} we
call a \(\Cst\)\nb-inclusion \(A\subseteq B\) \emph{regular} if it is
both nondegenerate and the set of normalisers of~\(A\) generates~\(B\)
as a \(\Cst\)\nb-algebra, see~\cite{Renault:Cartan.Subalgebras}.
Here a \emph{normaliser} of~\(A\) is an element \(b\in B\) such that
\(b A b^*\subseteq A\) and \(b^* A b\subseteq A\), see
\cite{Kumjian:Diagonals}.
With this convention, a \(\Cst\)\nb-inclusion \(A\subseteq B\) is
regular if and only if~\(B\)  has an inverse-semigroup grading with
the unit fibre~\(A\), see
\cite{Kwasniewski-Meyer:Stone_duality}*{Corollary~6.27} or
\cite{Kwasniewski-Meyer:Pure_infiniteness}*{Proposition~2.11}.
To give more details,
let us fix a  nondegenerate \(\Cst\)\nb-inclusion \(A\subseteq B\).
A closed \(A\)\nb-subbimodule $M$ of~\(B\) that consists entirely
of normalisers is called a \emph{slice} for \(A\subseteq B\), see~\cite{Exel:noncomm.cartan}.
The set \(\Slice_A(B)\) of slices with the operations
\[
  M\cdot N\defeq \clsp {}\setgiven{m n}{m \in
   M,\ n\in N}\quad\text{  and } \quad M^*\defeq \setgiven{m^*}{m \in M}
\]
is an inverse semigroup with unit \(A \in \Slice_A(B)\), see \cite{Kwasniewski-Meyer:Stone_duality}*{Proposition~6.26}.
If \(M\in \Slice_A(B)\), then \(A M \subseteq M\), \(M A \subseteq
M\), \(MM^*\subseteq A\), and \(M^*M \subseteq A\), so
that~\(M\) becomes a concrete Hilbert \(A\)\nb-bimodule.
The idempotent elements in \(\Slice_A(B)\) are precisely the ideals
in~\(A\).
The canonical partial order in the inverse semigroup \(\Slice_A(B)\)
is just inclusion.
Every normaliser \(b\in B\) generates a  slice \(\overline{AbA}\).
Thus  \(A\subseteq B\) is regular if only if \(B=\clsp\bigcup\Slice_A(B)\).
More generally, we call a unital inverse subsemigroup \(S\subseteq
\Slice_A(B)\) with \(B=\clsp\bigcup S\) an \emph{inverse-semigroup
  grading} of~\(B\), see
\cite{Kwasniewski-Meyer:Stone_duality}*{Definition~6.15}.
We call the grading \emph{wide} if \(\bigcup\setgiven{N\in S}{N\le A,
  M \text{ in }S}\) is linearly dense in \(M\cap A\) for every \(M\in
S\), see \cite{Kwasniewski-Meyer:Cartan}*{Definition~3.1}.
If~\(S\) is closed under intersections it is automatically wide.

We will now introduce topological inverse-semigroup gradings, which
generalise the topological gradings by groups defined by Exel,
see~\cite{Exel:Partial_dynamical}.
To this end, we need some notation.
If \(M\in \Slice_A(B)\), then the intersection \(M\cap A\) is an ideal
in~\(A\).
Let
\[
  (M\cap A)^\bot
  \defeq \setgiven{a\in A}{a (M\cap A)=0}
  = \setgiven{a\in A}{ (M\cap A)a=0}
\]
be its annihilator in \(A\).

\begin{lemma}
  \label{lem:slices_annihilators}
  Let  \(A\subseteq B\) be a  nondegenerate \(\Cst\)\nb-inclusion.
  If \(M\in\Slice_A(B)\), then \(M (M\cap A) = M\cap A = (M\cap A)M\),
  and
  \[
    M (M\cap A)^\bot
    = \setgiven{m\in M}{m(M\cap A)=0}
    = \setgiven{m\in M}{(M\cap A)m=0}
    = (M\cap A)^\bot M
  \]
  is an \(A\)\nb-subbimodule of~\(M\), which we will denote by~\(M^\bot\).
\end{lemma}

\begin{proof}
  Since  \(M\cap A\subseteq M\) is self-adjoint, it is contained in
  \(M\cap M^*\).
  Thus, if \(a,b\in M\cap A\) and \(m\in M\), then
  \(a b m\in MM^*M\cap AM^*M\subseteq M\cap A\).
  This implies \((M\cap A)M=M\cap A\).
  Similarly, \(M(M\cap A)=M\cap A\), and this shows the first part.

  The inclusion \(M(M\cap A)^\bot\subseteq \setgiven{m\in M}{m(M\cap
    A)=0}\) is clear.
  We prove the converse inclusion by taking \(m\in M\setminus M\cdot
  (M\cap A)^\bot\) and finding \(a\in M\cap A\) with \(m\cdot a \neq0\).
  We may write \(m=nn^*n\) for some \(n\in M\) by
  \cite{Blanchard:Deformations_Hopf}*{Lemme~1.3}.
  Since \(m\notin M\cdot  (M\cap A)^\bot\) by assumption,
  it follows that \(n^*n\notin (M\cap A)^\bot\).
  So there is \(a \in M\cap A\) with \(n^*n a\neq 0\).
  Then \(m a = n n^*n a\neq 0\) as desired.
  This shows the first equality in the displayed formula.
  The last one is shown similarly.
  The first part of the assertion implies \(m(M\cap A), (M\cap A) m
  \subseteq M\cap A\) for any \(m\in M\).
  Hence both \(m(M\cap A)\neq 0\) and \((M\cap A)m \neq 0\)
  are equivalent to \((M\cap A)m(M\cap A)\neq 0\).
  This shows the middle equality.
\end{proof}


\begin{definition}
  \label{def:topological_grading}
  We call a regular \(\Cst\)\nb-inclusion \(A\subseteq B\)
  \emph{topologically graded} if there are a wide inverse-semigroup
  grading \(S\subseteq \Slice_A(B)\) and a symmetric
  pseudo-expectation \(E\colon B\to I(A)\) such that \(E(M^\bot)=0\)
  for every \(M\in S\).
  If such an~\(E\) exists, we call it the  \emph{\(S\)\nb-canonical
    expectation}.  (We show below that it is unique.)
\end{definition}

Let~\(A\) be a \(\Cst\)\nb-algebra.
The \emph{local multiplier algebra}~\(\Locmult(A)\) of~\(A\) is the
inductive limit \(\Cst\)\nb-algebra of the inductive system of
mulitplier algebras \(\Mult(I)\) indexed by the set of essential
ideals~\(I\) in~\(A\) and directed by the relation~\(\supseteq\)
(here we use the canonical unital \Star{}momomorphisms \(\Mult(J) \to
\Mult(I)\) for two essential ideals \(I\subseteq J \subseteq A\)).
The local multiplier algebra embeds canonically into the injective
hull~\(I(A)\), see~\cite{Frank:Injective_local_multiplier}.

\begin{lemma}
  \label{lem:genuine_expectation}
  A \(\Cst\)\nb-inclusion \(A\subseteq B\) with an inverse-semigroup
  grading \(S\subseteq \Slice_A(B)\) admits at most one
  \(S\)\nb-canonical pseudo-expectation~\(E\).
  If it exists, then \(E\colon B\to \Locmult(A)\subseteq I(A)\) takes
  values in the local multiplier algebra.
  Moreover, \(E\) is a genuine expectation that preserves the grading
  in the sense that \(E(M)\subseteq M\) for all \(M\in S\) if and only
  if \((M\cap A)\oplus M^\bot=M\) for all \(M\in S\).
\end{lemma}

\begin{proof}
  The canonical embedding \(A\hookrightarrow I(A)\) extends to an
  embedding \(\Locmult(A) \hookrightarrow I(A)\) by
  \cite{Frank:Injective_local_multiplier}*{Theorem~1}.
  Let \(M\in S\).
  Then \(J\defeq (M\cap A)\oplus(M\cap A)^\bot\) is an essential ideal
  in~\(A\) and \(M \cdot (M\cap A)= M\cap A\) by
  Lemma~\ref{lem:slices_annihilators}.
  If \(\xi\in M\) and  \(a=a_0\oplus a_1\in (M\cap A)\oplus(M\cap
  A)^\bot=J\), then \(E(\xi) a=E(\xi a)=E(\xi a_0)=  \xi a_0 \in J\)
  because~\(E\) is an \(A\)\nb-bimodule map and it vanishes
  on~\(M^\bot\).
  Similarly, \(a E(\xi) =a_0\xi\).
  Hence, \(E(\xi)\) is a multiplier of~\(J\).
  Thus \(E(M)\subseteq \Mult(J)\), and the above formulas
  determine~\(E\) on~\(M\).
  Since \(B=\clsp\bigcup S\), this uniquely determines~\(E\) and shows
  that it takes values in~\(\Locmult(A)\).

  If \(M=(M\cap A)\oplus M^\bot\) for all slices~\(M\), then~\(E\) just
  projects each slice~\(M\) onto the first summand \(M\cap A\), which
  is contained in~\(A\).
  So~\(E\) is \(A\)\nb-valued and \(E(M)\subseteq M\cap A\).
  Conversely, assume that~\(E\) maps~\(M\) to \(M\cap A\).
  If \(\xi \in M\), then \(\xi=E(\xi) + \bigl(\xi-E(\xi)\bigr) \in
  (M\cap A)\oplus M^\bot\) because \(\bigl(\xi-E(\xi)\bigr) a=\xi a
  -E(\xi a)=\xi a -\xi a=0\) for any \(a\in M\cap A\).
  This implies the last part of the assertion.
\end{proof}

Let  \(A\subseteq B\) be a regular \(\Cst\)\nb-inclusion.
For any inverse-subsemigroup grading \(S\subseteq \Slice_A(B)\), the
family \(\B\defeq (M)_{M\in S}\) together with the multiplication and
involution inherited from~\(B\) is an \emph{action of~\(S\) on~\(A\)
  by Hilbert bimodules} in the sense
of~\cite{Buss-Meyer:Actions_groupoids}.
This is also equivalent to a \emph{saturated Fell bundle} over~\(S\)
as defined in~\cite{Exel:noncomm.cartan}.
Every saturated Fell bundle \(\B=(B_{s})_{s\in S}\) over a unital
inverse semigroup arises in this way, as its fibres embed naturally
into the \emph{full section \(\Cst\)\nb-algebra} \(\Cst(\B)\)
defined in~\cite{Exel:noncomm.cartan}.
The algebra \(\Cst(\B)\) coincides with the
\emph{crossed product} \(A\rtimes_{\B} S\) for an \(S\)\nb-action by
Hilbert \(A\)\nb-bimodules defined in~\cite{Buss-Exel-Meyer:Reduced}.
A representation of~\(\B\) in a \(\Cst\)\nb-algebra~\(B\) is a
collection of linear maps \(\pi_t\colon B_t\to B\) for \(t\in S\) that
turn the involution and multiplication in~\(\B\) into that in~\(B\).
The \(\Cst\)\nb-algebra \(\Cst(\B)\) is universal for such
representations, see
\cite{Kwasniewski-Meyer:Aperiodicity_pseudo_expectations}*{Remark~3.14}.
As we showed in \cite{Kwasniewski-Meyer:Essential}*{Proposition~4.3},
there is a well-defined \(S\)\nb-canonical pseudo-expectation
\(E\colon \Cst(\B)\to  \Locmult(A)\subseteq I(A)\), which then
descends to a faithful pseudo-expectation \(E_\ess\colon
\Cst_\ess(\B)\to  \Locmult(A)\subseteq I(A)\) on the quotient
\(\Cst\)\nb-algebra \(\Cst_\ess(\B)\defeq \Cst(\B)/\Null_E\).
The latter is called the \emph{essential section
  \(\Cst\)\nb-algebra} or \emph{essential crossed product}, see
\cite{Kwasniewski-Meyer:Essential}*{Definition 4.4}.
In~\cite{Kwasniewski-Meyer:Essential}, we called~\(B\) an \emph{exotic
section \(\Cst\)\nb-algebra} of~\(\B\) if it comes equipped with
\Star{}homomorphisms
\[
\Cst(\B) \onto B\onto \Cst_\ess(\B)
\]
that compose to the quotient map \(\Cst(\B)\onto \Cst_\ess(\B)\).
Following Exel \cite{Exel:Partial_dynamical}*{Definition 19.2} we
called such \(\Cst\)\nb-algebras topologically graded by~\(\B\) in
\cite{Kwasniewski-Meyer:Essential}*{Subsection~4.2}.
This agrees with Definition \ref{def:topological_grading}:

\begin{proposition}
  \label{prop:topological_grading_characterisation}
  A \(\Cst\)\nb-inclusion \(A\subseteq B\) is topologically graded if
  and only if~\(B\) is an exotic section \(\Cst\)\nb-algebra for some
  inverse-semigroup action \(\B=(B_{s})_{s\in S}\)  by Hilbert
  bimodules with unit fibre \(B_1=A\).
  If this holds, then \(B\cong\Cst_\ess(\B)\) if and only if  the
  canonical pseudo-expectation for \(A\subseteq B\) is faithful.
\end{proposition}

\begin{proof}
  Let~\(B\) be an exotic section \(\Cst\)\nb-algebra for a saturated
  Fell bundle over a unital inverse semigroup \(\B=(B_{s})_{s\in S}\).
  Let~\(\Phi\) be the quotient map \(\Cst(\B) \onto B\) and
  let~\(E_B\) be the composite of the quotient map \(B\onto
  \Cst_\ess(\B)\) and the faithful pseudo-expectation \(E_\ess\colon
  \Cst_\ess(\B)\to I(A)\).
  We claim that \(E_B\circ \Phi\) is the canonical expectation
  \(E\colon \Cst(\B)\onto A\).
  If \(t\in S\), then \(B_t\cap A\subseteq \Phi(B_t)\cap A\) because
  \(\Phi|_A=\Id_A\).
  Thus \((\Phi(B_t)\cap A)^\bot \subseteq (B_t\cap A)^\bot\) and so
  \[
    E_B(\Phi(B_t)(\Phi(B_t)\cap A)^\bot)
    \subseteq E_B\bigl(\Phi(B_t(B_t\cap A)^\bot)\bigr)
    = E(B_t(B_t\cap A)^\bot)=0.
  \]
  Accordingly, \(E_B\circ\Phi\) is the canonical expectation on
  \(\Cst(\B)\).
  Since~\(E_B\) descends to a faithful map, it is symmetric.
  Hence \(A\subseteq B\) is topologically graded.

  Now let \(A\subseteq B\) be any topologically graded regular
  inclusion.
  In particular, this defines a surjective \Star{}homomorphism
  \(\Phi\colon\Cst(\B) \onto B\).
  Being topologically graded means that the canonical expectation
  \(\Cst(\B)\to I(A)\) descends to a pseudo-expectation on~\(B\).
  The latter must be the unique canonical pseudo-expectation by
  Lemma~\ref{lem:genuine_expectation}.
  Then \(\Phi(\Null_E)\subseteq \Null_{E_B}\) and there is a
  well-defined surjective \Star{}homomorphism
  \[
    \Cst_\ess(\B)= \Cst(\B)/\Null_E \to B/\Null_{E_B},
    \qquad
    b+\Null_E\mapsto \Phi(b)+\Null_{E_B}.
  \]
  This map is injective as well as it intertwines the
  pseudo-expectations \(\Cst(\B)/\Null_E\to I(A)\) and
  \(B/\Null_{E_B}\to I(A)\), which are faithful because both  \(E\)
  and~\(E_B\) are symmetric.
  Composing the quotient map \(B\onto B/\Null_{E_B}\) with
  \(B/\Null_{E_B}\cong \Cst_\ess(\B)\) gives the surjective
  \Star{}homomorphism \(B\onto \Cst_\ess(\B)\) that witnesses
  that~\(B\) is an exotic section \(\Cst\)\nb-algebra of~\(\B\).
\end{proof}


The \emph{reduced \(\Cst\)\nb-algebra} \(\Cst_\red(\B)\) defined
in~\cite{Buss-Exel-Meyer:Reduced} is also graded by \(\B=(B_s)_{s\in
  S}\) and equipped with a faithful generalised expectation
\(E_\red\colon\Cst_\red(\B)\to A''\) with values in the bidual
of~\(A\).
The quotient map to \(\Cst(\B) \to \Cst_\ess(\B)\) factors through
\(\Cst_\red(\B)\), that is, the inclusion \(A\subseteq \Cst_\red(\B)\)
is topologically graded.
If the generalised expectation~\(E_\red\) is a genuine expectation,
then the Fell bundle or inverse-semigroup action \(\B=(B_s)_{s\in S}\)
is called \emph{closed}
(this may be characterised in many ways, see
\cite{Buss-Exel-Meyer:Reduced}*{Proposition~6.3} and
\cite{Kwasniewski-Meyer:Essential}*{Proposition~3.20}).
Then \(E_\red\colon\Cst_\red(\B)\to A\subseteq I(A)\) is already the
canonical expectation, and since it is faithful,
\(\Cst_\red(\B)=\Cst_\ess(\B)\).
In general, the surjection \(\Cst_\red(\B)\onto \Cst_\ess(\B)\) may
have a kernel, which is called the \emph{singular ideal}.

\begin{example}[Fell bundles over groupoids]
  For any upper-semicontinous Fell bundle \(\A=(A_\gamma)_{\gamma\in
    \Gr}\) over an \'etale groupoid~\(\Gr\), as defined
  in~\cite{BussExel:Fell.Bundle.and.Twisted.Groupoids}, the
  subbimodules of spaces of sections of~\(\A\) on open bisections
  of~\(\Gr\) form a natural saturated inverse-semigroup Fell bundle
  \(\B=(B_s)_{s\in S}\), see
  \cite{Kwasniewski-Meyer:Essential}*{Lemma~7.3}.
  Its unit fibre is the algebra \(A\defeq\Cont_0(\A|_X)\) of sections
  on the unit space~\(X\) of~\(\Gr\).
  The full, reduced and essential algebras for \(\A\) coincide with
  those for~\(\B\), respectively, see
  \cite{Kwasniewski-Meyer:Essential}*{Propositions 7.6, 7.9 and
    Definition~7.12}.
  Moreover, if~\(\Gr\) is Hausdorff, then \(\Cst_\red(\B) =
  \Cst_\ess(\B)\) and the canonical pseudo-expectation is a faithful
  genuine conditional expectation \(E\colon \Cst_\red(\B)\to
  A\subseteq \Cst_\red(\B)\).
\end{example}

Groups are special cases of inverse semigroups and groupoids, and for
groups the various notions of Fell bundle coincide with the classical
one, see~\cite{Exel:Partial_dynamical}.

\begin{example}[Fell bundles over groups]
  \label{ex:group_Fell_bundles}
  Let \(\B=(B_{g})_{g\in G}\) be a (not necessarily saturated) Fell
  bundle over a group~\(G\).
  Recall that the bundle structure turns each fibre~\(B_g\) into a
  Hilbert bimodule over the unit fibre \(\Cst\)\nb-algebra \(A\defeq B_1\).
  We call \(\B\) \emph{aperiodic}, \emph{topologically free},
  \emph{properly outer} or \emph{outer} if all the Hilbert
  \(A\)\nb-bimodules~\(B_g\) for \(g\in G\setminus\{1\}\) have the
  respective property.
  By definition, the implications in Figure~\ref{fig:diagram} remain
  valid for the Fell bundle~\(\B\).
  Let \(E\colon \bigoplus_{g\in G} B_g \to A\) be the identity
  on~\(A\) and zero on~\(B_g\) for all \(g\in G\setminus \{1\}\).
  An exotic section \(\Cst\)\nb-algebra for~\(\B\) is a completion of
  the \Star{}algebra \(\bigoplus_{g\in G} B_g\) in a \(\Cst\)\nb-norm
  that makes~\(E\) bounded.
  By  \cite{Kwasniewski-Meyer:Essential}*{Proposition~6.3}, for any
  exotic section \(\Cst\)\nb-algebra~\(B\), the inclusion \(A\subseteq
  B\) is
  aperiodic if and only if the Fell bundle~\(\B\) is aperiodic.
  If~\(G\) is amenable or, more generally, \(\B\) has Exel's
  approximation property, then \(\Cst(\B)=\Cst_\red(\B)\) and so there
  is no truly exotic section \(\Cst\)\nb-algebra for~\(\B\).
\end{example}

The following theorem generalises
\cite{Kwasniewski-Meyer:Aperiodicity}*{Theorem 9.12} and
\cite{Kwasniewski-Meyer:Aperiodicity_pseudo_expectations}*{Theorem
  7.2} by removing the assumption that the algebra~\(A\) contains an
essential ideal which is separable or of Type~I.
This generalisation relies on a recent result of Geffen and
Ursu~\cite{Geffen-Ursu:Simple_crossed_products}.

\begin{theorem}
  \label{the:special_groups}
  Let \(\B=(B_{g})_{g\in G}\) be a \textup{(}not necessarily
  saturated\textup{)} Fell bundle over a group of the form \(G=\Z\) or
  \(G=\Z/n\) for a square-free number \(n>0\).
  Then \(A\defeq B_0\) detects ideals in  \(B\defeq \Cst(\B)\) if and
  only if~\(A\) detects ideals in all intermediate
  \(\Cst\)\nb-algebras of \(A\subseteq C\subseteq B\), if and only if
  \(A\subseteq B\) is aperiodic.
  In particular, \(A\subseteq B\) is \(\Cst\)\nb-irreducible if and
  only if both \(A\) and~\(B\) are simple.
\end{theorem}

\begin{proof}
  Since~\(G\) is amenable, \(\Cst(\B) = \Cst_\red(\B)\) and the
  canonical expectation \(\Cst(\B)\to A\) is faithful.
  If the inclusion is aperiodic, then~\(A\) detects ideals in all
  intermediate \(\Cst\)\nb-algebras by
  Proposition~\ref{prop:aperiodic_implies_supports}.
  It remains to prove that the inclusion must be aperiodic if~\(A\)
  detects ideals in \(\Cst(\B)\).
  For group actions by automorphisms, this follows from
  \cite{Geffen-Ursu:Simple_crossed_products}*{Corollary~8.6}.
  We reduce our statement to this special case using the Morita
  globalisation of~\(\B\) from
  \cite{Kwasniewski-Meyer:Aperiodicity}*{Proposition~7.1}.
  This is a group action \(\gamma\colon G\to \Aut(C)\), and it has the
  property that detection of ideals and aperiodicity for \(A\subseteq
  B\) are equivalent to the corresponding properties of \(C\subseteq
  C\rtimes_\gamma^\red G\), see
  \cite{Kwasniewski-Meyer:Aperiodicity}*{Proposition 7.1 (2)--(3) and
    Proposition 6.8 (1)}.
  This proves the first part of the assertion.
  The second part now follows because any subalgebra~\(A\) of a simple
  algebra~\(B\) detects ideals in it, and if a nondegenerate simple
  subalgebra~\(A\) detects ideals in~\(C\), then~\(C\) has to be simple.
\end{proof}

\begin{remark}
  The theorem above fails for actions of more general abelian groups,
  see \cite{Geffen-Ursu:Simple_crossed_products}*{Proposition~8.7}.
\end{remark}

Our next goal is to apply Theorem~\ref{the:special_groups} to
general regular inclusions.
We first formulate some results of Hamana in a lemma, see, in
particular, \cite{Hamana:injective-C-dyn}*{Remark~7.5}:

\begin{lemma}
  \label{lem:inductively_not_aperiodic}
  For any \Star{}automorphism \(\alpha\colon A\to A\), there is a
  largest \(\alpha\)\nb-invariant ideal~\(I\) in~\(A\) such
  that~\(\alpha|_I\) is quasi-inner in the sense that its Borchers
  spectrum is trivial.
  Moreover,  \(\alpha|_{I^\bot}\) is properly outer, and every
  power~\((\alpha|_I)^n\) for \(n\neq 0\) is quasi-inner.
  In particular, if~\(\alpha\) is not properly outer, then no
  power~\(\alpha^n\) for \(n\neq 0\) is properly outer.
\end{lemma}

\begin{proof}
  The automorphism~\(\alpha\) extends uniquely to an automorphism
  \(I(\alpha)\) of Hamana's injective envelope \(I(A)\) of~\(A\).
  By \cite{Hamana:injective-C-dyn}*{Theorem~7.4}, \(\alpha\) is
  quasi-inner if and only if \(I(\alpha)\) is inner.
  Since \(I(\alpha^n)=I(\alpha)^n\) for \(n>0\) and powers of inner
  automorphisms are clearly inner, it follows that powers of
  quasi-inner automorphisms are quasi-inner.
  By \cite{Hamana:tensorI}*{Proposition~5.1}, there is a largest
  projection \(p\in I(A)\) such that  \(I(\alpha)|_{pI(A)p}\) is
  inner, and this~\(p\) is central, \(I(\alpha)\)-invariant and
  \(I(\alpha)|_{(1-p)I(A)}\) is purely outer.
  Therefore, \(I\defeq A\cap pA\) is the largest ideal such
  that~\(\alpha|_I\) is quasi-inner.
  Since \(I^\bot=A\cap (1-p)A\), it follows that~\(\alpha_{I^\bot}\)
  is properly outer.
\end{proof}

\begin{corollary}
  \label{cor:inductive_not_aperiodic}
  If a Hilbert \(A\)\nb-bimodule~\(M\) is not aperiodic, then no
  power~\(M^{\otimes_A n}\) for \(n>0\) is aperiodic.
\end{corollary}

\begin{proof}
  The powers of~\(M\) form a Fell bundle~\(\B\) over~\(\Z\),
  see~\cite{Abadie-Eilers-Exel:Morita_bimodules}.
  The Morita globalisation of~\(\mathcal{B}\) from
  \cite{Kwasniewski-Meyer:Aperiodicity}*{Proposition~7.1} is a
  \(\Z\)\nb-action generated by an automorphism \(\alpha\colon C\to
  C\).
  By  \cite{Kwasniewski-Meyer:Aperiodicity}*{Proposition 7.1(2) and
    Proposition 6.8(1)},
  the bimodule~\(M^{\otimes_A n}\) is aperiodic if and only
  if~\(\alpha^n\) is properly outer.
  Thus the assertion follows from the last part of
  Lemma~\ref{lem:inductively_not_aperiodic}.
\end{proof}

\begin{lemma}
  \label{lem:hereditarily_aperiodic}
  If a Hilbert \(A\)\nb-bimodule~\(M\) is not aperiodic, then it
  contains a nonzero subbimodule~\(N\) such that no nonzero
  subbimodule of~\(N\) is aperiodic.
\end{lemma}

\begin{proof}
  By assumption, there are \(b\in M\), \(D\in \Her(A)\) and
  \(\varepsilon >0\) such that \(\norm{a b a} \ge \varepsilon
  \norm{a}^2\) for all \(a\in D^+\).
  Then \(N_0\defeq\overline{D b D}\) is a Hilbert \(D\)\nb-bimodule.
  By the Rieffel correspondence, each nonzero subbimodule of~\(N_0\)
  is of the form~\(N_0I\) for a nonzero ideal~\(I\) in~\(D\).
  We fix such \(I\neq 0\).
  By \cite{Kwasniewski-Meyer:Aperiodicity}*{Lemma~2.9},
  there is \(d\in I^+_1\) such that \(D_0\defeq \setgiven{x\in A}{d x=
    x = x d}\) is a nonzero hereditary subalgebra of~\(I\).
  Then the element \(d b d\in N_0 I\) is not aperiodic because if
  \(x\in D_0^+\), then \(\norm{x dbdx}=\norm{xbx}\ge \varepsilon
  \norm{x}^2\).
  Thus no nonzero subbimodule of~\(N_0\) is aperiodic.
  Let \(\Hilm\defeq AD\), this is an equivalence Hilbert
  \(ADA\)-\(D\)-bimodule.
  It establishes an equivalence between the Hilbert bimodules~\(N_0\)
  over~\(D\) and \(N\defeq \overline{A D b D A}\cong \Hilm \otimes_D
  N_0 \otimes_D \Hilm^*\) over \(A D A\).
  When we transport Hilbert bimodules along an equivalence, then
  aperiodicity is preserved, see
  \cite{Kwasniewski-Meyer:Aperiodicity}*{Proposition 6.1}.
  Thus~\(N\) cannot be aperiodic.
\end{proof}

The following theorem should be compared with
\cite{Kwasniewski-Meyer:Aperiodicity_pseudo_expectations}*{Theorems
  7.2 and~7.3}, where separability, Type~I, or simplicity assumptions
were imposed.
In contrast, no such assumptions are required here.

\begin{theorem}
  \label{thm:topologically_graded_aperiodic}
  Let \(A\subseteq B\) be a regular topologically graded
  \(\Cst\)\nb-inclusion and let \(B_\ess=B/\Null_E\) be the associated
  essential quotient.
  Equivalently, \(B\) is an exotic section \(\Cst\)\nb-algebra for a
  Fell bundle \(\B=(B_{s})_{s\in S}\) and \(B_\ess=\Cst_\ess(\B)\) is
  the essential \(\Cst\)\nb-algebra for \(\B\).
  The following conditions are equivalent:
  \begin{enumerate}
  \item \label{cor:topologically_graded_aperiodic1}%
    \(A\subseteq B\)  is aperiodic;
  \item \label{cor:topologically_graded_aperiodic2}%
    \(A\subseteq B\) has a unique pseudo-expectation;
  \item \label{cor:topologically_graded_aperiodic3}%
    \(A\) supports all intermediate \(\Cst\)\nb-algebras \(A\subseteq
    C\subseteq B_\ess\);
  \item \label{cor:topologically_graded_aperiodic4}%
    \(A\) detects  ideals in all intermediate \(\Cst\)\nb-algebras
    \(A\subseteq C \subseteq B_\ess\).
  \end{enumerate}
  If \(A\subseteq B\) admits a conditional expectation~\(E\), then
  the above conditions imply that~\(E\) is pinching.
\end{theorem}

\begin{proof}
  Let \(S\subseteq \Slice_A(B)\) be the grading which admits a
  canonical pseudo-expectation \(E\colon B\to I(A)\).
  The latter descends to a faithful pseudo-expectation \(E_\ess\colon
  B_\ess\to I(A)\).
  If \(A\subseteq B\) is aperiodic, then so is \(A\subseteq B_\ess\)
  because aperiodicity for bimodules is hereditary for quotients.
  Hence~\ref{cor:topologically_graded_aperiodic1}
  implies~\ref{cor:topologically_graded_aperiodic2}
  and~\ref{cor:topologically_graded_aperiodic3} by
  Proposition~\ref{prop:aperiodic_implies_supports}.
  It is clear that~\ref{cor:topologically_graded_aperiodic3}
  implies~\ref{cor:topologically_graded_aperiodic4}.
  If \(A\subseteq B\) has a unique pseudo-expectation, then so
  does \(A\subseteq B_\ess\).
  Hence~\ref{cor:topologically_graded_aperiodic2}
  implies~\ref{cor:topologically_graded_aperiodic4} by
  Proposition~\ref{prop:faithful_pseudo_expectations}.
  To close the cycle of implications, it remains to prove
  that~\ref{enu:topologically_graded_aperiodic4}
  implies~\ref{cor:topologically_graded_aperiodic1}.
  This can be  proved in a similar way as
  \cite{Kwasniewski-Meyer:Aperiodicity_pseudo_expectations}*{Proposition~6.1},
  but using Theorem~\ref{the:special_groups} to get the desired
  conclusion without assuming separability (but for aperiodicity
  rather than topological freeness).
  For the sake of completeness and the benefit of the reader we spell
  out the details.

  Assume that \(A\subseteq B\) is not aperiodic.
  Since~\(B\) is an exotic section \(\Cst\)\nb-algebra for~\(S\),
  \cite{Kwasniewski-Meyer:Essential}*{Proposition~6.3} provides an
  \(M\in S\)  such that \(M^\bot\)  is not
  aperiodic as an \(A\)\nb-bimodule.
  By Lemma~\ref{lem:hereditarily_aperiodic}, there is a nonzero subbimodule  \(N\subseteq M^\bot\)
	where no nonzero subbimodule of~\(N\) is aperiodic.
  Let \(\Hilm_0 \defeq A\), \(\Hilm_k \defeq N^k\) and \(\Hilm_{-k}
  \defeq (N^*)^k\) for \(k>0\).
  These are slices for the inclusion \(A\subseteq B\),
  and they form a Fell bundle over~\(\Z\) with the multiplication
  and involution in~\(B\).
  The inclusion maps \(\Hilm_k \hookrightarrow B\) form an
  isometric Fell bundle representation, and so they induce a
  \Star{}homomorphism \(\Phi\colon \Cst((\Hilm_k)_{k\in \Z}) \to
  B\).
  The composition with the quotient map \(B\to B_\ess\) remains
  injective on~\(A\) and hence on~\(\Hilm_k\) for all \(k\in \Z\).
  Now we distinguish two cases.

  Assume first that \(\Hilm_k\subseteq (M^k)^\bot\) for all \(k\in \Z\setminus \{0\}\).
  Then~\(E_\ess(\Phi(\Hilm_k))=E(\Hilm_k)=0\)  for all \(k\in \Z\setminus \{0\}\), and so \(\Phi\) intertwines the two canonical expectations.
  Since both are faithful, \(\Phi\) is injective.
  Thus we may identify  \(C\defeq \Cst((\Hilm_k)_{k\in \Z})\)
  with an intermediate \(\Cst\)\nb-algebra~\(C\) for \(A\subseteq
  B_\ess\).
  Since \(\Hilm_1=N\) is not aperiodic, the inclusion \(A\subseteq C\)
  is not aperiodic.
  Hence~\(A\) cannot detect ideals in~\(C\) by Theorem~\ref{the:special_groups}.
  This finishes the proof in this case.

  The second case is that \(\Hilm_k\not\subseteq (M^k)^\bot\) for some
  \(k\in \Z\setminus \{0\}\).
  By passing to \(\Hilm_k^*\), we may choose \(k>0\), and we
  then pick the minimal such \(k>0\).
  Since \(\Hilm_1\subseteq M^\bot\),  \(k\ge 2\).
  As \((M^k)^\bot= \setgiven{m\in M^k}{m(M^k\cap A)=0}\) the
  ideal  \(K\defeq \Hilm_k\cap A=\Hilm_k\cap (M^k\cap A)
  =\Hilm_k (M^k\cap A)\) must be nonzero, see
  Lemma~\ref{lem:slices_annihilators}.
        
  Let \(\Hilm[F]_0 \defeq A\), and
   \(\Hilm[F]_n \defeq \Hilm_nK\) for \(n=1,\dotsc,k-1\).
  Then \(\Hilm_{k-n} \Hilm[F]_n=\Hilm_{k-n}\Hilm_n K=\Hilm_k K= K\) for \(n=1,\dotsc,k-1\).
	This implies that
	\((\Hilm[F]_n)_{n=0,\dotsc,k-1}\) is a Fell bundle
  with nonzero fibres over~\(\Z/k\).
   Write  \(k=p\cdot k_1\) with a prime number~\(p\).
  Since the \(A\)\nb-bimodule \(0\neq\Hilm[F]_1\subseteq N\) is not
  aperiodic, none of the fibres~\(\Hilm[F]_n\) for \(n=1,\dotsc,k-1\)
  is aperiodic by Corollary~\ref{cor:inductive_not_aperiodic}.
  Thus the Fell bundle \((\Hilm[F]_{k_1\cdot
    n})_{n=0,\dotsc,p-1}\) over~\(\Z/p\) also fails to be aperiodic.
  The inclusions \(\Hilm[F]_n \hookrightarrow B_\ess\) form a Fell
  bundle representation, and the induced \Star{}homomorphism
  \(\Cst((\Hilm[F]_{k_1n})_{n\in \Z/p}) \hookrightarrow B_\ess \) is injective,
  as it intertwines the canonical faithful expectations.
  This identifies \(C\defeq\Cst((\Hilm[F]_n)_{n\in \Z/k})\) with an
  intermediate \(\Cst\)\nb-subalgebra \(A\subseteq C\subseteq B_\ess\).
  As the Fell bundle \((\Hilm[F]_{k_1\cdot n})_{n=0,\dotsc,p-1}\) is not aperiodic,
  the inclusion \(A\subseteq C\) fails to be aperiodic by
  \cite{Kwasniewski-Meyer:Essential}*{Proposition~6.3}.
  Since~\(p\) is square-free, Theorem~\ref{the:special_groups} says
  that~\(A\) cannot detect ideals in~\(C\).

  This finishes the proof that the statements
  \ref{cor:topologically_graded_aperiodic1}--\ref{cor:topologically_graded_aperiodic4}
  are equivalent.
  Now assume that there is a genuine conditional expectation \(E\colon
  B\to A\).
  Then~\ref{cor:topologically_graded_aperiodic1} implies that~\(E\) is
  pinching by Lemma~\ref{lem:aperiodic_implies_pinching}.
\end{proof}

\begin{remark}
  Let~\(B\) be an exotic crossed product for a group action on~\(A\),
  so that  \(A\rtimes_\alpha G\onto B\onto A\rtimes_\alpha^\red G\).
  Then~\(B\) is topologically graded with
  \(B_\ess=A\rtimes_\alpha^\red G\) and the equivalent conditions in
  Theorem~\ref{thm:topologically_graded_aperiodic} hold if and only
  if~\(\alpha_g\) is properly outer for all \(g\in G\setminus\{1\}\),
  see Examples \ref{ex:conditions_automorphisms}
  and~\ref{ex:group_Fell_bundles}.
  The same holds for twisted group actions, see
  Example~\ref{ex:twisted_actions} below.
  In this way, Theorem~\ref{thm:topologically_graded_aperiodic}
  generalises and improves upon
  \cite{Zarikian:Unique_expectations}*{Theorem~2.2.2} and
  \cite{Zarikian:Unique_pseudo}*{Theorem~3.5}.
\end{remark}


\section{A Galois correspondence for regular C*-irreducible
  inclusions}
\label{sec:regular C*-irreducible inclusions}

We shall use the noncommutative generalisation of Renault's Cartan
subalgebras due to Exel:

\begin{definition}[\cite{Exel:noncomm.cartan}*{Definition~12.1}]
  \label{def:virtual_commutant}
  A \(\Cst\)\nb-inclusion \(A\subseteq B\) is a \emph{noncommutative
    Cartan subalgebra} if it is regular, equipped with an almost
  faithful conditional expectation \(E\colon B\to A\), and any
  \(A\)\nb-bimodule map \(J\to B\) that is defined on an ideal~\(J\)
  in~\(A\) has range in~\(A\).
\end{definition}

We gave many characterisations of non-commutative Cartan inclusions in
\cite{Kwasniewski-Meyer:Cartan}*{Theorem~4.2}, assuming that there is
an almost faithful conditional expectation \(E\colon B\to A\).
If, in addition, \(A\) is simple, then
\cite{Kwasniewski-Meyer:Cartan}*{Corollary~7.4} gives equivalences
between a number of conditions, which resemble some of those in
Theorem~\ref{thm:main_theorem}.
We now add to this theory by deducing that any irreducible, regular
inclusion of simple \(\Cst\)\nb-algebras is aperiodic and has a
(unique) conditional expectation.
In particular, it follows that irreducibility and
\(\Cst\)\nb-irreducibility are equivalent for regular inclusions of
simple \(\Cst\)\nb-algebras.
This fails for irregular inclusions by
Example~\ref{ex:crossed_product_inclusions}.

An \emph{isomorphism of Fell bundles} \(\B= (B_g)_{g\in G}\) and
\(\mathcal{C}= (C_g)_{g\in G}\) over a group~\(G\) is a collection of
isomorphisms \(B_t\congto C_t\) for \(t\in S\) that intertwine the
multiplications and involutions in the Fell bundles.
This generalises naturally to Fell bundles over isomorphic groups.

\begin{lemma}
  \label{lem:regular_vs_grading}
  Let \(A\subseteq B\) be a nondegenerate \(\Cst\)\nb-inclusion.
  Then~\(A\) is simple if and only if the nonzero slices
  \(\Slice_A(B)\setminus \{0\}\) form a group.
  If~\(A\) is simple, then the \(\Cst\)\nb-inclusion \(A\subseteq B\)
  is regular if and only if there is a discrete group~\(G\) and a Fell
  bundle \(\B=(B_g)_{g\in G}\) of closed subspaces of~\(B\) such that
  \(A=B_1\) and there is a surjective \Star{}homomorphism
  \(\Cst(\B)\onto B\) that is the identity on the spaces~\(B_g\) for
  all \(g\in G\).
\end{lemma}

\begin{proof}
  Idempotents in the inverse semigroup \(\Slice_A(B)\) are the same as
  ideals in~\(A\).
  Thus~\(A\) is simple if and only if the unit is the only idempotent
  in \(\Slice_A(B)\setminus \{0\}\).
  This happens if and only if \(\Slice_A(B)\setminus \{0\}\) is a
  group.
  In this case, the spaces in \(\Slice_A(B)\setminus \{0\}\)
  themselves form a canonical Fell bundle \(\B=(B_g)_{g\in G}\) over
  this group.
  The universal property of the full section \(\Cst\)\nb-algebra
  provides a canonical \Star{}homomorphism \(\Cst(\B)\onto B\).
  It is surjective if and only if the span of \(\Slice_A(B)\)  is
  dense in~\(B\).
\end{proof}

\begin{theorem}
  \label{thm:regular_irreducible_for_simple}
  Let \(A\subseteq B\) be a regular inclusion where~\(A\) is simple.
  Assume that there is an almost faithful pseudo-expectation \(B\to I(A)\)
  \textup{(}this follows if~\(B\) is simple\textup{)} or that~\(B\) is
  nuclear.
  Then the following conditions are equivalent:
  \begin{enumerate}
  \item \label{enu:regular_irreducible_for_simple1}%
    \(A\subseteq B\) is irreducible;
  \item \label{enu:regular_irreducible_for_simple2}%
    \(A\subseteq B\) is \(\Cst\)\nb-irreducible;
  \item \label{enu:regular_irreducible_for_simple3}%
    \(A\subseteq B\) is a noncommutative Cartan inclusion;
  \item \label{enu:regular_irreducible_for_simple4}%
    \(B\cong \Cst_\red(\B)\)  for an outer Fell bundle
    \(\B=(B_g)_{g\in G}\) over a discrete group~\(G\) with an
    isomorphism that restricts to the identity \(A=B_1\);
  \item \label{enu:regular_irreducible_for_simple5}%
    each slice in \(\Slice_A(B)\setminus \{A\}\) is aperiodic as an
    \(A\)\nb-bimodule.
  \end{enumerate}
  Assume the above equivalent conditions hold.
  Then the inclusion \(A\subseteq B\) is aperiodic and its unique
  pseudo-expectation is a faithful conditional expectation.
  The Fell bundle~\(\B\)
  in~\ref{enu:regular_irreducible_for_simple4} is isomorphic to the
  Fell bundle formed by the nonzero slices
  \(\Slice_A(B)\setminus\{0\}\).
  The algebra~\(B\) is nuclear if and only if~\(A\) is nuclear
  and~\(\B\) has Exel's approximation property.
\end{theorem}

\begin{proof}
  Condition~\ref{enu:regular_irreducible_for_simple4} implies
  \ref{enu:regular_irreducible_for_simple3}
  by
  \cite{Kwasniewski-Meyer:Cartan}*{Corollary~7.4}, which also gives
  that the Fell bundle in~\ref{enu:regular_irreducible_for_simple4}
  is unique up to isomorphism.
  Since \(A\) is simple, \ref{enu:regular_irreducible_for_simple3}
  implies that \(A\subseteq B\) is aperiodic and equipped with a
  faithful conditional expectation, see
  \cite{Kwasniewski-Meyer:Cartan}*{Theorem 6.3}.
  The latter  implies~\ref{enu:regular_irreducible_for_simple2} by
  Proposition~\ref{prop:aperiodic_implies_supports}.
  Lemma~\ref{C-irreducible_implies_irreducible}
  shows that~\ref{enu:regular_irreducible_for_simple2}
  implies~\ref{enu:regular_irreducible_for_simple1}.
  Now assume~\ref{enu:regular_irreducible_for_simple1}.
  Then \cite{Kwasniewski-Meyer:Cartan}*{Proposition~4.5} implies that
  each slice in \(\Slice_A(B)\setminus \{A\}\) is outer as a Hilbert
  \(A\)\nb-bimodule.
  Since~\(A\) is simple, outerness is equivalent to aperiodicity.
  Hence \ref{enu:regular_irreducible_for_simple1}
  implies~\ref{enu:regular_irreducible_for_simple5}, and
  to close the cycle of implications we only need to show
  that~\ref{enu:regular_irreducible_for_simple5}
  implies~\ref{enu:regular_irreducible_for_simple4}.

  Assume~\ref{enu:regular_irreducible_for_simple5}.
  Turn \(\Slice_A(B)\setminus \{0\}\) into a Fell bundle
  \(\B=(B_g)_{g\in G}\) over a discrete group~\(G\) as in
  Lemma~\ref{lem:regular_vs_grading}.
  So there is a canonical surjective \Star{}homomorphism
  \(\Cst(\B)\onto B\).
  \ref{enu:regular_irreducible_for_simple5} means that~\(\B\) is
  aperiodic.
  Both inclusions \(A\subseteq B\) and \(A\subseteq \Cst(\B)\)
  are aperiodic by
  \cite{Kwasniewski-Meyer:Essential}*{Proposition~6.3}.
  Thus both have a unique pseudo-expectation by
  Proposition~\ref{prop:aperiodic_implies_supports}.
  Hence the composite of the map \(\Cst(\B)\onto B\) with the
  pseudo-expectation for \(A\subseteq B\) must be the canonical
  conditional expectation \(\Cst(\B)\to A\).
  Therefore, \(B\) is topologically graded, that is, we have canonical
  maps \(\Cst(\B) \onto B\onto \Cst_\red(\B)\).
  The unique pseudo-expectation~\(E\) on~\(B\) has to be the canonical
  conditional expectation \(B\onto \Cst_\red(\B)\to A\).
  So assuming that $B$ admits an almost faithful pseudo-expectation, means 
	that \(E\) is   faithful, which implies \(B= \Cst_\red(\B)\).
  If~\(B\) is nuclear, then \(\Cst_\red(\B)\) is nuclear as well as a
  quotient of~\(B\).
  Then~\(\B\) has Exel's approximation property by
  \cite{Abadie-Buss-Ferraro:Amenability}*{Proposition~7.2}.
  This implies \(\Cst(\B)=\Cst_\red(\B)\) and then \(B=\Cst_\red(\B)\).
  So \(B= \Cst_\red(\B)\) in both cases.

  This proves the first part of the assertion, and also the second part except
  the last statement about nuclearity. If~\(B\) is nuclear, then so is~\(A\) because of the conditional
  expectation \(E\colon B\to A\).
  If~\(A\) is nuclear, then
  \cite{Abadie-Buss-Ferraro:Amenability}*{Proposition~7.2} says
  that~\(B\) is nuclear if and only if the Fell bundle has Exel's
  approximation property.
\end{proof}

Now we have all ingredients to prove Theorem~\ref{thm:main_theorem}.

\begin{proof}[Proof of Theorem~\ref{thm:main_theorem}]
  \label{proof:main_theorem}
  Conditions \ref{enu:main1} and~\ref{enu:main2} in
  Theorem~\ref{thm:main_theorem} are equivalent for any
  \(\Cst\)\nb-inclusion \(A\subseteq B\) with simple~\(A\),
  and~\ref{enu:main3} implies them by
  Remark~\ref{rem:supporting_detecting}.
  Condition~\ref{enu:main4} implies~\ref{enu:main3} by
  Proposition~\ref{prop:aperiodic_implies_supports}.
  It is clear that \ref{enu:main4.5} implies~\ref{enu:main4}.
  Theorem~\ref{thm:regular_irreducible_for_simple} shows that the
  conditions \ref{enu:main8}, \ref{enu:main8.5} and~\ref{enu:main10}
  are equivalent and that they imply~\ref{enu:main4.5}.
  Condition~\ref{enu:main1} implies~\ref{enu:main8} by
  Lemma~\ref{C-irreducible_implies_irreducible}.
  Hence the conditions \ref{enu:main1}--\ref{enu:main8.5}
  and~\ref{enu:main10} are all equivalent.
	
  It follows from \cite{Kwasniewski-Meyer:Cartan}*{Corollary~7.4} that
  \ref{enu:main10}\(\Leftrightarrow
  \)\ref{enu:main9}\(\Leftrightarrow
  \)\ref{enu:main5}\(\Rightarrow
  \)\ref{enu:main8.5} and that all Fell bundles as in~\ref{enu:main10}
  are isomorphic.
  Remark~\ref{rem:outer_expectation_means_aperiodic} shows that
  \ref{enu:main4}\(\Leftrightarrow
  \)\ref{enu:main6}.
  Lemma~\ref{lem:aperiodic_implies_pinching} shows
  that~\ref{enu:main6} implies~\ref{enu:main7},
  and Lemma~\ref{lem:pinching_plus_faithful} shows that~\ref{enu:main7}
  implies~\ref{enu:main1}.
  This shows that all conditions \ref{enu:main1}--\ref{enu:main10} are
  equivalent.

  Topologically free actions are aperiodic by
  \cite{Kwasniewski-Meyer:Aperiodicity_pseudo_expectations}*{Corollary~4.8},
  so that~\ref{enu:main11} always implies~\ref{enu:main10}.
  If~\(A\) is separable (or if~\(A\) contains a separable essential ideal),
  then the converse also holds by
  \cite{Kwasniewski-Meyer:Aperiodicity}*{Theorem~8.1}, see
  Figure~\ref{fig:diagram}.
  By
  \cite{Kwasniewski-Meyer:Aperiodicity_pseudo_expectations}*{Theorem
    5.5} (see Proposition~\ref{prop:almost_extension_aperiodic}), if
  \(A\subseteq B\) has the almost extension property, then it is
  aperiodicity, and the converse holds if~\(B\) is separable.
  Hence~\ref{enu:main12} implies~\ref{enu:main4.5}
  and~\ref{enu:main13} implies~\ref{enu:main4}, and the converse
  implications hold if \(B\) is separable.
\end{proof}

Next, we explain the bijection in
Theorem~\ref{Thm:Galois_correspondence}.
We start with a general Galois connection. Recall that \(\Slice_A(B)\) is the unital inverse semigroup of slices
for a nondegenerate inclusion
\(A\subseteq B\).

\begin{lemma}
  \label{lem:Galois_connection}
  Let \(A\subseteq B\) be a nondegenerate \(\Cst\)\nb-inclusion.
  Let \(\Sub_A(B)\) be the set of intermediate subalgebras for
  \(A\subseteq B\) and let \(\Sub(\Slice_A(B))\) be the set of unital
  inverse subsemigroups of \(\Slice_A(B)\), both ordered by inclusion.
  The following natural maps form a monotone Galois connection:
  \begin{align*}
    \Sub_A(B)\ni C
    &\stackrel{}{\longmapsto} \Slice_A(C)\in \Sub(\Slice_A(B)),\\
    \Sub_A(B)\ni \clsp\bigcup S
    & \stackrel{}{\mathrel{\reflectbox{\ensuremath{\longmapsto}}}} S\in \Sub(\Slice_A(B)).
  \end{align*}
  This yields a Galois correspondence
  \[
    \Sub_A(B)\supseteq \Sub_A^{\textup{reg}}(B)
    \cong  \Slice_A(\Sub_A(B))\subseteq  \Sub(\Slice_A(B)),
  \]
  where \(\Sub_A^{\textup{reg}}(B)\) is the set of all intermediate
  \(\Cst\)\nb-subalgebras~\(C\) such that \(A\subseteq C\) is regular.
\end{lemma}

\begin{proof}
  It is immediate that the two maps are well-defined and
  order-preserving.
  Let \(C\in \Sub_A(B)\) and \(S\in \Sub(\Slice_A(B))\).
  Then \(S\subseteq \Slice_A(C)\) if and only if \(M\subseteq C\)
  for all \(M\in S\), if and only if \(\clsp\bigcup S\subseteq C\).
  Hence the two maps form a Galois connection.
  It follows that the restricted maps between the images of the two
  maps are inverse to each other.
  If \(S\in \Sub(\Slice_A(B))\) is a unital inverse subsemigroup, then
  \(\clsp\bigcup S\in \Sub_A^{\textup{reg}}(B)\).
  Conversely, any algebra in \(\Sub_A^{\textup{reg}}(B)\) is of this
  form.
  This yields the Galois correspondence.
\end{proof}


Consider a noncommutative Cartan subalgebra \(A\subseteq B\) and write
\(B= \Cst_\red(\B)\) for an outer  Fell bundle over the inverse
semigroup \( \Slice_A(B)\) as in
\cite{Kwasniewski-Meyer:Cartan}*{Theorem~4.3}.
Let~\(\check{A}\) be the primitive ideal space of~\(A\), equipped with
the induced action of~\(\Slice_A(B)\), and let \(\check{A}\rtimes  \Slice_A(B)\) be the
resulting transformation groupoid.
By \cite{Kwasniewski-Meyer:Cartan}*{Theorem~5.6}, we have a
natural isomorphism \(\Slice_A(B)\ni M\mapsto \check{M}\in
\Bis(\check{A}\rtimes  \Slice_A(B))\) of inverse semigroups.
We get the following  noncommutative generalisation of
\cite{Brown-Exel-Fuller-Pitts-Reznikoff:Intermediate}*{Theorem~3.3}.

\begin{proposition}
  \label{pro:ncCartan_range_Galois}
  Let \(A\subseteq B\) be a noncommutative Cartan subalgebra.
  Then there are bijective correspondences between
  \begin{enumerate}
  \item\label{enu:ncCartan_range_Galois1}%
    inverse subsemigroups~\(S\) of \(\Slice_A(B)\) coming from an
    intermediate \(\Cst\)\nb-algebra or belonging to
    \(\Slice_A(\Sub_A(B))\);
  \item\label{enu:ncCartan_range_Galois2}%
    open, wide subgroupoids~\(H\) of the transformation groupoid
    \(\check{A}\rtimes  \Slice_A(B)\);
  \item\label{enu:ncCartan_range_Galois3}%
    intermediate \(\Cst\)-algebras~\(C\) such that  \(A\subseteq C\)
    is noncommutative Cartan.
  \end{enumerate}
  These correspondence are given by the relations
  \(
  S=\Slice_A(C)\cong \Bis(H)
  \) and \(H=\bigcup_{M\in S} \check{M}\).
\end{proposition}

\begin{proof}
  The inclusion \(A\subseteq C\) of an intermediate subalgebra remains
  a noncommutative Cartan inclusion if and only if it is regular.
  Hence Lemma~\ref{lem:Galois_connection} implies that the relation
  \(S=\Slice_A(C)\) yields a bijective correspondence between
  the objects in \ref{enu:ncCartan_range_Galois1} and~\ref{enu:ncCartan_range_Galois3}.
  Let~\(C\) be as in \ref{enu:ncCartan_range_Galois3}.
  The transformation groupoid \(H\defeq \check{A}\rtimes \Slice_A(C)\) is naturally an
  open, wide  subgroupoid of \(\check{A}\rtimes \Slice_A(B)\) because
  \(\Slice_A(C) \subseteq \Slice_A(B)\) and both inverse semigroups
  have the same idempotents.
  Applying \cite{Kwasniewski-Meyer:Cartan}*{Theorem~5.6}  to both
  Cartan inclusions \(A\subseteq B\) and \(A\subseteq C\), we conclude
  that the isomorphism \(\Slice_A(B) \cong \Bis(\check{A}\rtimes
  \Slice_A(B))\) restricts to the isomorphism \(\Slice_A(C) \cong
  \Bis(H)\) and, in particular, \(H\) determines~\(C\).
  Conversely, let \(H\subseteq \check{A}\rtimes \Slice_A(B)\) be any
  open, wide  subgroupoid and let~\(S\) be the inverse subsemigroup of
  \(\Slice_A(B)\) corresponding to the inverse subsemigroup
  \(\Bis(H)\) of  \(\Bis(\check{A}\rtimes  \Slice_A(B))\cong
  \Slice_A(B)\).
  Then \(C\defeq \clsp\bigcup S\) is a \(\Cst\)\nb-subalgebra of~\(B\)
  containing~\(A\) as a noncommutative Cartan subalgebra.
  By construction, \(S\) is an inverse-semigroup grading of~\(C\)
  and~\(S\) is closed under intersections.
  Hence~\(S\) is a wide grading of~\(C\) and so we have a canonical
  isomorphism \(C\cong A\rtimes S\) by the characterisation~(7) in 	
  \cite{Kwasniewski-Meyer:Cartan}*{Theorem~4.3}.
  Therefore, \(\Slice_A(C) \cong \Bis(\check{A}\rtimes S)\) again by
  \cite{Kwasniewski-Meyer:Cartan}*{Theorem~5.6}.
  Since \(\check{A}\rtimes S\cong \check{A}\rtimes \Bis(H)\cong H \),
  we get \(\Slice_A(C) \cong\Bis(H)\) and \(\Slice_A(C)=S\).
  This proves the correspondence between the objects in
  \ref{enu:ncCartan_range_Galois2}
  and~\ref{enu:ncCartan_range_Galois3}.
\end{proof}

Assume now that~\(A\) is a  simple Cartan subalgebra of \(B\).
Then \(\check{A}\) is a point and \(G\defeq\check{A}\rtimes \Slice_A(B)\) is
a group that can be identified with \(\Slice_A(B)\setminus \{0\}\).
So the intermediate subgroupoids in
Proposition~\ref{pro:ncCartan_range_Galois} are the same as subgroups
of~\(G\), and the Galois correspondence in
Lemma~\ref{lem:Galois_connection} gives a bijection between the sets
of regular intermediate subalgebras \(A\subseteq C\subseteq B\) and of
subgroups of~\(G\).
Theorem~\ref{Thm:Galois_correspondence}
implies that in fact  all intermediate subalgebras for \(A\subseteq B\) are
regular.
We now prove a slightly stronger result.

\begin{theorem}
  \label{thm:Galois_correspondence_explained}
  Let \(A\subseteq B\) be a regular topologically graded
  \(\Cst\)\nb-inclusion where~\(A\) is simple.
  Equivalently, \(B\) is an exotic section \(\Cst\)\nb-algebra of a
  saturated Fell bundle \(\B=(B_g)_{g\in G}\) over a discrete
  group~\(G\) with a simple unit fibre \(A=B_1\).
  Then there is a natural lattice isomorphism
  \(\Slice_A(\Sub_A(B))\cong \Sub(G)\), where \(\Sub(G)\) is the
  lattice of subgroups of~\(G\).
  The following conditions are equivalent:
  \begin{enumerate}
  \item \label{enu:Galois_correspondence_explained1}%
    all intermediate inclusions \(A\subseteq C\) are regular, that is,
    \(\Sub_A(B)=\Sub_A^{\textup{reg}}(B)\);
  \item \label{enu:Galois_correspondence_explained1.25}%
    the map \(\Sub(G)\ni H\mapsto \clsp\bigcup_{h\in H} B_h\in
    \Sub_A(B)\)  is a lattice isomorphism \(\Sub(G)\cong \Sub_A(B)\);
  \item \label{enu:Galois_correspondence_explained1.5}%
    \(B=\Cst_\red(\B)\)  and \(\Sub(G)\ni H \mapsto
    \Cst_\red(\B|_{H})\in \Sub_A(B)\) is a bijection;
  \item \label{enu:Galois_correspondence_explained3}%
    \(B=\Cst_\red(\B)\)  and~\(\B\) is outer;
  \item \label{enu:Galois_correspondence_explained2}%
    \(A\subseteq B\) is \(\Cst\)\nb-irreducible.
  \end{enumerate}
\end{theorem}

\begin{proof}
  The equivalence in the first two sentences follows from
  Proposition~\ref{prop:topological_grading_characterisation}, where
  \(G\defeq \Slice_A(B)\setminus \{0\}\), see also
  Lemma~\ref{lem:regular_vs_grading}.
  For every regular intermediate \(\Cst\)\nb-subalgebra \(C\in
  \Sub_A(B)\), the associated inverse semigroup is of the form
  \(\Slice_A(C)= H\cup \{0\}\) where \(H=\Slice_A(C)\setminus \{0\}\)
  is a subgroup of~\(G\).
  This gives the asserted isomorphism \(\Slice_A(\Sub_A(B))\cong
  \Sub(G)\).

  The Galois correspondence in Lemma~\ref{lem:Galois_connection}
  yields the lattice isomorphism \(\Sub(G)\cong
  \Slice_A(\Sub_A(B))\cong\Sub_A^{\text{reg}}(B)\).
  Thus it gives \(\Sub(G)\cong \Sub_A(B)\) if and only if
  \(\Sub_A^{\text{reg}}(B)=\Sub_A(B)\).
  This explains the equivalence between the conditions
  \ref{enu:Galois_correspondence_explained1}
  and~\ref{enu:Galois_correspondence_explained1.25}.
  Condition~\ref{enu:Galois_correspondence_explained1.5}
  implies~\ref{enu:Galois_correspondence_explained1.25} because we may
  naturally identify \(\Cst_\red(\B|_H)\) with \(\clsp\bigcup_{h\in H}
  B_h\subseteq \Cst_\red(\B)\) for any \(H\in \Sub(G)\).
  Conditions \ref{enu:Galois_correspondence_explained3}
  and~\ref{enu:Galois_correspondence_explained2} are equivalent by
  Theorem~\ref{thm:main_theorem}.
  Thus we only need to show the implications
  \ref{enu:Galois_correspondence_explained1}\(\Rightarrow
  \)\ref{enu:Galois_correspondence_explained2} and
  \ref{enu:Galois_correspondence_explained3}\(\Rightarrow
  \)\ref{enu:Galois_correspondence_explained1.5}.

  First, we prove that~\ref{enu:Galois_correspondence_explained1}
  implies~\ref{enu:Galois_correspondence_explained2}.
  We first show that~\(A\) detects ideals in~\(B\).
  Let~\(J\) be an ideal in~\(B\) with \(J\cap A=0\).
  Then  \(C\defeq J+A\) is an intermediate \(\Cst\)\nb-subalgebra.
  It is graded by our
  assumption~\ref{enu:Galois_correspondence_explained1}, and so
  \(C=\overline{\bigoplus_{h\in H} B_h}\) with \(H\defeq
  \setgiven{h\in G}{B_h\subseteq C}\).
  If there were \(h\in H\setminus\{1\}\), then \(B_h\subseteq J\) and
  so \(B_{h}^*B_{h}=B_{h^{-1}} B_h\subseteq J\cap A=0\), which implies
  that \(B_h=0\) and this  is impossible (because \(\B\) is saturated).
  Hence \(C=A\) and \(J=0\).
  Thus~\(A\) detects ideals in~\(B\) as claimed and so~\(B\) is
  simple.
  Now if~\(C\) is any intermediate \(\Cst\)\nb-algebra, then our
  assumption implies that  \(\Sub_A(C)=\Sub_A^{\textup{reg}}(C)\).
  Hence the argument above may also be applied to the inclusion \(A\subseteq
  C\) and shows that~\(C\) is simple.
  Thus we get~\ref{enu:Galois_correspondence_explained2}.

  Next, we prove that~\ref{enu:Galois_correspondence_explained3}
  implies~\ref{enu:Galois_correspondence_explained1.5}.
  For each \(g\in G\), there is a unique contractive linear map
  \(E_g\colon \Cst_\red(\B)\onto B_g\subseteq \Cst_\red(\B)\) such
  that \(b_g^*E_g(b)=E(b_g^*b)\) for all \(b_g\in B_g\) and \(b\in
  \Cst_\red(\B)\).
  Such a map is an idempotent \(A\)\nb-bimodule map,
  and for each \(b\in  \Cst_\red(\B)\) the ``Fourier coefficients''
  \((E_g(b))_{g\in G}\) determine~\(b\) uniquely.
  Fix \(b\in  \Cst_\red(\B)\) and \(g\in G\).
  We claim that \(E_g(b)\in \overline{A b A}\).
  Let \(\varepsilon >0\).
  Then \cite{Brown-Mingo-Shen:Quasi_multipliers}*{Remark~1.9} provides
  elements \(c_1,\dotsc,c_n\in B_g\) with \(\norm*{\sum_{i=1}^n c_i
    c_i^*E_g(b)- E_g(b)}< \varepsilon\).
  For each \(i=1,\dotsc, n\),
  Proposition~\ref{prop:module_approximation} provides elements
  \((a_{i,j})_{j=1}^{n_i}\), \((b_{i,j})_{j=1}^{n_i}\) in~\(A\) with
  \(\norm*{E(c_i^*b)- \sum_{j=1}^{n_i} a_{i,j}(c_i^*b)b_{i,j}}<
  \frac{\varepsilon}{n\norm{c_i}}\).
  Accordingly,
  \begin{multline*}
    \norm*{\sum_{i=1}^n\sum_{j=1}^{n_i} (c_i a_{i,j}c_i^*)b b_{i,j}-
      E_g(b)}\\
    <
    \sum_{i=1}^n \norm{c_i} \norm*{\sum_{j=1}^{n_i}  a_{i,j}(c_i^*b) b_{i,j}- E(c_i^*b)}
    +\norm*{\sum_{i=1}^n c_i E(c_i^*b)- E_g(b)}
  \end{multline*}
  Both terms are at most~\(\varepsilon\) because \(c_i E(c_i^*b) = c_i
  c_i^* E_g(b)\).
  This proves the claim because \(c_i a_{i,j}c_i^*\in A\) for
  \(i=1,\dotsc,n\) and \(\varepsilon>0\) was arbitrary.

  Now we fix an intermediate \(\Cst\)\nb-algebra \(A\subseteq
  C\subseteq \Cst_\red(\B)\).
  Let \(g\in G\).
  Since~\(A\) is simple, \(B_g=\overline{Ab_g A}\) for any nonzero
  \(b_g\in B_g\).
  Thus, the claim above implies that \(E_g(C)\neq 0\) if and only if
  \(B_g\subseteq C\).

  The set \(H\defeq\setgiven{g\in G}{B_g\subseteq C}=\setgiven{g\in
    G}{ E_g(C)\neq 0}\) is a subgroup of~\(G\) and
  \(\Cst_\red(\B|_H)\subseteq C\).
  We claim that \(\Cst_\red(\B|_H)= C\).
  There is a conditional expectation \(E_H\colon  \Cst_\red(\B)\to
  \Cst_\red(\B|_H)\subseteq \Cst_\red(\B)\) such that
  \(E_H(B_g)=\{0\}\) for \(g\in G\setminus H\).
  The map~\(E_H\) is a \(\Cst_\red(\B|_H)\)-bimodule map and
  \(E=E\circ E_H\).
  If \(b\in B\), \(h\in H\), and \(c_h\in B_h\),  then
  \[
    c_h^*E_h(b)
    = E(c_h^*b)
    = E(E_H(c_h^*b))
    = E(c_h^* E_H(b))
    = c_h^* E_h(E_H(b)).
  \]
  This shows that \(E_h=E_h\circ E_H\) for every \(h\in H\).
  If \(b\in C\) and \(g\in G\setminus H\), then
  \(E_g(b)=0=E_g(E_H(b))\).
  Thus \(E_g(b)=E_g(E_H(b))\) holds for all \(g\in G\), and this
  implies \(b=E_H(b)\in \Cst_\red(\B|_H)\) because the Fourier
  coefficients determine an element of \(\Cst_\red(\B|_H)\) uniquely.
\end{proof}

\begin{corollary}
  \label{cor:general_consequences_of_Galois}
  Let \(A\subseteq B\) be a regular \(\Cst\)\nb-irreducible inclusion.
  \begin{enumerate}
  \item \label{enu:general_consequences_of_Galois1}%
    There is a bijection between intermediate subalgebras and
    subgroups of \(\Slice_A(B)\setminus\{0\}\), which is a group.
  \item \label{enu:general_consequences_of_Galois2}%
    An intermediate \(\Cst\)\nb-subalgebra \(A\subseteq C \subseteq B\) is purely infinite
    if and only if every nonzero element in~\(A^+\) is infinite in~\(C\).
  \item \label{enu:general_consequences_of_Galois3}%
    If~\(A\)  is purely infinite, then all intermediate
    \(\Cst\)\nb-subalgebras \(A\subseteq C \subseteq B\) are purely
    infinite.
  \item \label{enu:general_consequences_of_Galois4}%
    If~\(B\) is nuclear, then all intermediate \(\Cst\)\nb-subalgebras
    \(A\subseteq C \subseteq B\) are nuclear.
  \end{enumerate}
\end{corollary}

\begin{proof}
  Theorem~\ref{thm:regular_irreducible_for_simple} allows us to assume
  that \(B=\Cst_\red(\B)\) for an outer (equivalently, aperiodic) Fell
  bundle \(\B=(B_g)_{g\in G}\) over a discrete group~\(G\)  with
  \(A=B_1\).
  The Fell bundle consists of the nonzero slices
  \(\Slice_A(B)\setminus\{0\}\).
  Hence~\ref{enu:general_consequences_of_Galois1} follows from
  Theorem~\ref{thm:Galois_correspondence_explained}.
  For any subgroup \(H\subseteq G\), the Fell bundle
  \(\B|_H=(B_h)_{h\in H}\) is aperiodic with a simple unit fibre
  \(B_1=A\).
  Therefore, \(\Cst_\red(\B|_H)\) is purely infinite if and only if
  every \(a\in A^+\setminus\{0\}\) is infinite in \(\Cst_\red(\B|_H)\)
  by \cite{Kwasniewski-Szymanski:Pure_infinite}*{Theorem~4.10}.
  This yields~\ref{enu:general_consequences_of_Galois2}, which clearly
  implies~\ref{enu:general_consequences_of_Galois3}.
  Statement~\ref{enu:general_consequences_of_Galois4} follows from
  \cite{Abadie-Buss-Ferraro:Amenability}*{Proposition 7.2} because
  Exel's approximation property passes to subgroups.
\end{proof}

As an application, we study tensor products of \(\Cst\)\nb-irreducible
inclusions by simple \(\Cst\)\nb-algebras.
For \(\Cst\)\nb-algebras \(A\) and~\(D\), let \(A\otimes D\) be their
minimal \(\Cst\)\nb-algebra tensor product.
By Takesaki's Theorem \cite{Takesaki:Theory_1}*{Corollary IV.4.21},
if \(A\) and~\(D\) are simple, then so is \(A\otimes D\).
Using a result of Zacharias and Zsido, Rørdam showed in
\cite{Rordam:Irreducible_inclusions}*{Theorem~2.2} that the
\(\Cst\)\nb-inclusion \(A\otimes D\subseteq B\otimes D\) is
\(\Cst\)\nb-irreducible whenever the inclusion \(A\subseteq B\) is unital
and \(\Cst\)\nb-irreducible  and~\(D\) is a unital
\(\Cst\)\nb-algebra with Wassermann's property~(S).
If \(A\subseteq B\) is regular,  our results allow  to remove the assumption of
property~(S).
To show this, we will use the following simple lemma.

\begin{lemma}
  \label{lem:tensoring_outer_gives_outer}
  Let~\(A\) be a simple \(\Cst\)\nb-algebra, let~\(M\)  be an outer
  Hilbert \(A\)\nb-bimodule, and let~\(D\) be a \(\Cst\)\nb-algebra
  that contains a nonzero projection~\(p\).
  Then the Hilbert \(A\otimes D\)\nb-bimodule \(M\otimes D\) is outer.
\end{lemma}

\begin{proof}
  The Hilbert \(A\otimes D\)-bimodule structure on \(M\otimes D\) is
  defined entrywise on simple tensors.
  Assume that \(M\otimes D\) is inner, that is, there is a unitary
  bimodule isomorphism
  \(U\colon M\otimes D\congto A\otimes D\).
  There is a state~\(\varphi\) on~\(D\) with \(\varphi(p)=1\).
  Consider the associated slice map \(R\colon A\otimes D\to A\),
  determined by \(R(a\otimes d)= \varphi(d)a\) for all \(a\in A\),
  \(d\in D\).
  Define a linear map \(V\colon M\to A\) by \(V(x)\defeq R(U(x\otimes
  p))\) for \(x\in M\).
  Since~\(U\) is \(A\otimes D\)-bilinear, \(V\) is \(A\)\nb-bilinear.
  Namely, if \(x\in M\) and \(a\in A\), then
  \[
    (Vx)  a
    = R(U(x\otimes p))R(a\otimes p)
    = R(U(x a \otimes p^2))
    = V(x a),
  \]
  and a similar computation shows \(a(V x)=V(a x)\).
  If \(x,y \in M\), then
  \begin{align*}
    (Vx)^*V(y)
    &= R\bigl(U(x\otimes p)^* U(y\otimes p)\bigr)
    = R(\langle x\otimes p, y\otimes p \rangle_{A\otimes D})
    \\
		&=R(\langle x,y\rangle_A \otimes p^2)
    = \langle x,y\rangle_A.
  \end{align*}
  This implies that~\(V\) is an isometry.
  Its range must be a closed ideal in~\(A\).
  Since~\(A\) is simple, \(V\) is surjective and thus a unitary.
  Then~\(M\) is not outer.
\end{proof}

\begin{proposition}[compare \cite{Rordam:Irreducible_inclusions}*{Theorem~7.1}]
  For any regular \(\Cst\)\nb-irreducible inclusion \(A\subseteq B\)
  and any simple \(\Cst\)\nb-algebra~\(D\) with a nonzero projection,
  the tensor product inclusion \(A\otimes D\subseteq B\otimes D\) is
  \(\Cst\)\nb-irreducible, and all intermediate \(\Cst\)\nb-algebras
  are of the form \(C\otimes D\) with \(A\subseteq C\subseteq B\).
\end{proposition}

\begin{proof}
  Theorem~\ref{thm:regular_irreducible_for_simple} allows us to write
  \(B=\Cst_\red(\B)\) for a saturated outer Fell bundle
  \(\B=(B_g)_{g\in G}\) over a discrete group~\(G\)  with \(A=B_1\).
  Then there is natural tensor product Fell bundle \(\B\otimes D\defeq
  (B_g\otimes D)_{g\in G}\) and a natural \Star{}isomorphism
  \(\Cst_\red(\B\otimes D)\cong \Cst_\red(\B)\otimes D\), see
  \cite{Abadie:Tensor}*{Proposition~4.6} or
  \cite{Buss-Martinez:Approximation_Fell}*{Section~5.1}.
  Hence the inclusion \(A\otimes D\subseteq B\otimes D\) may be identified with
  \(A\otimes D\subseteq \Cst_\red(\B\otimes D)\).
  The bundle \(\B\otimes D\) is outer by
  Lemma~\ref{lem:tensoring_outer_gives_outer}.
  It is saturated as well.
  Hence Theorems \ref{thm:main_theorem}
  and~\ref{Thm:Galois_correspondence} imply that \(A\otimes D\subseteq
  B\otimes D\) is \(\Cst\)\nb-irreducible and that every intermediate
  \(\Cst\)\nb-algebra is of the form \(\Cst_\red(\B\otimes
  D|_H)=\Cst_\red(\B|_H)\otimes D\) for a subgroup~\(H\) of~\(G\).
  This gives the assertion because  \(\Cst_\red(\B|_H)\) for \(H\le
  G\) are exactly the intermediate \(\Cst\)\nb-subalgebras for
  \(A\subseteq B\).
\end{proof}

\begin{remark}
By tensoring  with \(\Comp\),  \(\mathcal{Z}\), \(\mathcal{O}_2\),
  or~\(\mathcal{O}_\infty\),
  the above result allows  to turn any regular
  \(\Cst\)\nb-irreducible inclusion into a stable, \(\mathcal{Z}\)-stable,
  \(\mathcal{O}_2\)-stable or \(\mathcal{O}_\infty\)-stable one.
\end{remark}

We now make explicit that up to Morita equivalence all regular
\(\Cst\)\nb-irreducible inclusions arise from crossed products by
outer group actions.
We phrase this in terms of full hereditary subalgebras.
To this end, we use the following lemma.
It is stated in~\cite{Echterhoff-Rordam:Inclusions} for unital
inclusions, but the proof adapts readily to the setting of arbitrary
nondegenerate \(\Cst\)\nb-inclusions.

\begin{lemma}[\cite{Echterhoff-Rordam:Inclusions}*{Lemma~4.5}]
  \label{lem:Mortita_for_irreducibles}
  Let \(A\subseteq B\) be a nondegenerate \(\Cst\)\nb-inclusion and
  let \(p\in \Mult(A)\subseteq \Mult(B)\) be a projection with
  \(\overline{ApA}=A\) \textup{(}this is automatic when~\(A\) is
  simple\textup{)}.
  Then \(A\subseteq B\) is \(\Cst\)\nb-irreducible if and only if
  \(pAp\subseteq pBp\) is \(\Cst\)\nb-irreducible.
  If this holds, then \(C\mapsto pCp\) is a bijection between the sets
  of intermediate \(\Cst\)\nb-subalgebras of \(A\subseteq B\) and
  \(pAp\subseteq pBp\).
\end{lemma}

\begin{proposition}
  \label{prop:Morita_globalisation}
  For any regular \(\Cst\)\nb-irreducible inclusion \(A\subseteq B\),
  there are an outer action~\(\alpha\) of a discrete group~\(G\) on a
  simple \(\Cst\)\nb-algebra~\(C\) and a projection \(p\in \Mult(C)\)
  such that the \(\Cst\)\nb-inclusions \(A\subseteq B\) and
  \(pCp\subseteq p(C\rtimes^\red_\alpha  G)p\) are isomorphic.
  Then \(G\ge H\mapsto  p(C\rtimes^\red_\alpha  H)p\) gives
  a bijection between the subgroups of~\(G\) and the intermediate
  \(\Cst\)\nb-subalgebras for \(A\subseteq B\).
\end{proposition}

\begin{proof}
  By Theorem~\ref{thm:regular_irreducible_for_simple}, the inclusion
  \(A\subseteq B\) is modelled by an outer Fell
  bundle~\(\mathcal{B}\).
  The Morita globalisation of~\(\mathcal{B}\) from
  \cite{Kwasniewski-Meyer:Aperiodicity}*{Proposition~7.1} yields the
  action~\(\alpha\) that makes the first part of the assertion true.
  The last part now follows from
  Lemma~\ref{lem:Mortita_for_irreducibles} and
  Theorem~\ref{Thm:Galois_correspondence}.
\end{proof}

\begin{remark}
  Using the above proposition, one could, in principle, deduce
  Theorem~\ref{Thm:Galois_correspondence} from a Galois correspondence
  for the crossed product inclusion \(C\subseteq  C\rtimes^\red_\alpha
  G\).
  However, the \(\Cst\)\nb-algebra~\(C\) in
  Proposition~\ref{prop:Morita_globalisation} is  nonunital
  (unless~\(A\) is  unital and~\(G\) is finite).
  Since the Galois correspondence results in \cites{Cameron-Smith:Galois_reduced,
    Kennedy-Ursu:Intermediate_subalgebras_crossed} are established
  only for unital \(\Cst\)\nb-algebras, they do not apply to this inclusion.
  Moreover, unitality plays a fundamental role in the arguments
  of~\cite{Cameron-Smith:Galois_reduced} and cannot be simply removed.
  Therefore, we proved the Galois correspondence in
  Theorem~\ref{thm:Galois_correspondence_explained} directly, using
  Proposition~\ref{prop:module_approximation}, which is inspired by
  the techniques
  of~\cite{Kennedy-Ursu:Intermediate_subalgebras_crossed}.
\end{remark}

\section{Inclusions coming from group actions}
\label{sec:group_actions}

We first use our results to extend the Galois
correspondence for twisted crossed products by Cameron and
Smith~\cite{Cameron-Smith:Galois_reduced} to the nonunital setting.
For crossed products by finite groups, the corresponding result was
previously obtained by Izumi~\cite{Izumi:Inclusions_simple}.

\begin{example}[Twisted crossed products]
  \label{ex:twisted_actions}
  Recall that a \emph{Busby--Smith twisted action} of a discrete
  group~\(G\) on a \(\Cst\)\nb-algebra is a pair of maps
  \(\alpha\colon G\to \Aut(A)\) and \(\sigma\colon G\times G\to
  \UMult(A)\), where \(\UMult(A)\) is the group of unitaries in the
  multiplier algebra of~\(A\), such that \(\alpha_1=\Id_A\),
  \(\sigma(1,t)=\sigma(t,1)=1\), and
  \[
    \alpha_s\circ \alpha_t
    = \Ad_{\sigma(s,t)}\circ\alpha_{s t}, \qquad
    \alpha_r (\sigma(s,t)) \sigma(r,s t)
    =  \sigma(r,s) \sigma(rs,t),
  \]
  for all \(r,s,t\in G\), where \(\Ad_u(\cdot)\defeq u(\cdot)u^*\) and
  we also write~\(\alpha_t\) for the unique extension of this automorphism
  to \(\Mult(A)\), see
  \cite{Packer-Raeburn:Stabilisation}*{Definition~2.1}.
  This gives rise to the Fell bundle \(\B= (B_g)_{g\in G}\), where
  \(B_g\defeq \setgiven{a\delta_g}{a\in A}\) is canonically isomorphic
  to~\(A\) as a Banach space for all \(g\in G\), and the
  multiplication and involution are given by
  \[
    (a_{s}\delta_s) \cdot (a_t\delta_t)
    \defeq a_s \alpha_{s}(a_t)\sigma(s,t)\delta_{st}, \qquad
    (a_t\delta_t)^*
    \defeq \alpha^{-1}_t(a_t^*)\sigma(t^{-1},t)^*\delta_{t^{-1}},
  \]
  for all \(a_s,a_t\in A\) and \(s,t\in G\).
  There is a canonical \Star{}isomorphism between the associated
  reduced twisted crossed product \(A\rtimes^\red_{\alpha,\sigma} G\)
  and the reduced section \(\Cst\)\nb-algebra
  \(\Cst_\red(\B)\).
  The Fell bundle~\(\B\) is outer if and only if~\(\alpha\) is outer
  in the sense that~\(\alpha_t\) is outer for all \(t\in G\setminus
  \{1\}\).
  By \cite{Packer-Raeburn:Stabilisation}*{Definition~3.1}, two twisted
  actions \((\alpha,\sigma)\) and \((\beta,\omega)\)  of~\(G\)
  on~\(A\) are \emph{exterior equivalent} if  there is a map \(w\colon
  G\to \UMult(A)\) such that
  \[
    \Ad_{w_t}\circ \alpha_t=\beta_t, \qquad
    \omega(s,t)w_{st}=w_s\alpha_s(w_t)\sigma (s,t)
  \]
  for all \(s,t\in G\).  If the twists \(\sigma\) and~\(\omega\) are trivial, then exterior
  equivalence is the same as \emph{cocycle conjugacy} of \(\alpha\)
  and~\(\beta\). Having   \(w\) as above,
   the Fell bundles \(\B\) and~\(\mathcal{C}\) corresponding to
  \((\alpha,\sigma)\) and \((\beta,\omega)\) are isomorphic through
  the isomorphism given by \(B_t\ni a\delta_t\mapsto a
  w_t^*\delta_t\in C_t\) for all \(t\in G\).
  One also checks that every Fell bundle isomorphism from~\(\B\)
  onto~\(\mathcal{C}\) that is the identity on \(A=B_1=C_1\) arises
  this way.
  Therefore, if the inclusion \( A\subseteq
  A\rtimes^\red_{\alpha,\sigma} G\) is \(\Cst\)\nb-irreducible, then
  the twisted action \((\alpha,\sigma)\) is unique up to exterior
  equivalence by Theorem~\ref{thm:regular_irreducible_for_simple}.
\end{example}

\begin{theorem}[compare \cite{Cameron-Smith:Galois_reduced}*{Theorem~4.4},
  \cite{Izumi:Inclusions_simple}*{Corollary~6.6(1)},
  \cite{Rordam:Irreducible_inclusions}*{Theorem~5.8}]
  \label{cor:Cameron-Smith}
  Let \((\alpha,\sigma)\) be a twisted action of a discrete
  group~\(G\) on a simple \(\Cst\)\nb-algebra~\(A\).
  The following are equivalent:
  \begin{enumerate}
  \item \(\alpha\) is outer;
  \item \(A\subseteq A\rtimes^\red_{\alpha,\sigma} G\) is \(\Cst\)\nb-irreducible;
  \item \(A\subseteq A\rtimes^\red_{\alpha,\sigma} G\) is irreducible;
  \item the map
    \(G\supseteq H \mapsto A\rtimes^\red_{\alpha,\sigma} H\subseteq
    A\rtimes^\red_{\alpha,\sigma} G\) is a bijection between subgroups
    of~\(G\) and intermediate \(\Cst\)\nb-subalgebras of
    \(A\rtimes^\red_{\alpha,\sigma} G\).
  \end{enumerate}
  If the above equivalent conditions hold, then the inclusion
  \(A\subseteq A\rtimes^\red_{\alpha,\sigma} G\) determines the group \(G\) up to isomorphism
 and the twisted action \((\alpha,\sigma)\) up to exterior equivalence.
\end{theorem}

\begin{proof}
  Combine the discussion in Example~\ref{ex:twisted_actions} with
  Theorems \ref{thm:regular_irreducible_for_simple}
  and~\ref{thm:Galois_correspondence_explained}.
\end{proof}

The Fell bundles associated to twisted group actions may be
characterised as \emph{regular Fell bundles} over groups,
see \cites{Exel:TwistedPartialActions, BussExel:Regular.Fell.Bundle}.
In terms of inclusions, we may characterise them using a stronger
notion of regularity, which is standard in the setting of
\(\Wst\)\nb-algebras and is sometimes used for unital
\(\Cst\)\nb-inclusions, see for
instance~\cite{Bakshi-Gupta:Regular_inclusions}*{2.19.1}.
We generalise it to the nonunital case as follows:

\begin{definition}
  We call a nondegenerate \(\Cst\)\nb-inclusion \(A\subseteq B\)
  \emph{unitarily regular} if the \emph{unitary normalisers}
  \[
    \mathcal{U}N_A(B)\defeq\setgiven{u\in \UMult(B)}{u A u^*=A}
  \]
  multiplied by~\(A\) generate~\(B\), that is, if \(B=\clsp\setgiven{a
    u}{a\in A,\ u\in \mathcal{U}N_A(B)}\).
\end{definition}

\begin{remark}
  The inclusion \(A\subseteq A\rtimes^\red_{\alpha,\sigma} G\)
  for a twisted action \((\alpha,\sigma)\) of a discrete group~\(G\)
  is always unitarily regular.
  A unitarily regular \(\Cst\)\nb-inclusion is regular, but the
  converse may fail, see Examples \ref{ex:endomorphism_crossed_product},
  \ref{ex:endomorphism_crossed_product2}
  and~\ref{ex:transfer_crossed_product} below.
\end{remark}

\begin{proposition}
  \label{cor:unitary_regular_irreducible}
  For a \(\Cst\)\nb-inclusion \(A\subseteq B\), the following are
  equivalent:
  \begin{enumerate}
  \item \label{enu:unitary_regular_irreducible1}%
    \(A\subseteq B\) is \(\Cst\)\nb-irreducible and unitarily regular;
  \item \label{enu:unitary_regular_irreducible1.5}%
    \(A\subseteq B\) is an  irreducible, unitarily regular inclusion
    of simple \(\Cst\)\nb-algebras;
  \item \label{enu:unitary_regular_irreducible2}%
    \(B\cong A\rtimes^\red_{\alpha,\sigma} G\)  for an outer twisted
    action \((\alpha,\sigma)\) of a discrete group~\(G\) on~\(A\), and
    the isomorphism restricts to the identity on~\(A\).
  \end{enumerate}
  If the above equivalent conditions hold, then \(A\subseteq B\)
  admits a unique pseudo-expectation, the group~\(G\)
  in~\ref{enu:unitary_regular_irreducible2} is unique up to
  isomorphism, and if the group is fixed, then the twisted action
  \((\alpha,\sigma)\) is determined up to exterior equivalence.
  Moreover, \(G\) is finite if and only if \(A\subseteq B\) admits a
  conditional expectation with finite index.
\end{proposition}

\begin{proof}
  Theorem~\ref{cor:Cameron-Smith} shows
  that~\ref{enu:unitary_regular_irreducible2}
  implies~\ref{enu:unitary_regular_irreducible1} and that \(G\) and
  \((\alpha,\sigma)\) are unique as asserted.
  Theorem~\ref{thm:main_theorem} implies that
  \ref{enu:unitary_regular_irreducible1}
  and~\ref{enu:unitary_regular_irreducible1.5} are equivalent and that
  the pseudo-expectation is unique.
  Let us assume~\ref{enu:unitary_regular_irreducible1}.
  By  Theorem~\ref{thm:regular_irreducible_for_simple}, non-zero slices \(\Slice_A(B)\setminus \{0\}\) form both a discrete group \(G\) and an outer Fell bundle \(\B=(B_g)_{g\in G}\)
	such that \(B=\Cst_\red(\B)\). For any \(u\in  \mathcal{U}N_{A}(B)\)  the space \(Au\) is a non-zero slice, and their closed linear space is \(B\) because
	our inclusion is  unitarily regular. It follows that  each non-zero slice is of the
  form~\(A u\) for some \(u\in \mathcal{U}N_A(B)\) (one may use here the canonical projections \(E_g\colon \Cst_\red(\B)\onto B_g\subseteq \Cst_\red(\B)\), \(g\in G\)).
  In other words,  the Fell bundle~\(\B\) is regular in the sense of~\cite{BussExel:Regular.Fell.Bundle}. Hence it comes from a
  twisted action of~\(G\), see
  \cite{BussExel:Regular.Fell.Bundle}*{Corollary~4.17}.
  This action is outer because the Fell bundle is outer.
  Thus
  \ref{enu:unitary_regular_irreducible1}--\ref{enu:unitary_regular_irreducible2}
  are equivalent.

  It is well known that the conditional expectation
  \(A\rtimes^\red_{\alpha,\sigma} G\to A\) is of finite index
  when~\(G\) is finite, see
  \cite{Bakshi-Gupta:Regular_inclusions}*{Proposition~4.3}, which
  gives the last part of the assertion.
\end{proof}

\begin{remark}
  The above proposition recovers
  \cite{Bakshi-Gupta:Regular_inclusions}*{Corollary~3.16}, which is
  the application of the main result
  of~\cite{Bakshi-Gupta:Regular_inclusions} characterising irreducible
  unitarily regular inclusions of simple \(\Cst\)\nb-algebras with
  finite index expectations.
\end{remark}

\begin{proposition}
  \label{prop:sigma_stable}
  If \(A\subseteq B\) is a regular \(\Cst\)\nb-irreducible inclusion
  where~\(A\) is stable and \(\sigma\)-unital, then \(B\) is stable and
  \[
    B\cong A\rtimes^\red_{\alpha,\sigma}  G
  \]
  is a crossed product of a twisted action \((\alpha,\sigma)\)
  of a discrete group~\(G\) on~\(A\).
  If, in adition, the group~\(G\) is countable, then we may choose the twist to be
  trivial and so \(B\cong A\rtimes^\red_\alpha  G\) is an untwisted
  stable crossed product.
\end{proposition}

\begin{proof}
  Theorem~\ref{thm:main_theorem} implies \(B\cong \Cst_\red(\B)\)  for
  an outer Fell bundle \(\B=(B_g)_{g\in G}\) over a discrete
  group~\(G\).
  By our assumptions on~\(A\)  and
  \cite{Brown-Green-Rieffel:Stable}*{Theorem~3.4}, every~\(B_g\) is
  isomorphic to a Hilbert bimodule of an automorphism \(\alpha_g\colon
  A\to A\).
  Thus the bundle \((B_g)_{g\in G}\) is regular and there is a twisted
  action \(((\alpha_g)_{g\in G}, \sigma)\) on~\(A\) such that \(B\cong
  A\rtimes^\red_{\alpha,\sigma}  G\), see
  \cite{BussExel:Regular.Fell.Bundle}*{Corollary~4.17}.
  The algebra~\(B\) is stable by
  \cite{Hjelmborg-Rordam:Stability}*{Proposition~4.4}.
  If~\(G\) is countable, then the Packer--Raeburn stabilization trick
  (see \cite{Packer-Raeburn:Stabilisation}*{Theorem~3.4}) allows to
  choose~\(\sigma\) to be trivial.
\end{proof}

\begin{corollary}
  \label{cor:Gabe_Szabo_actions}
  Let \(A\subseteq B\) be a regular irreducible inclusion of a stable
simple  \(\Cst\)\nb-algebra~\(A\) into a nuclear separable
  \(\Cst\)\nb-algebra~\(B\).
  Then \(B\cong A\rtimes^\red_\alpha  G\) for an action~\(\alpha\) of
  a countable group~\(G\) on~\(A\), where the isomorphism restricts to
  the identity on~\(A\).
  The group~\(G\) is unique up to isomorphism, and the
  action~\(\alpha\) is outer, amenable and unique up to cocycle
  conjugacy.
\end{corollary}

\begin{proof}
  Combine Theorem~\ref{thm:regular_irreducible_for_simple} and
  Proposition~\ref{prop:sigma_stable}.
  The action~\(\alpha\) is amenable by
  \cite{Abadie-Buss-Ferraro:Amenability}*{Corollary~6.17 and
    Proposition~7.2}.
\end{proof}

Combining Corollary~\ref{cor:Gabe_Szabo_actions} and
\cite{Gabe-Szabo:Dynamical_Kirchberg}*{Theorem~F} gives

\begin{corollary}
  \label{cor:Gabe_Szabo_inclusions}
  Two regular irreducible inclusions \(A_i\subseteq B_i\) for
  \(i=1,2\) of stable purely infinite simple \(\Cst\)\nb-algebras into
  nuclear separable \(\Cst\)\nb-algebras are isomorphic
  \textup{(}that is, there is an isomorphism \(B_1\cong B_2\) that
  restricts to \(A_1\cong A_2\)\textup{)} if and only if these
  inclusions come from crossed products by \(KK^G\)-equivalent actions
  of the same discrete group~\(G\).
\end{corollary}

Now we consider the inclusion \(A^\beta\subseteq A\) given by the
fixed-point algebra of an action~\(\beta\) of a compact group~\(K\).
We may assume the action to be faithful because otherwise we may pass
to the action of the quotient~\(K/N\) of~\(K\) by the kernel~\(N\) of
the action, and this does not affects the fixed-point algebra.
Of course, any outer action is faithful. Recent results
of  Mukohara \cite{Mukohara:Inclusions} and Izumi \cite{Izumi:Minimal_compact} link \(\Cst\)\nb-irreducibility to quasi-product
actions as follows.

\begin{definition}[\cites{Bratteli-Elliott-Evans-Kishimoto:Quasi-product,
    Bratteli-Elliott-Kishimoto:Quasi-product}]
  \label{def:quasi_product_action}
  An action~\(\beta\) of a compact group~\(K\) on~\(A\) is a
  \emph{quasi-product action} if for any sequence~\((\xi_n)\)
  of finite-dimensional unitary representations of~\(K\), there are a
  \(\beta\)\nb-invariant sub-\(\Cst\)\nb-algebra~\(A_0\) of~\(A\) and
  a closed \(\alpha^{**}\)-invariant projection~\(q\) in the
  bidual~\(A^{**}\) with the following properties: \(q\in
  A_0'\), \(q A q = A_0 q\), \(q\in J^{**}\subseteq A^{**}\) for any
  nonzero ideal~\(J\) in~\(A\), and the \(\Cst\)\nb-dynamical system
  \((A_0q, K, \beta^{**}|_{A_0q})\) is isomorphic to the product
  system
  \(
    (\bigotimes_{n=1}^\infty M_{\dim \xi_n}, K,
    \bigotimes_{n=1}^\infty \Ad_{\xi_n}).
  \)
\end{definition}

\begin{theorem}[Izumi--Mukohara]
  Let~\(\beta\) be a faithful action of a second countable compact
  group~\(K\) on a separable \(\Cst\)\nb-algebra~\(A\) such
  that~\(A^\beta\) is simple.
  Then the following conditions are equivalent:
  \begin{enumerate}
  \item \(A^\beta\subseteq A\) is \(\Cst\)\nb-irreducible;
  \item \(A^\beta\subseteq A\) is irreducible;
  \item \(\beta\) is a quasi-product action.
  \end{enumerate}
\end{theorem}
\begin{proof}
Combine \cite{Mukohara:Inclusions}*{Remark 2.19, Example 4.2  and Theorem
  4.10} and   \cite{Izumi:Minimal_compact}*{Theorem 1.1}.
\end{proof}
In addition, Mukohara~\cite{Mukohara:Inclusions} obtained a Galois
correspondence for inclusions satisfying the conditions of the above
theorem.
For finite groups acting on a \(\sigma\)\nb-unital
\(\Cst\)\nb-algebra, such a correspondence was also proved by
Izumi~\cite{Izumi:Inclusions_simple}.
As another application of Theorems \ref{thm:main_theorem}
and~\ref{Thm:Galois_correspondence}, we reprove these results in the
special case where~\(K\) is abelian, without separability assumptions.
Our proof is qualitatively different, and we may say more about  the
properties of the resulting inclusions and intermediate
\(\Cst\)\nb-subalgebras.

\begin{lemma}
  \label{lem:abelian_actions}
  Let~\(\beta\) be a faithful action of a compact abelian group~\(K\)
  on a \(\Cst\)\nb-algebra~\(A\).
  Consider the crossed product \(A\rtimes_\beta K\) and the
  fixed-point subalgebra \(A^\beta\subseteq A\).
  \begin{enumerate}
  \item \label{enu:abelian_actions2}%
    Let~\(A^\beta\) be simple.  Then so is \(A\rtimes_\beta K\),
    and~\(\widehat{\beta}\) is outer if and only if the Fell
    bundle~\((A_g)_{g\in \widehat{K}}\) formed by the spectral
    subspaces of~\(\beta\) is outer.
  \item \label{enu:abelian_actions1}%
    If \(Z\Mult(A) = Z\Mult(A\rtimes_\beta K)=\C 1\) and~\(\beta\) is
    outer, then the dual action~\(\widehat{\beta}\) of~\(\widehat{K}\)
    on~\(A\rtimes_\beta K\) is outer.
  \end{enumerate}
\end{lemma}

\begin{proof}
  Recall that the Pontryagin dual~\(\widehat{K}\) of~\(K\) is a discrete
  abelian group.
  The spectral subspaces \(A_g\defeq \setgiven{a\in A}{\beta_k(a)=
    \langle g, k\rangle\cdot  a \text{ for all }k\in K}\) for \(g\in
  \widehat{K}\) with the operations inherited from~\(A\) form a Fell
  bundle \(\A\defeq (A_g)_{g\in \widehat{K}}\) such that
  \(\Cst(\A)=\Cst_\red(\A)\cong A\).

  We first prove~\ref{enu:abelian_actions2}.
  Since~\(\beta\) is faithful, the subspaces~\(A_g\) are nonzero for
  all \(g\in \widehat{K}\), see for instance
  \cite{Bratteli-Elliott-Evans-Kishimoto:Quasi-product}*{Lemma~2.2}.
  Since~\(A^\beta\) is simple, it follows that the Fell bundle~\(\A\)
  is saturated.
  Recall that the Morita globalisation of~\(\A\) from
  \cite{Kwasniewski-Meyer:Aperiodicity}*{Proposition~7.1} is based on
  Imai--Takai duality and is given by the crossed product \(C\defeq
  A\rtimes_{\delta_{\widehat{K}}} \widehat{K}\) by the dual
  coaction~\(\delta_{\widehat{K}}\) on   \(A=\Cst(\A)\),
  where~\(C\) is equipped with the dual action
  of~\(K\).
  Since~\(\widehat{K}\) is abelian, the duality specialises to the
  usual Takai duality in
  \cite{Pedersen:Cstar_automorphisms_vol2}*{Theorem 7.9.3}: the
  co-action crossed product \(A\rtimes_{\delta_{\widehat{K}}}
  \widehat{K}\) is \(K\)\nb-equivariantly isomorphic to the usual
  crossed product \(A\rtimes_\beta K\)  equipped with its canonical
  dual action~\(\widehat{\beta}\colon \widehat{K}\to
  \Aut(A\rtimes_\beta K)\) given by \(\widehat{\beta}_g(a)(k) =
  \langle g, k\rangle a(k)\) for \(a\in \Contc(K,A)\), see
  \cite{Pedersen:Cstar_automorphisms_vol2}*{Proposition~7.8.3}.
  Therefore, \cite{Kwasniewski-Meyer:Aperiodicity}*{Propositions
    7.1(2) and~6.8(3)} implies that~\(\beta\) is outer if and only
  if~\(\A\) is outer.
  Since~\(A^\beta\) is simple and \(A_g\neq0\) for all
  \(g\in\widehat{K}\),
  \cite{Kwasniewski-Meyer:Aperiodicity}*{Proposition~7.1(8)} implies
  that \(A\rtimes_\beta K\) is simple.

  Next we prove~\ref{enu:abelian_actions1}.
  We assume that \(\widehat{\beta}_{g_0}=\Ad_u\) for some \(g_0\in
  \widehat{K}\) and \(u \in \UMult(A\rtimes_\beta K)\) and want to
  deduce \(g_0=1\).
  Since we assumed \(\Mult(A\rtimes_\beta K)\) to have trivial centre,
  \(u\) is determined by~\(\beta_g\) up to a scalar multiple.
  A computation using that~\(\widehat{K}\) is abelian shows that
  \(\widehat{\beta}_{g_0}=\Ad_{\widehat{\beta}_g(u)}\)
  for all \(g\in \widehat{K}\).
  So \(u^*\widehat{\beta}_g(u)\in \T\).
  In addition, the map \(\widehat{K}\ni g\mapsto
  u^*\widehat{\beta}_g(u)\in \T\) must be a character.
  So there is \(k_0\in K\) such that
  \[
    \widehat{\beta}_g(u) = \langle g, k_0\rangle u \qquad \text{ for all }g\in \widehat{K}.
  \]
  Let \((\lambda_k)_{k\in K}\subseteq \UMult(A\rtimes_\beta K)\) be
  the canonical unitaries implementing~\(\beta\) in the crossed
  product \(A\rtimes_\beta K\).
  The definition of the dual action says that
  \[
    \widehat{\beta}_g(\lambda_{k_0})
    = \langle g, k_0\rangle \lambda_{k_0} \qquad \text{ for all }g\in \widehat{K}.
  \]
  It follows that \(v\defeq \lambda_{k_0} u^*\) is fixed
  by~\(\beta_g\) for all \(g\in \widehat{K}\).
  That is, \(v\in \Mult(A\rtimes_\beta K)^{\widehat{\beta}}\).
  If \(k \in K\), then
  \[
    v^*\lambda_k v
    = u \lambda_k u^*
    = \widehat{\beta}_{g_0}(\lambda_k)
    = \langle g_0, k\rangle \lambda_k,
  \]
  and so \(\lambda_k v \lambda_k^{-1} = \langle g_0, k\rangle v\).
  Therefore, for each \(y\in  A\rtimes_\beta K\), the map \(K\ni k\to
  \lambda_k v \lambda_k^{-1}y\in A\rtimes_\beta K\) is continuous.
  Hence \(v\in \Mult(A)\) by
  \cite{Pedersen:Cstar_automorphisms_vol2}*{Proposition~7.8.9}.
  Since \(\Ad_v|_A=\beta_{k_0}\) and~\(\beta\) is outer, it follows
  that \(k_0=1\).
  Then \(u=v^*\in \Mult(A)\) and \(\widehat{\beta}_{g_0}=\Ad_u\) acts
  as the identity on~\(A\).
  It follows that~\(u\) is in the centre of~\(\Mult(A)\), which we
  assumed to be trivial.
  Thus~\(\widehat{\beta}_{g_0}\) is the identity on \(A\rtimes_\beta
  K\).
  This implies \(g_0=1\) as desired.
\end{proof}

\begin{definition}[\cite{Bratteli-Elliott-Kishimoto:Quasi-product}*{3.1}]
  We call a \(\Cst\)\nb-inclusion \(A\subseteq B\) \emph{strongly
    prime} if every intermediate \(\Cst\)\nb-subalgebra \(A\subseteq
  C\subseteq B\) is prime; equivalently, \(x A y \neq 0\) for any
  nonzero \(x,y\in B\).
\end{definition}

\begin{theorem}[compare \cite{Mukohara:Inclusions}*{Theorem~4.18},
  \cite{Izumi:Inclusions_simple}*{Corollary~6.6(2)}]
  \label{thm:Mukohara_compact_abelian}
  Let~\(\beta\) be a faithful action of a compact abelian group~\(K\)
  on a \(\Cst\)\nb-algebra~\(A\) and assume that its fixed-point
  algebra~\(A^\beta\)  is simple.
  The following conditions are equivalent:
  \begin{enumerate}
  \item \label{enu:Mukohara_compact_abelian8}%
    the action~\(\beta\) is outer and \(Z\Mult(A)=\C\);
  \item \label{enu:Mukohara_compact_abelian0}%
    \(A^\beta\subseteq A\) is strongly prime; 
  \item \label{enu:Mukohara_compact_abelian1}%
    \(A^\beta\subseteq A\) is irreducible;
  \item \label{enu:Mukohara_compact_abelian2}%
    \(A^\beta\subseteq A\) is \(\Cst\)\nb-irreducible;
  \item \label{enu:Mukohara_compact_abelian6}%
    the Fell bundle \((A_{\kappa})_{\kappa\in\widehat{K}}\) formed by
    the spectral subspaces of~\(\beta\) is outer;
  \item \label{enu:Mukohara_compact_abelian7}%
    the dual action \(\widehat{\beta}\colon \widehat{K}\to
    \Aut(A\rtimes_\beta K)\) is outer;
  \item \label{enu:Mukohara_compact_abelian9}%
    there is a bijection from the set of all closed subgroups of~\(K\)
    onto the set of all intermediate \(\Cst\)\nb-subalgebras between \(A^\beta\) and~\(A\), given by
    \[
      K\supseteq L \mapsto A^{\beta|_L} \subseteq A.
    \]
  \end{enumerate}
  If the above equivalent conditions hold, then the inclusion
  \(A^\beta\subseteq A\) determines the group~\(K\) and the
  action~\(\beta\) up to an isomorphism, and if any of the
  intermediate subalgebras is nuclear, then all of them are.
\end{theorem}

\begin{proof}
  As we noticed in the proof of Lemma~\ref{lem:abelian_actions},
  \(\A\defeq (A_g)_{g\in \widehat{K}}\) is a saturated Fell
  bundle such that the canonical isomorphism
  \(\Cst(\A)=\Cst_\red(\A)\cong A\) turns the
  inclusion \(A^\beta\subseteq A\)  into \(A_1\subseteq
  \Cst_\red(\A)\).
  Thus the equivalence between the conditions
  \ref{enu:Mukohara_compact_abelian1}--\ref{enu:Mukohara_compact_abelian6}
  follows from Theorem~\ref{thm:regular_irreducible_for_simple}.
  Recall that \(G\supseteq  H \mapsto L\defeq \setgiven{k \in
    K}{\langle k ,h\rangle=1 \text{ for all }h\in H}\subseteq K\) is a
  bijection between subgroups of~\(G\) and closed subgroups of~\(K\).
  Under this bijection, the subalgebra  \(\Cst_\red(\A|_H)\subseteq \Cst_\red(\A)\) is the fixed-point
  algebra  \(A^{\beta|_L}\) for the restriction of the action~\(\beta\) to~\(L\).
  Hence
  \ref{enu:Mukohara_compact_abelian1}--\ref{enu:Mukohara_compact_abelian6}
  are equivalent to~\ref{enu:Mukohara_compact_abelian9} by
  Theorem~\ref{thm:Galois_correspondence_explained}.
  By Lemma \ref{lem:abelian_actions}.\ref{enu:abelian_actions2},
  \(A\rtimes_\beta K\) is simple and
  \ref{enu:Mukohara_compact_abelian6}
  and~\ref{enu:Mukohara_compact_abelian7} are equivalent.
  If \(A\rtimes_\beta K\) is simple, then \(Z\Mult(A\rtimes_\beta
  K)=\C 1\) and so~\ref{enu:Mukohara_compact_abelian8}
  implies~\ref{enu:Mukohara_compact_abelian7} by Lemma
  \ref{lem:abelian_actions}.\ref{enu:abelian_actions1}.
  Conversely, the equivalent conditions
  \ref{enu:Mukohara_compact_abelian1}
  and~\ref{enu:Mukohara_compact_abelian2} readily
  imply~\ref{enu:Mukohara_compact_abelian8}.
  Indeed, if~\(A\) is simple, then \(Z\Mult(A)=\C\), and  if
  \(\beta_k=\Ad_u\) for some \(k\in K\) and \(u \in \UMult(A)\), then
  \(u\in (A^\beta)'\cap \UMult(A)=\C\), which implies that
  \(\beta_k=\Id_A=\beta_1\).
  Hence \(k=1\) because~\(\beta\) is faithful, and so~\(\beta\) is
  outer.
  This proves that~\ref{enu:Mukohara_compact_abelian8} is equivalent
  to
  \ref{enu:Mukohara_compact_abelian1}--\ref{enu:Mukohara_compact_abelian9}.

  As simple \(\Cst\)\nb-algebras are prime,
  \ref{enu:Mukohara_compact_abelian2}
  implies~\ref{enu:Mukohara_compact_abelian0}.
  Conversely, we will show that~\ref{enu:Mukohara_compact_abelian0}
  implies~\ref{enu:Mukohara_compact_abelian7}.
  Let \(g\in \widehat{K}\setminus\{1\}\) and consider the cyclic group
  \(H\defeq \langle g\rangle\) it generates and the corresponding
  closed subgroup \(L\subseteq K\).
  As~\(A^{\beta|_L}\) is an intermediate subalgebra for
  \(A^\beta\subseteq A\), it is prime by assumption.
  Hence the crossed product \(A\rtimes_\beta K
  \rtimes_{\widehat{\beta}} H\) is also prime by
  \cite{Bratteli-Elliott-Evans-Kishimoto:Quasi-product}*{Lemma~2.1}.
  Then the Connes spectrum of~\(\widehat{\beta}|_H\) is full,
  \(\Gamma(\widehat{\beta})=\widehat{H}\), by
  \cite{Olesen-Pedersen:Applications_Connes}*{Theorem~5.8} (or
  \cite{Pedersen:Cstar_automorphisms_vol2}*{Theorem 8.11.10}).
  Since \(A\rtimes_\beta K\) is simple, Lemma
  \ref{lem:abelian_actions}.\ref{enu:abelian_actions2} and
  \cite{Olesen-Pedersen:Applications_Connes}*{Theorem~6.5}
  (or \cite{Pedersen:Cstar_automorphisms_vol2}*{Theorem 8.11.12})
  imply that \(A\rtimes_\beta K \rtimes_{\widehat{\beta}} H\) is
  simple.
  If \(H\cong \Z\) is infinite or if \(H\cong \Z/n\) for a square-free
  number~\(n\), this implies that \(\widehat{\beta}|_H\) is outer, by
  \cite{Kwasniewski-Meyer:Aperiodicity}*{Theorem~9.15} or by
  Theorem~\ref{the:special_groups}.
  Hence~\(\widehat{\beta}_g\) is outer in these cases.
  If \(H\cong \Z/n\) and~\(n\) is not square-free, then we write
  \(n=p\cdot m\) with a prime number~\(p\) and apply the above
  reasoning to~\(g^m\).
  We conclude that~\(\widehat{\beta}_{g^m}\) is outer, which implies
  that so is~\(\widehat{\beta}_g\).
  This proves~\ref{enu:Mukohara_compact_abelian7}.

  Hence the conditions
  \ref{enu:Mukohara_compact_abelian8}--\ref{enu:Mukohara_compact_abelian9}
  are equivalent.
  If they hold, then the Fell bundle of spectral subspaces
  in~\ref{enu:Mukohara_compact_abelian6} is uniquely determined, up to
  isomorphism, by the last part of Theorem~\ref{thm:main_theorem}.
  This determines both~\(K\) and~\(\beta\), up to isomorphisms.
  If, in addition, \(\Cst_\red(\A|_H)\) is nuclear for some subgroup
  \(H\subseteq G\), then \(A_1=A^\beta\) is nuclear as well, because
  there is a conditional expectation \(\Cst_\red(\A|_H)\onto
  A^\beta\).
  If~\(A^\beta\) is nuclear, then the intermediate subalgebras
  \(\Cst_\red(\A|_H)\) are nuclear for all subgroups \(H\subseteq G\)
  because~\(H\) is amenable (see, for instance,
  \cite{Abadie-Buss-Ferraro:Amenability}*{Proposition~7.2}).
\end{proof}

\begin{remark}
  If~\(A\) is separable, then the equivalent conditions
  \ref{enu:Mukohara_compact_abelian8}--\ref{enu:Mukohara_compact_abelian9}
  in Theorem~\ref{thm:Mukohara_compact_abelian} hold if and only if \(\beta\) is a quasi-product action.
  The authors of~\cite{Bratteli-Elliott-Evans-Kishimoto:Quasi-product}
  provided more characterisations of quasi-product actions in this
  setting.
  They also characterised them under the additional assumption
  that~\(K\) is \(\T\) or~\(\Z/p\) for a prime \(p>0\).
  Our counterpart of this situation is given in the last part of Theorem~\ref{thmx:fixed_point_for special_compact}.
\end{remark}

Now we combine actions \(\alpha\) and~\(\beta\) of a discrete
group~\(G\) and a compact group~\(K\), respectively, and consider the
inclusion \(A^\beta\subseteq A\rtimes_\alpha^\red G\).
Echterhoff and Rørdam~\cite{Echterhoff-Rordam:Inclusions} showed a
Galois correspondence for such an inclusion, assuming that
\(\alpha\) and~\(\beta\) commute, are jointly outer, \(A\) is unital
and simple, and~\(K\) is finite abelian.
We will now generalise their result by allowing~\(K\) to be any
compact abelian group, \(\alpha\) to be twisted and~\(A\) to be
nonunital.

\begin{definition}
  We say that two twisted actions \((\alpha,\sigma)\) and
  \((\beta,\omega)\) of locally compact groups \(G\) and~\(K\) on a
  \(\Cst\)\nb-algebra \emph{commute} if \(\beta_k\circ
  \alpha_g=\alpha_g\circ \beta_k\) for all \((k,g)\in K\times G\), and
  then they are \emph{jointly outer} if the automorphisms
  \(\beta_k\circ \alpha_g\) are outer for all \(k,g\in K\times G\)
  different than \((1,1)\).
\end{definition}

The following result unifies Theorems \ref{cor:Cameron-Smith}
and~\ref{thm:Mukohara_compact_abelian}.

\begin{theorem}[compare \cite{Echterhoff-Rordam:Inclusions}*{Theorem
    4.6}]
  \label{thm:Echterhoff-Rordam}
  Suppose that a \(\Cst\)\nb-algebra~\(A\) is equipped with a twisted
  action \((\alpha,\sigma)\) of a discrete group~\(G\) and a faithful
  action~\(\beta\) of a compact abelian group~\(K\).
  Assume that these actions commute.
  Let \((A_{\widehat{k}})_{\widehat{k}\in \widehat{K}}\) be the
  spectral subspaces for~\(\beta\) and let~\((u_g)_{g\in G}\) be the
  unitaries in \(\UMult(A\rtimes_{\alpha,\sigma}^\red G)\)
  implementing~\(\alpha\).
  The following conditions are equivalent:
  \begin{enumerate}
  \item \label{enu:Echterhoff-Rordam0}%
    \(A^\beta\) is simple, \(Z\Mult(A)=\C\) and  \(\alpha\), \(\beta\)
    are jointly outer;
  \item \label{enu:Echterhoff-Rordam1}%
    \(A^\beta\subseteq A\rtimes_{\alpha,\sigma}^\red G\) is
    \(\Cst\)\nb-irreducible;
  \item \label{enu:Echterhoff-Rordam2}%
    \(A^\beta\) is simple and \(A^\beta\subseteq
    A\rtimes_{\alpha,\sigma}^\red G\) is irreducible;
  \item \label{enu:Echterhoff-Rordam3}%
    \(A^\beta\) is simple and \(A^\beta\subseteq
    A\rtimes^\red_{\alpha,\sigma} G\) is aperiodic;
  \item \label{enu:Echterhoff-Rordam31}%
    \(A^\beta\) is simple and \(A^\beta\subseteq
    A\rtimes^\red_{\alpha,\sigma} G\) has a unique pseudo-expectation;
  \item \label{enu:Echterhoff-Rordam32}%
    \(A^\beta\subseteq A\rtimes^\red_{\alpha,\sigma} G\) is a simple
    noncommutative Cartan subalgebra;
  \item \label{enu:Echterhoff-Rordam4}%
    \(A^\beta\) is simple and for every \((\widehat{k},g)\in
    \widehat{K}\times G\setminus \{(1,1)\}\) the Hilbert
    \(A^\beta\)-bimodule \(A_{\widehat{k}}u_g\) is  outer;
  \item \label{enu:Echterhoff-Rordam45}%
    \(A^\beta\) is simple and there is a bijection between the
    subgroups of \(\widehat{K}\times G\) and intermediate
    \(\Cst\)\nb-subalgebras of \(A^\beta\subseteq
    A\rtimes_{\alpha,\sigma}^\red G\) given by
    \[
      \widehat{K}\times G\supseteq C \longmapsto \clsp\setgiven{A_{\widehat{k}}u_g}{(k,g)\in C} \subseteq A\rtimes^\red_{\alpha,\sigma} G.
    \]
    In particular, if \(C=L^\bot\times H\) is a product of subgroups,
    then it corresponds to the intermediate \(\Cst\)\nb-algebra
    \(A^{\beta|_L}\rtimes^\red_{\alpha,\sigma} H\) with \(L\defeq
    \setgiven{k \in K}{\langle k ,l\rangle=1 \text{ for all }l\in
      L^\bot}\subseteq K\).
  \end{enumerate}
  If~\(K\) is finite, then the above conditions are further equivalent
  to
  \begin{enumerate}[resume]
  \item \label{enu:Echterhoff-Rordam5}%
    \(A\) is simple and \(\alpha\), \(\beta\) are jointly outer.
  \end{enumerate}
  If the equivalent conditions
  \ref{enu:Echterhoff-Rordam0}--\ref{enu:Echterhoff-Rordam45} hold,
  then the group \(\widehat{K}\times G\) and the Fell bundle
  \((A_ku_g)_{(k,g)\in \widehat{K}\times G}\) are uniquely determined,
  up to isomorphism, by the inclusion \(A^\beta\subseteq
  A\rtimes_{\alpha,\sigma}^\red G\).
\end{theorem}

\begin{proof}
  Let \(A_k\defeq \setgiven{a\in A}{ \beta_{\chi}(a)=  \chi(k) a \text{
      for all }\chi\in K}\) for \(k\in \widehat{K}\).
  We have shown in the proof of
  Theorem~\ref{thm:Mukohara_compact_abelian} that this is a Fell
  bundle~\(\A\) with \(A_1 = A^\beta\) and \(\Cst(\A) \cong A\).
  Let \((u_g)_{g\in G}\) be the unitaries in
  \(\UMult(A\rtimes_{\alpha,\sigma}^\red G)\) that
  implement~\(\alpha\).
  Since \(\alpha\) and \(\beta\) commute for every \((k,g)\in
  \widehat{K}\times G\) we get \(\alpha_g(A_k)=A_k\) and so \(u_g
  A_k=A_k u_g\).
  It follows that \((A_ku_g)_{(k,g)\in \widehat{K}\times G}\) forms a
  grading for \(A^\beta\subseteq A\rtimes_{\alpha,\sigma}^\red G\).
  This grading is topological because the composite of the canonical
  conditional expectations \(A\rtimes_{\alpha,\sigma}^\red G\onto A\)
  and \(A\onto A^\beta\) yields a conditional expectation \(E\colon
  A\rtimes_{\alpha,\sigma}^\red G\onto A^\beta\) such that \(E(A_k
  u_g)=0\) whenever \((k,g)\neq (1,1)\).
  As a composite of faithful maps, \(E\) is faithful, and so
  \(A\rtimes_{\alpha,\sigma}^\red G=\Cst_\red(\B)\) is the reduced
  section \(\Cst\)\nb-algebra of the group Fell bundle
  \(\B\defeq (A_ku_g)_{(k,g)\in \widehat{K}\times G}\).

  Now the
  conditions \ref{enu:Echterhoff-Rordam1}--\ref{enu:Echterhoff-Rordam4}
  are equivalent and they imply the last part of the assertion by
  Theorem~\ref{thm:main_theorem}.
  They are also equivalent to~\ref{enu:Echterhoff-Rordam45} by
  Theorem~\ref{thm:Galois_correspondence_explained}.
  They imply~\ref{enu:Echterhoff-Rordam0} by
  Theorem~\ref{cor:Cameron-Smith} and
  Theorem~\ref{thm:Mukohara_compact_abelian}.
  To close the cycle of implications is suffices to show
  that~\ref{enu:Echterhoff-Rordam0} implies that~\(\B\) is outer.
  Let us then assume~\ref{enu:Echterhoff-Rordam0}.
  Theorem~\ref{thm:Mukohara_compact_abelian} implies that
  \(A=\Cst_\red(\A)\) is simple and the bundle~\(\A\) is outer.
  Assume that for some \((k,g)\in \widehat{K}\times G\) there
  is a unitary \(V\colon A_k u_g\congto A^\beta\).
  Let \(U(a)\defeq \alpha_g(V(au_g))\) for \(a\in A_k\).
  This is a unitary \(U\colon A_k\congto A^\beta\) because if \(a,b\in
  A_k\), then
  \[
    U(a)^*U(b)
    = \alpha_g\bigl(V(au_g)^*V(bu_g)\bigr)
    = \alpha_g\bigl((au_g)^*bu_g\bigr)
    = a^*b.
  \]
  Since~\(\A\) is outer, this implies that \(k=1\), and so \(V\colon
  A^\beta u_g\congto A^\beta\).
  We claim that~\(V\) extends to a unitary \(\widetilde{V}\colon A
  u_g\congto A\) such that \(\widehat{V}(a_k a_0 u_g)=a_k V(a_0 u_g)\)
  for all \(k\in \widehat{K}\), \(a_k\in A_k\), and \(a_0\in A^\beta\).
  Indeed, any two elements \(a, b\in \bigoplus_{k\in \widehat{K}} A_k\)
  may be written as \(a=\sum_k a_k a_k^0\) and \(b=\sum_k b_k b_k^0\)
  with \(a_k, b_k\in A_k\), \(a_k^0,  b_k^0\in A^\beta\) for \(k\in \widehat{K}\).
  Then
  \begin{align*}
    \biggl( \sum_k a_k V(a_k^0 u_g) \biggr)
    \biggl( \sum_k b_k V(b_k^0 u_g) \biggr)^*
    &= \sum_{j,k} a_k  V(a_k^0 u_g)V(b_j^0 u_g)^* b_j
    \\
		&= \sum_{j,k} a_k  a_k^0 {b_j^0}^*  b_j
    = a b^*.
  \end{align*}
  This implies our claim.
  Then \(g=1\) because~\(\alpha\) is outer by assumption.
  Hence~\(\B\) is outer.
  This finishes the proof that the conditions
  \ref{enu:Echterhoff-Rordam0}--\ref{enu:Echterhoff-Rordam45} are
  equivalent.
  They clearly imply~\ref{enu:Echterhoff-Rordam5}.
  If~\(K\) is finite, then~\ref{enu:Echterhoff-Rordam5}
  implies~\ref{enu:Echterhoff-Rordam0} by
  \cite{Olesen-Pedersen-Stormer:Compact_abelian_auto_simple}*{Theorem~2}.
\end{proof}

\section{Inclusions coming from Cuntz--Pimsner algebras}
\label{sect:Cuntz-Pimnser}

The last part of Theorem~\ref{the:special_groups} is useful to check
whether the inclusion of the core subalgebra in a Cuntz--Pimsner
algebra is \(\Cst\)\nb-irreducible because there are very efficient simplicity
criteria for Cuntz--Pimsner algebras, see
\cite{Schweizer:Dilations_correspondences} or
\cite{Carlsen-Kwasniewski-Ortega:Free_correspondence}*{Subsection~9.2}.
The results for non-unital algebras are less clean,
see~\cite{Paulovicks-Tomforde:CK_uniqueness}, so that we often
restrict to unital \(\Cst\)\nb-inclusions in this section.
We combine this here with our Galois correspondence in
Theorem~\ref{Thm:Galois_correspondence} and provide some natural
examples of regular \(\Cst\)\nb-irreducible inclusions that are not
modeled by  crossed products by group actions.

Recall that a \(\Cst\)\nb-correspondence over~\(A\) is a right Hilbert
\(A\)\nb-module~\(\Hilm\) together with a left \(A\)\nb-module
structure given by a nondegenerate \Star{}homomorphism \(\phi\colon
A\to \Bound(\Hilm)\) into the adjointable operators on~\(\Hilm\).
Let \(J(\Hilm) \defeq \phi^{-1}(\Comp(\Hilm))\) be the ideal of
elements in~\(A\) that act by compact operators on the left
of~\(\Hilm\), and let \((\ker\phi)^\bot\) be the annihilator of the
kernel of~\(\phi\) in~\(A\).
The Cuntz--Pimsner algebra \(\Cst(\Hilm)\) is the universal
\(\Cst\)\nb-algebra generated by a copy of \(A\) and~\(\Hilm\) so that
the whole \(\Cst\)\nb-correspondence structure  of \((A,\Hilm)\) comes
from the \(\Cst\)\nb-algebra operations in \(\Cst(\Hilm)\) and, in
addition, \(A\cap \Hilm\Hilm^* =J(\Hilm)\cap(\ker\phi)^\bot\).
When writing products of sets, we mean the closed linear span of
the products of their elements.
We also define powers of sets similarly.
The inclusions \(A\Hilm\subseteq \Hilm\), \(\Hilm A\subseteq \Hilm\)
and \(\Hilm^*\Hilm\subseteq A\) imply that for each \(k\in \N\defeq\{0,1,2,\dots\}\), the
subspace
\[
  C_k\defeq \clsp\bigcup\setgiven{\Hilm^n\Hilm^{* m}}
  {n,m \in \N  \text{ and }k|(n-m) }.
\]
is a \(\Cst\)\nb-subalgebra \(C_k\subseteq \Cst(\Hilm)\).
Here \(k|m\) means that~\(m\) is a multiple of~\(k\) in~\(\Z\).
In particular, \(0|m\) only if \(m=0\).
In particular, \(C_0=\clsp\bigcup_{n\ge 0}\Hilm^n \Hilm^{*n}\) is the
\emph{core subalgebra} of \(\Cst(\Hilm)\), which is also sometimes
denoted by~\(\Hilm[F]_{\Hilm}\) or \(\Cst(\Hilm)^\T\), as it is the
fixed-point algebra for the canonical gauge action of the
circle~\(\T\).
For \(k\geq 1\),~\(C_k\) is the fixed-point
algebra \(\Cst(\Hilm)^{\Z/k\Z}\) for the restriction of the gauge
action to the subgroup \(\Z/k\Z\subseteq \T\).
Note that \(C_1= \Cst(\Hilm)\) and
\[
  \Hilm[F]_{\Hilm}
  \subseteq C_k
  \subseteq \Cst(\Hilm)
  \qquad \text{for  all }k\in \N.
\]
Therefore, we call~\(C_k\) for \(k\in \N \) the \emph{canonical
  intermediate \(\Cst\)\nb-algebras} for the core inclusion.
The Hasse diagram for these algebras partially ordered by inclusion is
isomorphic to the Hasse diagram of divisibility of positive natural
numbers together with the smallest element given by zero.
In particular, if \(n,m\geq 1\), then
\[
  \Cst(C_n\cup C_m) = C_{\operatorname{lcm}(n,m)}
  \qquad \text{and}\qquad
  C_n\cap C_m = C_{\operatorname{gcd}(n,m)}.
\]

\begin{theorem}
  \label{thm:irreducible_from_Cuntz_Pimsner}
  Let~\(\Hilm\) be a \(\Cst\)\nb-correspondence over a
  \(\Cst\)\nb-algebra~\(A\).
  The following are equivalent:
  \begin{enumerate}
  \item \label{enu:irreducible_from_Cuntz_Pimsner1}%
    the core \(\Cst\)\nb-inclusion \(\Hilm[F]_{\Hilm}\subseteq
    \Cst(\Hilm)\) is \(\Cst\)\nb-irreducible;
  \item \label{enu:irreducible_from_Cuntz_Pimsner2}%
    both \(\Hilm[F]_{\Hilm}\) and~\(\Cst(\Hilm)\) are simple;
  \item \label{cor:irreducible_from_Cuntz_Pimsner3}%
    the only intermediate \(\Cst\)\nb-algebras for
    \(\Hilm[F]_{\Hilm}\subseteq \Cst(\Hilm)\) are the canonical ones.
  \end{enumerate}
  If~\(A\) is unital and simple, then \(\Hilm[F]_{\Hilm}\subseteq
  \Cst(\Hilm)\) is \(\Cst\)\nb-irreducible if and only if
  \(J(\Hilm)=A\)  and there is no \(n>0\) for which \(\Hilm^n\cong A\)
  as a correspondence.
\end{theorem}

\begin{proof}
  Since~\(\Hilm[F]_{\Hilm}\) is the fixed-point algebra of the
  gauge-action of the circle group on \(\Cst(\Hilm)\), we may view
  \(\Cst(\Hilm)\)
  as \(\Cst(\B)\) for a Fell bundle \(\B=(B_n)_{n\in \Z}\) with
  \(B_0=\Hilm[F]_{\Hilm}\).
  In fact, \(B_k=\Hilm^k \Hilm[F]_{\Hilm}\) for \(k\ge 0\), and the
  \(\Cst\)\nb-algebra generated by \(B_0\cup B_k\)
  is \(C_k\defeq\clsp\setgiven{\Hilm^n\Hilm^{*m}}{k|(n-m)}\).
  Hence the first part of the assertion follows from Theorems
  \ref{the:special_groups} and~\ref{thm:Mukohara_compact_abelian}.

  Assume that~\(A\) is simple.
  Then either \(J(\Hilm)=0\) or \(J(\Hilm)=A\).
  Recall that
  \(\Hilm[F]_{\Hilm}=\overline{\bigcup_{n=0}^\infty F_n}\) is the
  inductive limit of the \(\Cst\)\nb-algebras \(F_n\defeq A + \Hilm
  \Hilm^* +\dotsb + \Hilm^n\Hilm^{*n}\) for \(n>0\) and \(F_0\defeq
  A\).
  Each \(\Cst\)\nb-subalgebra \(\Hilm^n\Hilm^{*n}\) is Morita
  equivalent to~\(A\) and hence simple.
  It follows from
  \cite{Katsura:Cstar_correspondences}*{Proposition~5.11} that \(F_n
  \cap \Hilm^{n+1}\Hilm^{*n+1}=\Hilm^nJ(\Hilm)\Hilm^{*n}\).
  If \(J(\Hilm)=A\), then \(F_n=\Hilm^{*n}\Hilm^n\) for all \(n>0\),
  and so~\(\Hilm[F]_{\Hilm}\) is simple as an inductive limit of
  simple algebras.
  If \(J(\Hilm)=0\), then \(F_n\cap \Hilm^{*n+1} \Hilm^{n+1} = \{0\}\)
  and it follows that \(J_n\defeq \Hilm \Hilm^* + \dotsb + \Hilm^n\Hilm^{*n}\)
  is a nontrivial ideal in~\(F_n\) and so \(J\defeq
  \overline{\bigcup_{n=1}^\infty J_n}\) is a nontrivial ideal
  in~\(\Hilm[F]_{\Hilm}\), see
  \cite{Katsura:Cstar_correspondences}*{Proposition~5.13}.
  Hence~\(\Hilm[F]_{\Hilm}\) is simple if and only if \(J(\Hilm)=A\).
  If, in addition, \(A\) is unital, then
  \cite{Schweizer:Dilations_correspondences}*{Theorem~3.9} implies
  that \(\Cst(\Hilm)\) is simple if and only if there is no \(n>0\)
  for which~\(\Hilm^n\) is an inner Hilbert bimodule, that is,
  \(\Hilm^n\cong A\) as a correspondence.
  This proves the last part of the assertion.
\end{proof}

\begin{remark}
  The first part of the above theorem can be phrased in a seemingly
  stronger way.
  Namely, in our notation, \(C_0\subseteq C_k\) for \(k\ge 1\)
  is \(\Cst\)\nb-irreducible if and only if \(C_0\) and~\(C_k\)
  are simple, and then the intermediate \(\Cst\)\nb-subalgebras are
  \((C_n)_{k\mid n}\).
  The reason is that~\(C_k\) is the Cuntz--Pimsner algebra of the
  Hilbert \(C_0\)\nb-bimodule~\(B_k\).
\end{remark}

We illustrate Theorem~\ref{thm:irreducible_from_Cuntz_Pimsner} by
specialising to crossed products by endomorphisms and by transfer
operators.
Both are special cases of crossed products by completely positive
maps, see~\cite{Kwasniewski:Exel_crossed}.

\begin{example}[Crossed products by endomorphisms]
  \label{ex:endomorphism_crossed_product}
  Let \(\alpha\colon A\to A\) be an endomorphism of a unital simple
  \(\Cst\)\nb-algebra~\(A\).
  The crossed product \(A\rtimes_\alpha\N\), often called
  \emph{Stacey's crossed product}, is the universal
  \(\Cst\)\nb-algebra generated by~\(A\) and an isometry \(u\in
  A\rtimes_\alpha\N\) such that \(\alpha(a)=u a u^*\) for all \(a\in
  A\).
  By \cite{Kwasniewski:Exel_crossed}*{Proposition~3.20},
  \(A\rtimes_\alpha\N\)  is a special case of a crossed product by a
  completely positive map.
  For each \(k\in \N\), define a \(\Cst\)\nb-subalgebra of
  \(A\rtimes_\alpha\N\) by
  \[
    C_k\defeq
    \clsp\setgiven{u^{*n}au^m}{a\in A, n,m\ge 0  \text{ and } k |(n-m)}.
  \]
  Then \(C_0=\overline{\bigcup_{n\ge 0}u^{*n} A u^n}\) is a
  simple \(\Cst\)\nb-algebra and \(C_1=A\rtimes_\alpha\N\) is the
  crossed product.
  We call~\(\alpha\) \emph{outer} if there is no isometry \(v\in A\)
  such that \(\alpha(\cdot)=v(\cdot)v^*\).
  It follows that
  \(C_0\subseteq A\rtimes_\alpha\N\) is \(\Cst\)\nb-irreducible
  if and only if~\(\alpha^n\) is outer for all \(n>0\).
  If this holds, then all intermediate \(\Cst\)\nb-algebras are of the
  form~\(C_k\) for some \(k\in \N\).
  Indeed, \(A\rtimes_\alpha\N\) is the Cuntz--Pimsner algebra of the
  \(\Cst\)\nb-correspondence \(M_\alpha\defeq \alpha(1)A\) with the
  operations \(a\cdot x=\alpha(a)x\), \(x\cdot a=xa\) and
  \(\braket{x}{y}_A\defeq x^*y\) for all \(x,y\in M_\alpha\) and
  \(a,b\in A\).
  The left action of \(A\) on~\(M_\alpha\) is by compact operators,
  and there is a canonical isomorphism \(\Cst(M_\alpha)\cong
  A\rtimes_\alpha\N\) which is the identity on~\(A\) and extends the
  map \(M_\alpha\ni x\mapsto u^*x\in A\rtimes_\alpha\N\).
  Under this isomorphism, the core subalgebra and the canonical
  intermediate \(\Cst\)\nb-algebras for~\(M_\alpha\) coincide with the
  subalgebras~\(C_k\) for \(k\in \N\) defined above.
  Hence our claim follows from
  Theorem~\ref{thm:irreducible_from_Cuntz_Pimsner} and
  \cite{Schweizer:Dilations_correspondences}*{Theorem~4.1}.
\end{example}

\begin{example}[Crossed products by endomorphisms with hereditary
  range]
  \label{ex:endomorphism_crossed_product2}
  We retain the assumptions and notation from the previous example.
  In addition, we assume that the range of~\(\alpha\) is a hereditary
  \(\Cst\)\nb-subalgebra of~\(A\), namely,
  \(\alpha(A)=\alpha(1)A\alpha(1)\).
  Then \(A\rtimes_\alpha\N\) is \emph{Paschke's crossed product}
  from~\cite{Paschke:Crossed_endomorphism}.
  In addition to \(u A u^*\subseteq A\), we also have \(u^* A
  u\subseteq A\).
  Indeed,  \(u^* A u = u^* u u^* A u u^* u = u^* \alpha(1) A \alpha(1)
  u = u^* \alpha(A) u = u^* u A u^* u = A\).
  In this case, \(A=C_0\) is the core algebra, and for each \(k\ge 1\)
  there is a natural identification \(A\rtimes_{\alpha^k}\N=C_k\).
  Therefore, \(A\subseteq A\rtimes_\alpha\N\) is
  \(\Cst\)\nb-irreducible if and only if~\(\alpha^n\) is outer for all
  \(n>0\), and if this holds then the nontrivial intermediate
  \(\Cst\)\nb-algebras are of the form
  \(A\rtimes_{\alpha^k}\N\subseteq A\rtimes_\alpha\N\) for \(k\in
  \N\).

  For instance, the Cuntz algebra~\(\mathcal{O}_n\) is the crossed
  product \(\mathcal{F}_{n^\infty}\rtimes_\alpha\N\) of its core
  UHF-algebra of type~\(n^\infty\) by an endomorphism implemented by
  the canonical isometry \(u\defeq \frac{1}{\sqrt{n}}\sum_{i=1}^n
  u_i\) or by any of the isometries \(u_1,\dotsc,u_n\)
  generating~\(\mathcal{O}_n\).
  As a consequence, any intermediate \(\Cst\)\nb-algebra for the
  inclusion \(\mathcal{F}_{n^\infty}\subseteq \mathcal{O}_n\) is of
  the form \(\mathcal{F}_{n^\infty}\rtimes_{\alpha^k}\N\) for some
  \(k\ge 1\), which was stated without proof in
  \cite{Rordam:Irreducible_inclusions}*{Example~5.11}.
\end{example}

\begin{example}[Crossed products by transfer operators of finite type]
  \label{ex:transfer_crossed_product}
  Let \(L\colon A\to A\) be a unital, faithful completely positive
  map on a simple, unital \(\Cst\)\nb-algebra~\(A\).
  Assume also that there is a unital endomorphism \(\alpha\colon A\to
  A\) such that \(L(\alpha(a)b)=a L(b)\) for \(a,b\in A\).
  Then~\(\alpha\) is uniquely determined by~\(L\), and \(E\defeq
  \alpha\circ L\) is a conditional expectation onto~\(\alpha(A)\).
  We also assume that~\(E\) is of finite type, that is, there are
  \(u_1,\dotsc,u_n\in A\) such that \(a=\sum_{i=1}^n u_i E(u_i^*a)\)
  for all \(a\in A\).
  We call~\(L\) a \emph{transfer operator of finite type}.
  Then the crossed product \(A\rtimes_L \N\) is the universal
  \(\Cst\)\nb-algebra generated by a copy of~\(A\) and an
  isometry~\(\su\) subject to the relations:
  \[
    L(a) = \su^*a\su
    \qquad \text{ and }\qquad
    \sum_{i=1}^n u_i\su\su^*u_i^* = 1
  \]
  for all \(a\in A\).
  In fact, \(A\rtimes_L\N\)  coincides with \emph{Exel's crossed
    product} and the above relations are given in
  \cite{Exel-Vershik:Irreversible}*{Corollary~7.2} (the first relation
  there is superfluous by
  \cite{Kwasniewski:Exel_crossed}*{Proposition~4.3}).
  By \cite{Kwasniewski:Exel_crossed}*{Proposition~4.17},
  \(A\rtimes_L\N\)  is a special case of a crossed product by a
  completely positive map.
  If \(k\in \N\), then we define a \(\Cst\)\nb-subalgebra of
  \(A\rtimes_L\N\) by
  \[
    C_k\defeq
    \clsp\setgiven{a\su^n\su^{*m}b}
    {a,b\in A, n,m\ge 0  \text{ and } k \mid (n-m)}.
  \]
  In particular, \(C_0 = \overline{\bigcup_{n\ge 0} A\su^n\su^{*n}A}\)
  is a simple \(\Cst\)\nb-algebra and \(C_1=A\rtimes_L\N\) is the
  crossed product.
  The Cuntz--Pimsner model for \(A\rtimes_L\N\) is based on the
  \(\Cst\)\nb-correspondence \(M_L\defeq A\) with operations \(a\cdot
  x=a x\), \(x\cdot a=x\alpha(a)\) and \(\braket{x}{y}_A\defeq
  L(x^*y)\) for all \(x,y\in M_L\) and \(a,b\in A\).
  As~\(L\) is of finite type, \(A\) acts by compacts on the left
  of~\(A\).
  There is a canonical isomorphism \(\Cst(M_L)\cong A\rtimes_L\N\)
  that extends the identity map on~\(A\) and the map \(M_L\ni x\mapsto
  x\su\in A\rtimes_L\N\).
  Under this isomorphism, the core subalgebra and the canonical
  intermediate \(\Cst\)\nb-algebras for \(\Cst(M_L)\) coincide with
   \(C_k\) for \(k\in \N\cup \{\infty\}\).
   Moreover, if~\(L\) is not an automorphism, then no power of the
   \(\Cst\)\nb-correspondence~\(M_L\) is inner, see
   \cite{Schweizer:Dilations_correspondences}*{Theorem~4.6}.
   Thus, if~\(L\) is not an automorphism, then the inclusion
   \(C_0\subseteq A\rtimes_L\N\) is \(\Cst\)\nb-irreducible and
   any intermediate \(\Cst\)\nb-algebra is of the form~\(C_k\) for
   some \(k\in \N\).
\end{example}

\(\Cst\)\nb-correspondences over commutative \(\Cst\)\nb-algebras with
discrete spectrum are equivalent to directed graphs, and the
corresponding Cuntz--Pimsner algebras are graph \(\Cst\)\nb-algebras.
In this case, we may efficiently characterise
\(\Cst\)\nb-irreducibility of the core inclusion in full generality.
We use the well-known criterion for a graph \(\Cst\)\nb-algebra to be
simple and the somewhat less-known criterion for the core subalgebra
to be simple from~\cite{Pask-Rho:simple_graph_core}.

Let \(E = (E^0,E^1,\rg,\sr)\) be a \emph{directed graph} where~\(E^0\)
is the set of vertices, \(E^1\) is the set of edges, and \(\rg,\sr
\colon E^1 \rightrightarrows E^0\) are the range and source maps.
The \emph{graph \(\Cst\)\nb-algebra} \(\Cst(E)\) is the universal
\(\Cst\)\nb-algebra generated by a family \(\setgiven{\pu_v}{v\in
  E^0}\) of mutually orthogonal projections and a family
\(\setgiven{\su_v}{v\in E^1}\) of partial isometries with mutually
orthogonal ranges such that
\[
  \su_e^*\su_e=\pu_{\sr(e)},\qquad
  \su_e\su_e^*\le \pu_{\rg(e)}, \qquad
  \sum_{f\in \sr^{-1}(v)}\su_f\su_f^*=\pu_v
\]
for every \(e\in E^1\) and every \(v\in E^0\) with
\(0<\abs{\sr^{-1}(v)}<\infty\).
A path of length \(n\ge 1\) is a sequence of edges
\(e_1\ldots e_n\) with \(e_i \in E^1\) and \( \sr(e_i) =
\rg(e_{i+1})\) for \(i=1,\dotsc,n-1\).
Let~\(\abs{\mu}\) for a path~\(\mu\) be its length; here a vertex is a
path of length~\(0\).
Let~\(E^*\) be the set of all finite paths.
For a path \(\mu=e_1\ldots e_n\), let \(\su_{\mu}=\su_{e_1}\dotsm
\su_{e_n}\).
For each \(k\in \N\), the formula
\[
  C_k\defeq
  \clsp\setgiven{\su_{\mu}\su_{\eta}^*}
  {\mu,\eta\in E^*,\ k \mid \abs{\mu}-\abs{\eta}}
\]
defines a \(\Cst\)\nb-subalgebra of
\(\Cst(E)=\setgiven{\su_{\mu}\su_{\eta}^*}{\mu,\eta\in E^*}=C_1\) and
\[
  C_0=\clsp\setgiven{\su_{\mu}\su_{\eta}^*}
  {\mu,\eta\in E^*,\ \abs{\mu}=\abs{\eta}}
  = \Cst(E)^\T
\]
is the core subalgebra.

\begin{definition}
  The graph~\(E\) is \emph{regular} if
  \(0<\abs{\sr^{-1}(v)}<\infty\) for every \(v\in E^0\).
  A \emph{cycle} is a path \(e_1\ldots e_n\) with \(\sr(e_n)=\rg(e_1)\).
  Cycles of length~\(1\) are called \emph{loops}.
  A vertex \(v\in E^0\) \emph{connects} to a vertex \(w\in E^0\) if
  there is a path \(e_1\cdots e_n\) such that \(w=\rg(e_1)\) and
  \(v=\sr(e_n)\).
  The graph~\(E\) is \emph{cofinal} if for every infinite path
  \(e_1e_2\ldots\) and every \(v\in E^0\) there is \(n\in \N\) such
  that \(\rg(e_n)\) connects to~\(v\).
  The graph~\(E\) is \emph{strongly connected} if every \(v\in E^0\)
  connects to every \(w\in E^0\).
  A strongly connected graph~\(E\) \emph{has period~\(1\)} if the
  lengths of all cycles which begin at a fixed vertex \(v\in E^0\) are
  coprime.
\end{definition}

The definition of period~\(1\) does not depend on the choice of~\(v\)
by \cite{Pask-Rho:simple_graph_core}*{Lemma~4.1}.
A graph is strongly connected of period~\(1\) if and only if the
incidence matrix~\(A_E\) is aperiodic, that is, some power of it is
strictly positive.

\begin{theorem}
  \label{thm:irreducible_graph_algebras}
  Let~\(E\) be a directed graph which is not a single vertex or a single loop.
  The following are equivalent:
  \begin{enumerate}
  \item \label{enum:irreducible_graph_algebras1}%
    the inclusion \(\Cst(E)^\T\subseteq \Cst(E)\) is
    \(\Cst\)\nb-irreducible;
  \item \label{enum:irreducible_graph_algebras3}%
    \(E\) is regular, cofinal and contains a finite strongly
    connected subgraph with period~\(1\);
  \item \label{enum:irreducible_graph_algebras2}%
    \(\{C_k\}_{k\in \N}\) are the only intermediate subalgebras for the
    inclusion \(\Cst(E)^\T\subseteq \Cst(E)\).
  \end{enumerate}
\end{theorem}

\begin{proof}
  There is a mismatch between the proof and the statement of
  \cite{Pask-Rho:simple_graph_core}*{Theorem~6.1}.
  The proof shows that if~\(\Cst(E)^\T\) is simple, then~\(E\) is
  regular and cofinal.
  Thus we may assume that.
  Then the last part of the proof of
  \cite{Pask-Rho:simple_graph_core}*{Theorem~6.1} shows that
  \(\Cst(E)^\T\) is simple if and only if~\(E\) contains a finite
  strongly connected subgraph with period~\(1\).
  Hence assuming that~\(\Cst(E)^\T\)  is simple, it follows
  that~\(\Cst(E)\) is simple if and only if~\(E\) is not a single
  loop, see \cite{Pask-Rho:simple_graph_core}*{Theorem~3.1}.
  Thus the conditions in~\ref{enum:irreducible_graph_algebras3}
  characterise when both \(\Cst(E)^\T\) and~\(\Cst(E)\) are simple.
  Hence
  \ref{enum:irreducible_graph_algebras1}--\ref{enum:irreducible_graph_algebras2}
  are equivalent by Theorem~\ref{thm:irreducible_from_Cuntz_Pimsner}.
\end{proof}

\begin{remark}
  A graph satisfying the conditions of
  Theorem~\ref{thm:irreducible_graph_algebras} looks like a comet.
  The ``nucleus'' is based on the set~\(F^0\) of all vertices that lie
  on cycles, and we require that it forms a finite, strongly connected
  subgraph \(F=(F^0, \rg^{-1}(F^0)\cap \sr^{-1}(F^0),\rg,\sr)\) with
  period~\(1\).
  Vertices in \(E^0\setminus F^0\)  form a ``tail'' as they lie on paths directed towards the ``nucleus''.
  Namely, every \(v\in E^0\setminus F^0\)
  connects to~\(F^0\) and all sufficiently long paths that start
  in~\(v\) land in~\(F^0\) and then stay there forever.
  Since we rule out sources, the ``tails'' must be either empty or infinite:
  every \(v\in E^0\setminus F^0\) is the range of an infinite path
  with all base vertices in \(E^0\setminus F^0\).
\end{remark}

\end{document}